\documentclass[reqno]{amsart}

\usepackage{amsmath,amssymb,amsthm,amsfonts}
\usepackage{color}
\usepackage{mathrsfs}
\usepackage{hyperref}
\usepackage{a4wide}
\numberwithin{equation}{section}
\usepackage{tikz}

\newtheorem{theorem}{Theorem}[section]
\newtheorem{lemma}[theorem]{Lemma}
\newtheorem{prop}[theorem] {Proposition}

\newtheorem{definition}[theorem] {Definition}

\theoremstyle{definition}
\newtheorem{assumption}[theorem]{Assumption}
\theoremstyle{remark}
\newtheorem*{remark}{Remark}

\newcommand{\e}{\mathrm{e}} 
\newcommand{\N}{\mathbb{N}}
\newcommand{\R}{\mathbb{R}}

\newcommand{\C}{\mathbb{C}}

\renewcommand{\P}{\mathbb{P}}
\newcommand{\E}{\mathbb{E}}
\newcommand{\dd}{\mathrm{d}} 
\newcommand{\eps}{\varepsilon}

\renewcommand{\Re}{\mathrm{Re}\,} 

\newcommand{\V}     {\mathbb{V}}

\newcommand{\be}{\begin{equation}}
\newcommand{\ee}{\end{equation}}
\newcommand{\ba}{\begin{equation} \begin{aligned}}
\newcommand{\ea}{\end{aligned}\end{equation}}
\newcommand{\bes}{\begin{equation*}}
\newcommand{\ees}{\end{equation*}}

\def\1{{\mathchoice {1\mskip-4mu\mathrm l}      
{1\mskip-4mu\mathrm l}
{1\mskip-4.5mu\mathrm l} {1\mskip-5mu\mathrm l}}}

\definecolor{darkgreen}{rgb}{0.0, 0.5, 0.0}
\definecolor{purple}{rgb}{0.6, 0.4, 0.8}
\makeatletter
\def\@tocline#1#2#3#4#5#6#7{\relax
  \ifnum #1>\c@tocdepth 
  \else
    \par \addpenalty\@secpenalty\addvspace{#2}%
    \begingroup \hyphenpenalty\@M
    \@ifempty{#4}{%
      \@tempdima\csname r@tocindent\number#1\endcsname\relax
    }{%
      \@tempdima#4\relax
    }%
    \parindent\z@ \leftskip#3\relax \advance\leftskip\@tempdima\relax
    \rightskip\@pnumwidth plus4em \parfillskip-\@pnumwidth
    #5\leavevmode\hskip-\@tempdima
      \ifcase #1
       \or\or \hskip 1em \or \hskip 2em \else \hskip 3em \fi%
      #6\nobreak\relax
    \dotfill\hbox to\@pnumwidth{\@tocpagenum{#7}}\par
    \nobreak
    \endgroup
  \fi}
\makeatother

\begin{document}

\title{The method of cumulants for the normal approximation}
\author{Hanna D{\"o}ring}
\author{Sabine Jansen}
\author{Kristina Schubert}

\address[Hanna D{\"o}ring]{Institut f\"ur Mathematik, Universität Osnabr\"uck, Albrechtstr. 28a, 49076 Osnabr\"uck, Germany}
\email{hanna.doering@uni-osnabrueck.de}

\address[Sabine Jansen]{Mathematisches Institut, Ludwig-Maximilians-Universit\"at M\"unchen, Theresienstr.~39, 80333 Munich, Germany; Munich Center for Quantum Science and Technology (MCQST), Schelling\-str.~4, 80799 Munich, Germany}
\email{jansen@math.lmu.de}

\address[Kristina Schubert]{Institut f\"ur Mathematik, Universität Osnabr\"uck, Albrechtstr. 28a, 49076 Osnabr\"uck, Germany}
\email{kristina.schubert@uni-osnabrueck.de}

\date{4 March 2021}

\begin{abstract} 
The survey is dedicated to a  celebrated series of quantitave results, developed by the Lithuanian school of probability, on the normal approximation for a real-valued random variable. The key ingredient is a bound on cumulants of the type $|\kappa_j(X)| \leq j!^{1+\gamma} /\Delta^{j-2}$, which is weaker than Cram{\'e}r's condition of finite exponential moments.  We give a self-contained proof of some of the ``main lemmas'' in a book by Saulis and Statulevi{\v{c}}ius (1989), and an accessible introduction to the Cram{\'e}r-Petrov series.  In addition, we explain relations with heavy-tailed Weibull variables, moderate deviations, and mod-phi convergence. We  discuss some methods for bounding cumulants such as summability of mixed cumulants and dependency graphs, and briefly review a few recent applications of the method of cumulants for the normal approximation.\\


\end{abstract}

\maketitle

\tableofcontents

\section{Introduction}

\subsection{Aims and scope of the article} 
The method of cumulants is a central tool in comparing the distribution of a random variable with the normal law. It enters the proof of central limit theorems, moderate and large deviation principles, Berry-Esseen bounds, and concentration inequalities in various fields of probability theory: stochastic geometry, random matrices, random graphs, random combinatorial structures, functionals of stochastic processes, mathematical biology---the list is not exhaustive. 

The present survey shines a spotlight on a celebrated series of bounds developed by the Lithuanian school, notably Rudzkis, Saulis, and Sta\-tu\-levi{\v{c}}ius \cite{rudzkis-saulis-statulevicius78} and Bentkus and Rudzkis \cite{bentkus}. The bounds work under a condition on cumulants that allows for heavy-tailed behavior and is considerably weaker than the Cram{\'e}r condition of finite exponential moments frequently invoked in the theory of large deviations. The conditions on cumulants can be verified in many situations of interest, beyond sums of independent random variables. In their monograph \cite{SS91}, Saulis and Sta\-tu\-levi{\v{c}}ius study applications, for example, to random processes with mixing, multiple stochastic integrals, and $U$-statistics. 

Our primary aim is to give a self-contained and accessible presentation, including proofs, of the ``main lemmas'' from Chapter 2 in the book~\cite{SS91} by Saulis and Statulevi{\v{c}}ius; along the way, we correct a few minor errors in the proofs. We have not aimed at a reconstruction of all numerical constants. The presentation should be accessible to a reader with little exposure to complex-analytic or Fourier methods in probability or slightly arcane concepts such as the Cram{\'e}r-Petrov series. The reader familiar with the classical books by Ibragimov and Linnik \cite{ibragimov-linnik}, Petrov \cite{petrov-book}, Gnedenko and Kolmogorov \cite{gnedenko-kolmogorov}, or even Feller \cite{feller-vol2} will easily recognize a classical set of ideas, however an in-depth study of characteristic functions, inversion formulas, and asymptotic expansions are nowadays frequently eschewed in graduate probability courses and these methods are no longer part of every probabilist's toolbox. Moreover the presentation in \cite{SS91} is extremely technical, making it very hard to extract the proof philosophy from the long series of technical estimates. To remedy this situation, we have strived to make explicit the logical structure and key ideas, notably the truncation procedures needed to deal with Taylor expansions with zero radius of convergence. 

In addition, we mention a few exemplary applications and discuss the relation of the aforementioned results with other techniques and fields, in particular, large deviations for sums of heavy-tailed variables \cite{embrechts-klueppelberg-mikosch,mikosch-nagaev1998}, analytic combinatorics \cite{flajolet-sedgewick-book}, and mod-phi convergence \cite{FMNbook}. 

In the remaining part of the introduction, we define the cumulants, summarize the main bounds studied in the present survey, address methods for bounding cumulants, and list a few recent applications. 

\subsection{Cumulants}
 The \emph{cumulants} of a real-valued random variable $X$ are the numbers $\kappa_j\equiv \kappa_j(X)$, $j\in \N$, given by
\[
	 \kappa_j(X):=(-\mathrm i)^j \frac{\dd^j} {\dd t^j} \log \E\bigl[\e^{\mathrm itX}] \Big|_{t=0},
\] 
provided the derivative exists. 
Equivalently, assuming $r$-fold differentiability of the characteristic function at the origin, the cumulants of order $j=1,\ldots, r$ are related to the Taylor expansion by
\[
	\log \E\bigl[\e^{\mathrm i t X}\bigr] = \sum_{j=1}^r \frac{\kappa_j}{j!} (\mathrm i t)^j + o(t^r) \quad (t\to 0).
\] 
The cumulant of order $1$ is the expected value $\kappa_1 = \E[X]$, the cumulant of order $2$ is the variance $\kappa_2 = \mathbb V(X)$. More generally, there exists a recurrence formula 
 between the centered moments $\E[(X- \E[X])^j]$ and the cumulants, which is how Thiele~\cite{Thiele1889} originally defined them. Cumulants are often called \emph{semi-invariants} because of the homogeneity $\kappa_j(\lambda X) = \lambda^j \kappa_j(X)$ and shift-invariance $\kappa_j(X+c) = \kappa_j(X)$ for $j\geq 2$.  The name \emph{cumulants} was proposed by Fisher and Wishart \cite{FW32}; see Hald~\cite{hald2000} for a historical  account and a translation  of Thiele's article from Danish to English. 
 
There are many moment-to-cumulants formulas. We list a few to give a first impression of cumulants but emphasize that none of them, except perhaps the first, is used in the sequel. The most common relation, obtained from Fa{\`a} di Bruno's formula, is 
 \[
  \kappa_j(X)= \sum_{m=1}^j \frac{(-1)^{m-1}}{m} \sum_{ \genfrac{}{}{0pt}{}{k_1+\dots + k_m =j}{k_1,\dots,k_m\geq 1}} \frac{j!}{k_1! \cdots k_m!} \prod_{\ell=1}^m \E\bigl[ X^{k_\ell}\bigr].
 \]
Equivalently, the $j$-th cumulant is a sum over set partitions $\{B_1,\ldots, B_m\}$ of $\{1,\ldots, j\}$, 
\[
	\kappa_j(X) = \sum_{m=1}^j \sum_{\{B_1,\ldots, B_m\}} (-1)^{m-1}(m-1)! \prod_{\ell=1}^m \E\bigl[X^{\#B_\ell}\bigr]. 
\]  
It can be obtained by a M{\"o}bius inversion on the lattice of set partitions, the function $(-1)^{m-1}(m-1)!$ is the M{\"o}bius function~\cite{Speed:1983}. Another set of relations is obtained by differentiating the logarithm of the characteristic function: The relation 
\[
	\frac{\dd}{\dd t} \E \bigl[ e^{\mathrm i t X} \bigr]  =   \E \bigl[ e^{\mathrm i t X} \bigr] \frac{\dd} {\dd t}\log \E \bigl[e^{\mathrm i t X}\bigr]
\]	
implies the recurrence relations
\[
  \kappa_j(X)= \E\bigl[ X^j \bigr]- \sum_{r=1}^{j-1} \binom{j-1} {r-1} \E\bigl[X^{j-r}\bigr] \kappa_r(X).
\]
Cram\'er's rule for solving linear equations yields an expression of the cumulant as the determinant of a Toeplitz matrix~\cite[Cor.~3.1]{RotaShen:2000},
 \[
	 \kappa_j(X)= (-1)^{j-1} (j-1)! \begin{vmatrix}
	 \E X & \frac{\E X^2}{2!} & \cdots & \frac{\E X^{j-1}}{(j-1)!} & \frac{\E X^k}{k!} \\
	 1 & \E X & \cdots  & \frac{\E X^{j-2}}{(j-2)!} & \frac{\E X^{j-1}}{(j-1)!}\\
	 \vdots & \vdots & \ddots & \vdots & \vdots \\
	 0 & 0 & \cdots & \E X & \frac{\E X^2}{2!}\\
	 0 & 0 & \cdots & 1 & \E X
	 \end{vmatrix}.
 \]
 
Cumulants offer some advantages over moments. Crucially, the $j$-th order cumulant of a sum of independent random variables is simply the sum of the $j$-th order cumulants.  For a standard normal random variable all cumulants of order $j\geq 3$ vanish. Cumulants of a Poisson random variable with parameter $\lambda$ are all given by $\kappa_j \equiv\lambda$, while formulas for moments are more involved. 

Cumulants help prove central limit theorems: A sequence $(X_n)_{n\in \N}$ of random variables converges in distribution to a standard normal variable if and only if the expectation and the variance go to zero and one, respectively, and in addition all higher-order cumulants go to zero. In fact for the higher-order cumulants it is enough to check that the cumulants with $j$ larger than any fixed order $s\geq 3$ go to zero, see Janson~\cite{JansonCLT}. Bounds on cumulants translate into quantitative bounds for the normal approximation. 

\subsection{Short description of the ``main lemmas''} 
The main theorems discussed in the present survey apply to real-valued random variables $X$ for which all moments, hence also all cumulants, exist and for which the cumulants can be bounded as
\begin{equation*}
	 |\kappa_j(X)| \leq \frac{(j!)^{1+\gamma}} {\Delta^{j-2}} \qquad (j\geq 3)
\end{equation*}
with $\gamma\geq 0$, $\Delta >0$. The variable $X$ is assumed to be centered and normalized, $\E[X] =0$ and $\V(X) =1$. Following \cite{amosova99} we shall refer to this condition as the \emph{Statulevi{\v{c}}ius condition}. The main results, roughly, are the following:
\begin{enumerate} 
	\item \emph{Normal approximation with Cram{\'e}r corrections.} Let $Z\sim \mathcal N(0,1)$ be a standard normal variable. Then 
	for $x\in (0,c \Delta^{1/(1+2\gamma)})$ with suitable constant $c>0$, 
	\[
		\P(X\geq x) = \e^{\tilde L(x)} \P(Z\geq x) \Bigl( 1+ O\Bigl( \frac{x+1}{\Delta^{1/(1+2\gamma)}} \Bigr)\Bigr)
	\] 
	where $\tilde L(x)$ is related to the so-called \emph{Cram{\'e}r-Petrov series} (reviewed in Appendix~\ref{app:cramer}) and satisfies $|\tilde L(x)|= O(x^3 /\Delta^{1/(1+2\gamma)})$. See Rudzkis, Saulis, and Statulevi{\v{c}}ius \cite{rudzkis-saulis-statulevicius78}, Lemma 2.3 in \cite{SS91}, and Theorem~\ref{thm:lemma23} below.
	\item \emph{Bound on the Kolmogorov distance}. The following bound of Berry-Esseen type holds true:
	\[
		\sup_{x\in \R}\bigl|\P(X\leq x) - \P(Z\leq x)\bigr| \leq \frac{C}{\Delta^{1/(1+2\gamma)}}
	\] 
	for some constant $C>0$. See Rudzkis, Saulis, and Statulevi{\v{c}}ius \cite{rudzkis-saulis-statulevicius78}, Corollary~2.1 in \cite{SS91}, and Theorem~\ref{thm:besseen} below.
	\item \emph{Concentration inequality}. Assuming $|\kappa_j|\leq \frac12 j!^{1+\gamma} H/\Delta^{j-2}$ for some $H,\Delta >0$ and all $j$  (a minor variant of the Statulevi{\v{c}}ius condition), one has 
	\begin{align*}
		\P( X \geq x) & \leq  \exp \Bigl( - \frac12 \frac{x^2}{H+ x^{2-\alpha}/\Delta^\alpha}\Bigr),\quad \alpha:= \frac{1}{1+\gamma}\\
			&\leq \exp\Bigl( - \frac 14 \min \Bigl( \frac{x^2}{ H}, (x \Delta)^\alpha\Bigr)\Bigr)
	\end{align*} 
	for all $x\geq 0$. See Bentkus and Rudzkis \cite{bentkus}, the corollary to Lemma~2.4 in \cite{SS91}, and Theorem~\ref{thm:concentration} below. 
\end{enumerate}
Let us briefly put these results in perspective and discuss the nature of the Statulevi{\v{c}}ius condition. When $\gamma =0$, the condition implies that the cumulant generating function $\varphi(t) = \log \E[\exp(t X)]= \sum_j \kappa_j t^j/ j!$ is analytic in a neighborhood of the origin and finite for $|t|<\Delta$. This is precisely Cram{\'e}r's condition of finite exponential moments. The normal approximation with Cram{\'e}r corrections is proven with standard techniques also employed for sums of independent identically distributed random variables \cite[Chapter 8]{ibragimov-linnik}.  The concentration inequality is similar to the Bernstein inequality \cite{eom-bernstein}. 

The bounds are more intriguing for $\gamma>0$. In that case the Taylor expansions  $\sum_j \kappa_j t^j/j!$ of the cumulant generating function can have zero radius of convergence and the random variable can be  heavy-tailed, meaning that it has infinite exponential moments $\E[\exp(t X)]$  for arbitrarily small $t\neq 0$.  The concentration inequality shows that the tails of $X$ have at least stretched exponential decay $O(\exp( - \mathrm{const}\, x^\alpha))$ with $\alpha = 1/(1+\gamma)$, i.e., $X$ has (in the worst case) Weibull-like tails. In fact there is equivalence: It is known that the Statulevi{\v{c}}ius condition holds true if and only if Linnik's condition $\E[\exp( \delta |X|^\alpha)]<\infty$, for some $\delta>0$, is satisfied, see Section~\ref{sec:cralistat}. Results on large deviations under conditions more general than Linnik's condition are available as well, see Section~\ref{sec:cralistat}.  
Hence, the concentration inequality morally says that \emph{if a random variable has cumulants similar to those of a Weibull-like variable, then it also has Weibull-like tails}. 

This result is rather amazing. True, it is well-known that statements on the tails of a random variable can be inferred from information on the characteristic function $\chi(t) = \E[\exp(\mathrm i t X)]$ near $t=0$--- or, if the random variable is non-negative or with values in $\N_0$, from the behavior Laplace transform $\E[\exp( - \lambda X)]$ as $\lambda \downarrow 0$ or the behavior of the probability generating function $G(z) = \E[z^X]$ as $z\to 1$. Such relations are at the heart of analytic proofs of limit theorems in probability and combinatorics with Fourier analysis, Tauberian theorems, or complex analysis \cite{ibragimov-linnik, flajolet-sedgewick-book}. However, it is not clear at all that the coefficients of a divergent Taylor expansion carry enough information to allow for rigorous statements. 

Proofs for $\gamma >0$ require ingenious truncation procedures. Some of them are similar to procedures employed for large deviations of sums of heavy-tailed random variables, see Section~\ref{sec:heavy-tailed}. The critical scale $\Delta^{1/(1+2\gamma)}$ is comparable to the monomial zones of attraction to the normal law and ``Cram{\'e}r's system of limiting tails'' for sums of i.i.d.\ random variables discussed by Ibragimov and Linnik \cite[Chapters 9 and 10]{ibragimov-linnik}. Let us stress that the critical scale is not just some technical limitation. For sums of i.i.d.\ Weibull-like variables, it corresponds exactly to the scale at which the normal approximation ceases to be good and heavy-tailed effects kick in \cite{nagaev68}, see also Section~\ref{sec:weibull-ex} for an elementary example. 

 Note that the same cumulant bound also implies Rosenthal type bounds i.e.~estimates for the difference of the $k$-th moment to the corresponding moment of the normal distribution, see \cite{EichelsbacherK19}.
Connections with mod-phi convergence and moderate deviation principles are discussed in Section~\ref{sec:heavy-tailed} and~\ref{sec:modphi}. 

\subsection{How to bound cumulants}  There is no one-size-fits-all way to bound cumulants. Nevertheless, there are some repeating features, depending on the type of bound one seeks to establish and the type of random variable under investigation. For the bound, it matters whether one is after an analytic bound $\gamma =0$ or after the weaker case $\gamma>0$. Models fall broadly in two classes: first, random variables built out of fields or processes with underlying independence or good control on dependencies and decay of correlations, for example, empirical mean for a stationary ergodic process or functionals of Poisson point processes; second, models directly defined on more complex structures, including random matrices, models from analytic and probabilistic number theory, or random combinatorial structures.

The Statulevi{\v{c}}ius bound for $\gamma =0$ implies that the cumulant generating function is analytic.
Conversely, when the cumulant generating function is analytic in a neighborhood of the origin, then the cumulants can be bounded by applying Cauchy's formula. Accordingly one may shift perspective away from the cumulants and focus directly on generating functions. This is especially helpful for random combinatorial structures, where often the recursive structure of combinatorial objects translates into functional equations for generating functions and information on the analytic behavior  \cite{flajolet-sedgewick-book}. Working directly with generating functions is also of advantage for random matrices and probabilistic number theory \cite{FMNbook}. 

Analyticity of the cumulant generating function is equivalent to zero-freeness of the moment generating function. The role of analyticity and zero-freeness was already emphasized in Bryc's central limit theorem \cite{bryc1993}. In statistical mechanics, controlling zeros is related to Lee-Yang theory and  relations with central limit theorems were explored, for example, by Iagolnitzer and Souillard \cite{iagolnitzer-souillard1979} or Ruelle, Pittel, Lebowitz and Speer \cite{lebowitz-pittel-ruelle-speer2016}; see also \cite[Chapter 8.1]{FMNbook}. For new results and an account of modern developments, the reader is referred to Michelen and Sahasrabudhe \cite{michelen-saha2019,michelen-saha2019b} (Section~7 in \cite{michelen-saha2019b} has the telling title ``Taming the cumulant sequence'').

When a direct control of generating functions is not possible, cumulants are often treated with combinatorial bounds and correlation estimates. 
Two prototypes are sums of dependent random variables and $U$-statistics of sequences of independent random variables \cite[Chapters 4 and 5.1]{SS91} or $m$-dependent vectors \cite{Heinrich1985, Heinrich1990}. 
 Consider for example a sequence of independent random variables $(X_n)_{n\in \N}$ and the random variable $Y=  \sum_{i =1}^{n-1} X_i X_{i+1}$. The cumulant of $Y$ is not the sum of the cumulants $\kappa(X_iX_j)$ because $X_iX_{i+1}$ and $X_kX_{k+1}$ are not independent for $\{i,i+1\}\cap \{k,k+1\}\neq \varnothing$. But clearly for a given nearest-neighbor pair $\alpha = \{i,i+1\}$ the number of pairs $\beta = \{j,j+1\}$ with $\alpha \cap \beta \neq \varnothing$ is bounded by $2$. This can be exploited with \emph{dependency graphs}, see Section~\ref{sec:dependency}. 
 
Another example is when the input variables $X_n$, $n\in \N$, are not independent. The cumulants of the partial sum are given by the sum of mixed cumulants $\kappa(X_{\alpha_1},\ldots, X_{\alpha_j})$ \cite{LeonovS1959}  (see Eq.~\eqref{eq:mixed-cumulants} below) as 
\[
	\kappa_j(X_1+\cdots + X_n) = \sum_{1\leq \alpha_1,\ldots, \alpha_j\leq n} \kappa(X_{\alpha_1},\ldots X_{\alpha_j}),
\] 
hence
\[
	\bigl|\kappa_j(X_1+\cdots + X_n)\bigr|\leq n \sup_{\alpha_1\in \N}\sum_{\alpha_2,\ldots, \alpha_j\in \N} \bigl|\kappa(X_{\alpha_1},\ldots, X_{\alpha_j})\bigr|.
\] 
Thus bounds on the cumulants of the sum are intimately tied to summability properties of mixed cumulants.  Analogous relations apply in the context of point processes; here summability of mixed cumulants is replaced with bounds on the total variation of reduced (factorial) cumulant measures, which leads to the notions of weak and strong \emph{Brillinger mixing} \cite{brillinger1975,ivanoff1982}. In the language of statistical mechanics, Brillinger mixing corresponds to integrability of truncated correlation functions, a condition typically satisfied by Gibbs point processes at low density \cite{ruelle-book}. 

Brillinger mixing can be rather difficult to establish directly. An alternative approach, still focused on the decay of correlations, was devised by Baryshnikov and Yukich \cite{baryshnikov-yukich2005} and then further developed, see B{\l}aszczyszyn, Yogeshwaran and Yukich \cite{BYY2018} and the references therein. Eichelsbacher, Rai{\v{c}}  and Schreiber \cite{EichRaicSchreiber} pushed the method all the way up to the Statulevi{\v{c}}ius condition. 
The method works with approximate factorization properties of moment measures; cumulants are represented as combinations of \emph{semi-cluster measures}. Combinatorics enter when bounding the number of summands in the latter representation  \cite[Lemma 3.2]{EichRaicSchreiber}.

For functionals of Markov chains or stochastic processes with mixing, it is convenient to work with another set of quantities, called \emph{centered moments} \cite[Chapter 4]{SS91}, \emph{higher-order covariances} \cite{Heinrich2007} or \emph{Boolean cumulants} \cite[Section 10]{Feray:2018}. Mixing properties of the underlying stochastic process lead to good bounds on the centered moments, and then bounds on cumulants are deduced from a Boolean-to-classical cumulants formula.

For Poisson or Gaussian input data, another class of methods builds on diagrammatic formulas for cumulants  and chaos decompositions (related to Feynman diagrams and Fock spaces in mathematical physics). This applies to multiple stochastic integrals of Brownian motion or of Poisson point processes \cite[Chapter 5.3]{SS91}. A modern account of diagrammatic formulas is given by Peccati and Taqqu \cite{PeccatiTaqqu}. Useful formulas for cumulants can also be derived using Malliavin calculus and the infinite-dimensional Ornstein-Uhlenbeck operator, see Nourdin and Peccati \cite{NP2009}.

\subsection{Some recent applications}
In recent years the method of cumulants has attracted interest in various areas of application. We list a few; the examples below involve bounds on cumulants though not necessarily of the form  $|\kappa_j|\leq j!^{1+\gamma}/\Delta^{j-2}$.

In  stochastic geometry, the method of cumulants is used in numerous ways. In \cite{GroteThaele:2015} and \cite{GT2016} it is used for the volume and the number of faces of a Gaussian polytope to show concentration inequalities and a Marcinkiewicz-Zygmund-type strong law of large numbers as well as a central limit theorem including error bounds and moderate deviations. 
Furthermore for the volumes of random simplices a central limit theorem, mod-phi convergence as well as  moderate and large deviations are proved in \cite{GKT17}, where the dimension and the number of points grow to infinity. 
Poisson cylinder processes are studied in \cite{HeinrichSpiess}, \cite{HeinrichSpiess13} and volumes of simplices in high-dimensional Poisson-Delaunay tessellations in \cite{GT:2021}.
The method is also applicable to the covered volume in the Boolean model, \cite{GoetzeHeinrichHipp95,Heinrich2005,Heinrich2007}.
General functionals of random $m$-dependent fields are studied in \cite{GoetzeHeinrichHipp95,Heinrich1985,Heinrich1990,HeinrichRichter}. 
 A survey covering $m$-dependent variables, Markov chains, and mixing random variables is given by Heinrich \cite{heinrich1987survey}.
Moderate deviation results for classical stabilizing functionals in stochastic geometry can be found in \cite{EichRaicSchreiber}.

 The method of cumulants also plays  an important role in the theory of random matrices and determinantal point processes. For the latter cumulants for the sine kernel were  studied by Costin and Lebowitz \cite{costin}.  In \cite{Soshnikov:2000} and \cite{Soshnikov:2002} Soshnikov studied the Gaussian limit for linear statistics of Gaussian fields and determinantal random point fields. The methods of \cite{costin} were also applied to spacing distribution in the circular unitary ensemble \cite{soshnikov1998}. The method of \cite{Soshnikov:2000} was further extended e.g.~to study general one-cut unitary-invariant matrix models  \cite{lambert} and to prove mesoscopic fluctuations for unitary invariant ensembles \cite{lambert2018}.  Further the methods were applied to (generalized) Ginibre ensembles, where linear eigenvalue statistics and characteristic polynomials were studied \cite{rider2007} and  moderate deviations for counting statistics were shown \cite{lambert2020}.
The method of cumulants is also applied for eigenvalue counting statistics and determinants of Wigner matrices \cite{DEi1,DEi4}. This was generalized to the linear spectrum statistics of orthogonal polynomial ensembles in
\cite{PWZ:2017}. Studying the cumulants also works for characteristic polynomials of random matrices from the circular unitary ensemble \cite{CNN17} as well as the determinants of random block Hankel matrices, \cite{DetteT:2019}. Cumulants also enter the analysis of matrix models in which the eigenvalues do not form a determinantal point process, see Borot and Guionnet \cite{borot}.

The method of cumulants, combined with the concept of dependency graphs, yields useful results in random graphs, in particular subgraph count statistics in Erd\H{o}s-R\'enyi random graphs \cite{DEi3,FMNbook}, or more generally for graphons in \cite{FMN17}, for the profile of a branching random walk \cite{GruebelKabluchko} as well as for the winding number of Brownian motion in the complex plane \cite{DelbaenModPhi}.
F{\'e}ray gives criteria for normality convergence of dependency graphs by using the method of cumulants \cite{Feray:2018}. This is applied to vincular permutation patterns in \cite{Hofer:2017} (yielding the same order of convergence in the Kolmogorov distance than achieved by applying Stein's method).

The generalization to weighted dependency graphs allows for applications to the Ising model as \cite{DousseF:2019} shows. Already in \cite{MeliotModPhi} models of statistical mechanics such as the Curie-Weiss and the Ising model are studied.

For an application for multiple Wiener-It\^{o} integrals in stochastic analysis see \cite{NP2009}, the survey of a series of articles in the book \cite{PeccatiTaqqu} as well as \cite{SchulteThaele:2014}.

Another area of applications are random logarithmic structures in combinatorics, random permutations and random integer partitions as well as so-called character values, see \cite{Barbour, FMNbook, FMN17}. Random arithmetic functions, the Riemannian $\zeta$ function and $L$ functions over finite fields are considered in the series of publications \cite{DelbaenModPhi,FMNbook,JacodModPhi,KowalskiModPhi10,KowalskiModPhi12}. The latest publication \cite{Feray:2020} shows a central limit theorem for weighted dependency graphs and generalizes results for the number of occurrences of any fixed pattern in multiset permutations and in set partitions.

In mathematical biology M\"ohle and Pitters derived the absorption time and tree length of the Kingman coalescent by bounding the cumulants, see \cite{MoehlePitters}. Restricted to the infinitely many sites model of Kimura, in \cite{Pitters2017} Pitters derives a formula for the cumulants of the number of segregating sites of a Kingman coalescent implying a central limit theorem and a law of large numbers. Studying the multivariate cumulants, the number of cycles in a random permutation and the number of segregating sites jointly converge to the Brownian sheet \cite{Pitters2019}.

To conclude this non-exhaustive list, we stress that cumulants are highly relevant in areas somewhat outside the scope of this survey. Perhaps closest to traditional probability are asymptotic techniques in statistics, see Barndorff-Nielsen and Cox \cite{barndorff-nielsen-cox}. A tensorial view of cumulants, with applications in statistics, is given by McCullagh \cite{McCullagh87}. Cumulants are also useful in algebraic statistics, see the book by Zwiernik \cite{ZwiernikBook}. For example, Sturmfels and Zwiernik \cite{SturmfelsZwiernik:2013} discuss cumulants for algebraic varieties and binary random variables on hidden subset models. A completely different line of inquiry is free probability and non-commutative probability, in which different notions of independence come with different notions of cumulants. The relation between free, monotone and Boolean cumulants is studied by Arizmendi, Hasebe, Lehner and Vargas \cite{AHLV2015}. Finally, cumulants also feature prominently in kinetic theory and the analysis of time-dependent models in mathematical physics~\cite{lukkarinen-marcozzi-nota, bodineau-straymond-hardspheres}.

\section{The main lemmas}

Here we state four theorems, roughly the ``main lemmas'' in Saulis and Statulevi{\v{c}}ius \cite[Chapter 2]{SS91}. Two theorems are about the normal approximation, with Cram{\'e}r corrections, to a random variable under conditions on finitely many cumulants (Theorem~\ref{thm:lemma22}) or under the growth condition~\eqref{condition-sgamma}, which allows for heavy-tailed random variables (Theorem~\ref{thm:lemma23}). An 
inequality of the Berry-Esseen type and a concentration inequality  are given in Theorems~\ref{thm:besseen} and~\ref{thm:concentration}, again under the condition~\eqref{condition-sgamma}.

The main theorems are illustrated with two elementary examples---Gamma and Weibull dis\-tri\-butions---in Section~\ref{sec:examples}. The meaning of the Statulevi{\v{c}}ius condition~\eqref{condition-sgamma} is further clarified in Section~\ref{sec:cralistat}. 

\subsection{Normal approximation with Cram{\'e}r corrections} 
In the following $X$ is a real-valued random variable, defined on some probability space $(\Omega,\mathcal F, \P)$, with cumulants $\kappa_j(X)= \kappa_j$. We assume throughout the text that the variable $X$ is normalized, i.e., $\E[X]=0$ and $\mathbb V(X) =1$. Two important quantities are the cumulant generating function 
\be
	\varphi(t) := \log \E\bigl[\exp(tX)\bigr] \in \R \cup\{\infty\} \qquad (t\in \R)
\ee
and its Legendre transform 
\be
	I(x):= \sup_{t\in \R}\bigl( t x - \varphi(t)\bigr) \qquad (x\in \R).
\ee
The first theorem works under the condition that there exists some $\Delta>0$ such that 
\be \label{condition1} \tag{$\mathcal S$}
	\forall j \geq 3: \quad |\kappa_j(X)|\leq \frac{(j-2)!}{\Delta^{j-2}}.
\ee
Under this condition the cumulant generating function is finite on $(-\Delta,\Delta)$ with absolutely convergent Taylor expansion 
\be
	\varphi(t) = \frac{ t^2}{2} + \sum_{j=3}^\infty \frac{\kappa_j}{j!}\, t^j \quad (|t|<\Delta),
\ee
where we have used $\kappa_1 = \E[X]=0$ and $\kappa_2 = \mathbb V(X) =1$. The Cram{\'e}r rate function admits a Taylor expansion as well (Proposition~\ref{prop:cramer-petrov1}), with radius of convergence at least $0.3\, \Delta$ (Proposition~\ref{prop:cramer-concrete1}).  The expansion is of the form 
\be
	I(x) = \frac{x^2}{2} - \sum_{j=3}^\infty \lambda_j x^j  \qquad (|x|<\frac{3}{10} \, \Delta).
\ee
The series $\sum_{j=3}^\infty \lambda_j \, x^{j-3}$ is called \emph{Cram{\'e}r series} or \emph{Cram{\'e}r-Petrov series} after~\cite{cramer38,petrov54}. Cram{\'e}r's original article~\cite{cramer38} was recently made accessible in electronic form, together with an English translation, by Touchette~\cite{cramer-touchette2018}. Appendix~\ref{app:cramer} collects some relevant background on the series. 
It is convenient to set 
\be
	L(x):= \sum_{j=3}^\infty \lambda_j x^j
\ee
so that $I(x) = \frac12 x^2 - L(x)$. From now on $Z$ is always a standard normal variable, $Z\sim \mathcal N(0,1)$. 

\begin{theorem} \label{thm:easy}
	Under condition~\eqref{condition1} there exist universal constants $c,C,C'>0$ such that for all $x\in [0, c\Delta]$ and some $\theta=\theta(x) \in [-1,1]$, 
	$$
		\P(X\geq x) = \e^{L(x)} \P(Z \geq x) \Bigl( 1+ C \theta\,\frac{x+1}{\Delta}\Bigr)
	$$
	and $|L(x)| \leq C' x^3 /\Delta$. 
\end{theorem}

The theorem is proven in Section~\ref{sec:easy}. It is a special case of Theorem~\ref{thm:lemma23} below; we have chosen to provide a separate statement as it is easier to grasp, and its proof is a helpful warm-up for the proof of Theorem~\ref{thm:lemma23}.

The next theorem asks what subsists when the cumulants satisfy the bound~\eqref{condition1} only up to some order, i.e., 
\be \label{condition-s} \tag{$\mathcal S^*$}
	\forall j \in \{3,\ldots, s+2\}: \quad |\kappa_j(X)|\leq \frac{(j-2)!}{\Delta^{j-2}}
\ee
for some $s\in \N$. We say that $X$ satisfies condition~\eqref{condition-s} if all moments $\E[X^j]$, $j\leq s+2$ exist---hence also all cumulants $\kappa_j$ with $j\leq s+2$---and the cumulants satisfy the required inequality. Under condition~\eqref{condition-s} the random variable $X$ need not have exponential moments and the cumulant generating function may be infinite, therefore the definitions of $\varphi(t)$ and $I(x)$ are modified as follows. We set 
\be
	\tilde \varphi(t) = \frac{t^2}{2} + \sum_{j=3}^s \frac{\kappa_j}{j!} t^j.
\ee
For small $x$ and $t$ the equation $\tilde \varphi'(t) = x$ reads $t + O(t^2) = x$ and it has a solution $t(x) = x + \sum_{j=2}^\infty \tilde b_j x^j$ with suitably defined coefficients $\tilde b_j$. We define 
\be \label{eq:Ltildedef} 
	\tilde I(x):= t(x) x - \tilde \varphi(t(x)),\quad \tilde L(x): =  \frac{x^2}{2}-\tilde I(x)
\ee
and note that $\tilde L(x)$ has a Taylor expansion $\tilde L(x) = \sum_{j=3}^\infty \tilde \lambda_j x^j$ with radius of convergence at least $0.3\, \Delta$ (Propositions~\ref{prop:cramer-petrov2}, \ref{prop:cramer-concrete2}). In addition, $\tilde \lambda_j = \lambda_j$ for $j \leq s $ (Eq.~\eqref{eq:tilde-notilde}). 

\begin{theorem} \label{thm:lemma22}
	Let $X$ be a real-valued random variable with $\E[X]=0$, $\mathbb V(X) =1$. Assume that $X$ satisfies condition~\eqref{condition-s} for some even $s\geq 2$ and $\Delta>0$ with $s\leq 2 \Delta^2$.
	Then for all $x\in [0, \sqrt s/(3\sqrt \e))$ and some $\theta=\theta(x) \in [-1,1]$, 
	\[
		\P(X\geq x) = \e^{\tilde L(x)}\P(Z\geq x) \Bigl( 1+ \theta f(\delta,s)\, \frac{x+1}{\sqrt s}\Bigr)
	\]
	with 
	\[
		\delta = \frac{x}{\sqrt s/(3 \sqrt \e)}\in [0,1),\quad f(\delta,s) = \frac{1}{1-\delta}\Bigl( 127+113\, s\, \e^{- (1-\delta) s^{1/4} /2} \Bigr) 
	\] 	
	and $|\tilde L(x)| \leq 1.2\, x^3 /\Delta$. 
\end{theorem} 

\noindent The theorem is proven in Section~\ref{sec:sstar}. It corresponds to Lemma~2.2 in \cite{SS91} and is due to Rudzkis, Saulis, and Statulevi{\v{c}}ius \cite{rudzkis-saulis-statulevicius78}. The constants are slightly worse than the constants given in~\cite{SS91} but of a similar order of magnitude.
 We are not aware of any application of the concrete formula for $f(\delta,s)$. Instead, what matters is that $f(\delta,s)$ is bounded on $[0,\delta_0]\times \N$, for all $\delta_0<1$. 

The next theorem works under a condition that allows for heavy-tailed variables. For $\gamma>0$ and $\Delta>0$, consider the \emph{Statulevi{\v{c}}ius condition} 
\be \label{condition-sgamma} \tag{$\mathcal S_\gamma$}
	\forall j \geq 3: \quad |\kappa_j(X)|\leq \frac{j!^{1+\gamma}}{\Delta^{j-2}}.
\ee
The relation of this condition with Weibull tails and \emph{Linnik's condition} $\E[\exp( \delta |X|^{1/(1+\gamma)})]< \infty$ is clarified by Lemma~\ref{lem:conditionsequiv} below, see~\cite{amosova99}. 
Define 
\be \label{eq:dsmdef}
	\Delta_\gamma:= \frac 16 \Bigl( \frac{\Delta}{\sqrt{18}}\Bigr)^{1/(1+2\gamma)},
	\quad s_\gamma:= 2 \Bigl \lfloor \frac12 \Bigl( \frac{\Delta^2}{18}\Bigr)^{1/(1+2\gamma)}\Bigr \rfloor - 2, \quad m_\gamma:= \min \Bigl( \left \lceil \frac 1 \gamma \right \rceil +1, s_\gamma\Bigr).
\ee

\begin{theorem} \label{thm:lemma23}
	Let $X$ be a real-valued random variable with $\E[X]=0$, $\mathbb V(X) =1$. Suppose that $X$ satisfies condition~\eqref{condition-sgamma}. Then there exist some universal constant $C>0$ such that for all $x\in [0, \Delta_\gamma)$ and some $\theta=\theta(x) \in [-1,1]$,
	$$
		\P(X\geq x) = \e^{\tilde L_\gamma(x)} \P(Z \geq x) \Bigl( 1+ \theta g (\delta, \Delta_{\gamma})\, \frac{x+1}{\Delta_{\gamma}}\Bigr)
	$$
	with $|\tilde L_\gamma(x)|\leq x^3/(1.54\, \Delta_\gamma) $ and 
	\[
		\tilde L_\gamma(x) = \begin{cases} 
					\theta (\frac{x}{\Delta_\gamma})^3, &\quad \gamma \geq 1,\\
				\sum_{j=3}^{m_\gamma} \lambda_j x^j + \theta C_\gamma (\frac{x}{\Delta_\gamma})^{m_\gamma+1}, &\quad \gamma <1.
		\end{cases}  	
	\]
	Here
	we set 
	\[
		g(\delta,\Delta_\gamma) := \frac{1}{1-\delta} \Bigl( 24+ 749 \Delta_\gamma^2 \exp\bigl(- (1-\delta)  \sqrt {\Delta_\gamma}	\bigr) \Bigr),
	\] 
	and $\delta = x/\Delta_\gamma$.
\end{theorem} 

The theorem is proven in Section~\ref{sec:weibull}. It corresponds to Lemma~2.3 in \cite{SS91} and is due to Rudzkis, Saulis, and Statulevi{\v{c}}ius \cite{rudzkis-saulis-statulevicius78}. 
The constants given in \cite{SS91} are 60 and 600 instead of 24 and 749. Our second constant 749 is worse but our first constant 24 is better.

We shall see that under the conditions of the theorem, $s_\gamma$ is larger or equal to $4$ so that $m_\gamma =2$ for $\gamma \geq 1$ and $m_\gamma \geq 3$ for $\gamma<1$. 
Let us briefly comment on the two bounds for $\tilde L_\gamma(x)$ in Theorem~\ref{thm:lemma23}. The global bound $|\tilde L_\gamma(x)| \leq x^3/(1.54\, \Delta_\gamma)$ is similar to the bounds for $L(x)$ and $\tilde L(x)$ in Theorems~\ref{thm:easy} and~\ref{thm:lemma22}. It gives the leading order of $\tilde L_\gamma(x)$. Note, however, that it can be quite large, since $x^3/\Delta_\gamma$ can be of order up to $\Delta_\gamma^2$. The case distinction in Theorem~\ref{thm:lemma23} provides a representation of $\tilde L_\gamma(x)$ that is precise in the sense that the remainder is small when $x$ is small compared to $\Delta_\gamma$. 

\begin{remark} 
Correction terms from the Cram{\'e}r series need only be taken into account when $\gamma <1$. This corresponds to variables with tails decaying like $\exp( -c\, x^\alpha)$ with $\alpha = 1/(1+\gamma) > 1/2$, see Section~\ref{sec:weibull-ex} and Lemma~\ref{lem:conditionsequiv} below. It should be pointed out that the value $\alpha =1/2$ plays a role as well for large deviations of i.i.d.\ heavy-tailed variables with Weibull tails, see Mikosch and Nagaev~\cite[Proposition~3.1]{mikosch-nagaev1998}.
\end{remark} 

The error term in $O((x+1)/\Delta_\gamma)$ after the exponential is known to be not optimal for sums of i.i.d.\ random variables. If $X_n = S_n /\sqrt{n}$ is a normalized sum of i.i.d.\ random variables that satisfy the Statulevi{\v{c}}ius condition for some fixed $\Delta$, then $X_n$ satisfies the Statulevi{\v{c}}ius condition with a $n$-dependent $\Delta(n)$ proportional to $\sqrt n$ (see Section~\ref{sec:heavy-tailed}), and $1/\Delta_\gamma$ is of the order of $(x+1)/\sqrt n^\beta$ for some $\beta< 1$, which is larger than the error term $O(x/\sqrt n)$ proven e.g.\ in Ibragimov and Linnik  \cite[Eq.~(13.4.4)]{ibragimov-linnik}.

 The principal idea in the proof of Theorem~\ref{thm:lemma23} is to apply Theorem~\ref{thm:lemma22} for suitably chosen $s$ and $\Delta_s$ such that condition~\eqref{condition-s} is satisfied if  condition~\eqref{condition-sgamma} holds true. Thus, we seek $s$ and $\Delta_s$ such that 
$$
	\frac{j!^{1+\gamma}}{\Delta^{j-2}}\leq \frac{(j-2)!}{\Delta_s^{j-2}}\qquad (j=3,\ldots, s+2).
$$
The inequality is equivalent to 
$$
	\bigl( j(j-1)\bigr)^{1+\gamma}	(j-2)!^{\gamma}\leq \Bigl(\frac{\Delta}{\Delta_s}\Bigr)^{j-2} \qquad (j=3,\ldots, s+2)
$$
and 
\be \label{eq:weibullineq}
	\max_{k=1,\ldots,s} \Bigl( \frac{\gamma}{k} \log k! + \frac{1+\gamma}{k}\bigl( \log (k+2) + \log (k+1)\Bigr)   \leq \log \frac{\Delta}{\Delta_s}. 
\ee
By Stirling's formula, the term to be maximized behaves like 
$$
	\gamma \log \frac{k}{\e} + O\Bigl( \frac 1k \log k \Bigr) = \gamma\bigl( 1+ o(1)\bigr) \log k \qquad (k\to \infty).
$$
For a heuristic evaluation of~\eqref{eq:weibullineq}, let us keep the leading order term only, then~\eqref{eq:weibullineq} becomes $\gamma \log s \leq \log( \Delta/\Delta_s)$ hence $\Delta_s = \Delta / s^\gamma$. Then $s\leq 2\Delta_s^2$ if and only if $s\leq 2 (\Delta^2)^{1/(1+2\gamma)}$. This suggests to pick 
$
	\sqrt{s} = \mathrm{const}\, \Delta^{1/(1+2\gamma)},
$
which is precisely the power of $\Delta$ appearing in Theorem~\ref{thm:lemma23}, via $\Delta_\gamma$.

\subsection{Berry-Esseen bound and concentration inequality} 

Theorem~\ref{thm:lemma23} is complemented by a Berry-Esseen bound and a concentration inequality that provides statements for all of $x\geq 0$---no need to restrict to $x\in (0,\Delta_\gamma)$. 

\begin{theorem} \label{thm:besseen}
	Under the Statulevi{\v{c}}ius condition~\eqref{condition-sgamma}, we have 
	\[
		\sup_{x\in\R} \bigl| \P(X\geq x) - \P(Z\geq x)\bigr| \leq \frac{C_\gamma}{\Delta^{1/(1+2\gamma)}}.
	\]
	for some constant $C_{\gamma}$ that does not depend on the random variable $X$ or on $\Delta$. 
\end{theorem}

Theorem~\ref{thm:besseen} is proven in Section~\ref{sec:besseen}, it corresponds to Corollary~2.1 in~\cite{SS91}. The precise bound given by~\cite{SS91} is $18/\Delta_\gamma$ with $\Delta_\gamma$ defined in~\eqref{eq:dsmdef}, we have not checked the numerical constants. 

\begin{theorem}\label{thm:concentration} 
	Suppose $\E[X] =0$ and 
	\be 
		|\kappa_j|\leq \frac{j!^{1+\gamma}}{2} \frac{H}{\overline{\Delta}^{j-2}}.
	\ee
	for some $\gamma \geq 0$ and $H,\overline{\Delta}>0$. 	
	Set $\alpha:= 1/(1+\gamma)$. Then there exists $C>0$ such that for all $x\geq 0$,  
	\be \label{eq:concentration} 
		\P(X\geq x) 		
		 \leq C \exp\Bigl(- \frac12 \frac{x^2}{H+ x^{2-\alpha} /{\bar \Delta}^\alpha} \Bigr).
	\ee
	The constant does not depend on $X$, $H$, $\overline{\Delta}$, or $\gamma$.
\end{theorem} 

Theorem~\ref{thm:concentration} is proven in Section~\ref{sec:concentration}, it corresponds to Lemma~2.4 in~\cite{SS91} and is due to Bentkus and Rudzkis \cite{bentkus}. As noted by Kallabis and Neumann \cite{kallabis-neumann2006}, the statement of Lemma~2.4 in \cite{SS91} contains a typo. Bentkus and Rudzkis give the constant $C=1$, we give a shorter proof but provide no concrete numerical bound on $C$.

The exponent in~\eqref{eq:concentration} can be expressed in terms of the harmonic mean $\mathrm{Harm}(a,b) = 2 (a^{-1} + b^{-1})^{-1}$ as 
\[
	\frac{x^2}{H+ x^{2-\alpha}/{\bar \Delta}^\alpha} = \Bigl( \frac{H}{x^2} + \frac{1}{(x\bar \Delta)^\alpha}\Bigr)^{-1} = \frac 12 \mathrm{Harm}\Bigl( \frac{x^2}{H}, (x \bar \Delta)^\alpha\Bigr). 
\] 
Thus Theorem~\ref{thm:concentration} provides an upper bound that smoothly interpolates between Gaussian tails $\exp( - x^2/(2H))$ and stretched exponential tails $\exp( - (x \bar \Delta)^\alpha/2)$.

\subsection{Two examples} \label{sec:examples}

Before we turn to the proofs, we provide two examples that illustrate the theorems and explain the role of $\Delta$ and $\Delta^{1/(1+2\gamma)}$. 

\subsubsection{A Gamma-distributed random variable} 
Pick $\Delta>0$ and let  $Y$ be a Gamma random variable with parameters $\beta = \Delta$ and  $\alpha = \Delta^2$. Thus the random variable $Y$ has probability density function $\Gamma(\alpha)^{-1}\1_{(0,\infty)}(x) \beta^\alpha x^{\alpha -1 }\e^{-\beta x}$,  moment generating function $(1 - t/\beta)^{-\alpha}$, variance $\alpha /\beta^2=1$ and expected value $\alpha /\beta=1$. 
Set $X := Y-\Delta$. Then $X$ has pdf 
\be \label{eq:gamma}
	\rho_\Delta(x) =\frac{\Delta^{\Delta^2}}{\Gamma(\Delta^2)} (x+\Delta)^{\Delta^2 - 1} \e^{- \Delta (x+ \Delta) } \1_{[-\Delta,\infty)}(x)
\ee
and cumulant generating function, for $|t|<\Delta$, given by
$$
	\varphi(t) = \log \E\bigl[\e^{tX}\bigr] = - \Delta^2\log \Bigl( 1- \tfrac t\Delta\Bigr) - t\Delta 
		= \sum_{j=2}^\infty \frac{1}{j} \frac{t^j}{\Delta^{j-2}},
$$
from which we read off the cumulants 
$$
	\kappa_j = \frac{(j-1)!}{\Delta^{j-2}}\qquad (j\geq 2).
$$
The explicit formula for the probability density function allows us to check that the normal approximation for $X$ is good when $\Delta$ is large and $x$ is small compared to $\Delta$. 

\begin{prop}
	As $\Delta\to \infty$ and $x/\Delta\to 0$, the probability density function~\eqref{eq:gamma} satisfies 
	$$
		\rho_\Delta(x) = \frac{\exp( - \tfrac12 x^2 [1+ O(\tfrac x\Delta)])}{\sqrt{2\pi}}\Bigl( 1+ O\bigl(\tfrac {x+1}{\Delta}\bigr)\Bigr).
	$$
\end{prop}

\begin{proof}
We rewrite 
\begin{align}
	\rho_\Delta(x) & = \frac{\Delta^{2\Delta^2- 1}}{\Gamma(\Delta^2)} (1+x/\Delta)^{\Delta^2-1} \e^{-\Delta^2(1+x/\Delta)} \1_{[-1,\infty)}(x/\Delta) \notag \\
		& = \frac{\Delta^{2\Delta^2- 1}\e^{-\Delta^2}}{\Gamma(\Delta^2)} \frac{1}{1+ x/\Delta}\,  \e^{ \Delta^2 [ \log (1+ x/\Delta) - x/\Delta]}  \1_{[-1,\infty)}(x/\Delta). \label{eq:gammadens}
\end{align} 
Using $\Gamma(x+1) = x\Gamma(x)$ and Stirling's approximation, we have 
\be \label{eq:gammastirling}
	\Gamma(\Delta^2) = \Delta^{-2} \Gamma(\Delta^2+1) = \Bigl(1+ O\bigl(\tfrac1\Delta)\Bigr) \Delta^{-2} \sqrt{2\pi \Delta^2} \Bigl( \tfrac{\Delta^2}{\mathrm{e}}\Bigr)^{\Delta^2} 
		 = \Bigl(1+ O\bigl(\tfrac1\Delta)\Bigr) \sqrt{2\pi} \Delta^{2\Delta^2 - 1} \e^{-\Delta^2}. 
\ee
Combining this with the Taylor expansion $\log (1+u) = u - u^2/2 + o(u^2)$ of the logarithm, we deduce
\begin{equation*}
	\rho_\Delta(x) = \Bigl( 1+ O\bigl(\tfrac{1}{\Delta}\bigr)\Bigr)\Bigl( 1+ O\bigl(\tfrac x\Delta\bigr)\Bigr) \frac{1}{\sqrt{2\pi}}\,\exp\Bigl( - \tfrac12 x^2 \bigl(1+ O(\tfrac x\Delta)\bigr)\Bigr). \qedhere
\end{equation*}
\end{proof}

\noindent The tilted normal approximation (compare Bahadur-Rao~\cite{bahadur-rao60}, \cite[Theorem~3.7.4]{dembo-zeitouni} or the proof of the lower bound in Cram{\'e}r large deviation principle  \cite[Chapter 2.2]{dembo-zeitouni}) consists in the following.  
Let $I(x):=\sup_{t\in \R}(t x - \varphi(t))$ be the rate function in the Cram{\'e}r large deviation principle. An explicit computation yields
$$
	I(x) = \Delta x - \Delta^2 \log(1+x/\Delta) = \frac12 x^2 + \sum_{j=3}^\infty \frac{(-1)^j}{j \Delta^{j-2}} x^j,
$$
from which we read off the Cram{\'e}r-Petrov series
$$
	L(x)=- \sum_{j=3}^\infty \frac{(-1)^j}{j \Delta^{j-2}} x^j \qquad (|x|<\Delta).
$$
For later purposes we extend the definition of $L(x)$ to all of $\R$ by putting $L(x)=I(x)- x^2/2$. Notice that as $x/\Delta\to 0$, 
$$
	L(x)= x^2\, O(x/\Delta).
$$
Given $x\geq0$, let $h\geq 0$ be the solution of $\varphi'(h) = x$ and let $\widehat X_h$ be a random variable with distribution $\P(\widehat X_h \leq y) = \e^{-\varphi(h)} \E[ \e^{hX} \1_{\{X\leq y\}}]$. 
The variable $\widehat X_h$ has expected value $\varphi'(h)=x$,  variance $\varphi''(h)$, and probability density function $\e^{-\varphi(h)} \e^{hy} \rho_\Delta(y)$. The approximation $\mathscr L(\widehat X_h) \approx \mathcal N(x, \varphi''(h))$ suggests
$$
	\e^{-\varphi(h)} \e^{hy} \rho_\Delta(y) = \rho_{\widehat X_h}(y) \approx \frac{1}{\sqrt{2\pi\varphi''(h)}}\, \e^{- (y- x)^2/[2 \varphi''(h)]}.
$$
Remember $I(x) = hx - \varphi(h)$ and $I''(x) = 1/\varphi''(h)$, so 
$$
	\rho_\Delta(y) \approx \sqrt{I''(x)} \e^{- I(x)}\times\frac{1}{\sqrt{2\pi}} \e^{- (y- x)^2/[2 \varphi''(h)] - h (y-x)}.
$$

\begin{prop}
As $\Delta \to \infty$, we have for all $x\geq -\Delta$
$$
	\rho_\Delta(x)  =\Bigl( 1+O\Bigl( \frac{1}{\Delta}\Bigr) \Bigr) \sqrt{\frac{I''(x)}{2\pi}}\, \e^{- I(x)}
	=\Bigl( 1+O\Bigl( \frac{1}{\Delta}\Bigr) \Bigr) \frac{\exp(L(x))}{1+x/\Delta}\,\frac{\exp(- x^2/2)}{\sqrt{2\pi}}
$$
with an error term $O(1/\Delta)$ uniform in $x$.   
\end{prop}

\noindent Notice that this approximation is much more precise than the direct normal approximation. 

\begin{proof}
An explicit computation yields
$$
	I'(x)
	= \frac{x}{1+x/\Delta},\quad I''(x) = \frac{1}{(1+ x/\Delta)^2}.
$$
Hence for all $x\geq -\Delta$
$$
	\frac{1}{\sqrt{2\pi}} \sqrt{I''(x)}\e^{- I(x)} = \frac{1}{\sqrt{2\pi}}\frac{1}{(1+ x/\Delta)}
		\e^{- \Delta x} (1+ x/\Delta)^{\Delta^2} = \frac{1}{\sqrt{2\pi}}(1+ x/\Delta)^{\Delta^2-1}\e^{- \Delta x}.
$$
A quick look at ~\eqref{eq:gammadens} reveals that this expression differs from the density $\rho_\Delta(x)$ only through the prefactor, $1/\sqrt{2\pi}$ vs. $\Delta^{2\Delta^2- 1}\e^{-\Delta^2}/\Gamma(\Delta^2)$. The ratio between these two prefactors is independent of $x$ and behaves like $1+O(1/\Delta)$ by Stirling's approximation~\eqref{eq:gammastirling}.
\end{proof}

\subsubsection{Weibull tails} \label{sec:weibull-ex} 

Fix $\gamma>0$. Set $\alpha := 1/(1+\gamma)$ and consider a  non-negative random variable with survival function
$$
	\P(Y \geq y) = \exp( - y^\alpha) \quad (y\geq 0).
$$
The moments of $Y$ are given by
$$
	\mathbb E[Y^m] = \Gamma\bigl(1+ \tfrac{m}{\alpha}\bigr).
$$
Notice that as $m \to\infty$, 
\be \label{eq:Ywei}
	\mathbb E[Y^m]  =(1+o(1)) \sqrt{2\pi m /\alpha} \Bigl(\frac{m}{\alpha \mathrm e}\Bigr)^{m/\alpha} 
	= \frac{m!^{1+\gamma}}{\alpha^{m/\alpha}}\,  \e^{o(m)},
\ee
where we have used Stirling's approximation and $1/\alpha =1+\gamma$. 
Let $Z\sim \mathcal N(0,1)$, independent of $Y$. For small $\eps>0$, set 
$$
	X_\eps := Z+\eps Y. 
$$ 
Then the expected value $\mu_\eps = \E[X_\eps]$ and the variance $\sigma_\eps^2 = \mathbb V(X_\eps)$ satisfy $\mu_\eps = O(\eps)$ and  $\mathbb V[X_\eps]=1+O(\eps^2)$. The centered variable $\widehat X_\eps = (X_\eps - \mu_\eps) / \sigma_\eps$ has cumulants 
\[
	\kappa_j(\widehat X_\eps) = \Bigl(\frac{\eps}{\sigma_\eps}\Bigr)^j \kappa_j(Y)\quad (j\geq 3). 
\] 
In view of~\eqref{eq:Ywei} and Lemma~\ref{lem:conditionsequiv}, it seems plausible that the centered variable $\widehat X_\eps$ satisfies condition~\eqref{condition-sgamma} with $\eps$-dependent $\Delta = \Delta (\eps)$ of the order of 
\[
	\Delta(\eps) \approx \frac{\sigma_\eps}{\eps} = \frac1\eps (1+ O(\eps^2)). 
\]
We would like to understand the behavior of the tails $\P(X_\eps\geq x_\eps) = \P(\widehat X_\eps\geq (x_\eps-\mu_\eps)/\sigma_\eps)$ as $x_\eps \to \infty$ and $\eps \to 0$. Theorem~\ref{thm:lemma23} suggests that the normal approximation $\P(X_\eps \geq x_\eps) \approx \P( Z\geq x_\eps)$ should be good as long as $x_\eps$ is small compared to 
\be \label{eq:crica} 
	\Delta(\eps)^{1/(1+2\gamma)} \approx \eps^{- 1/(1+2\gamma)}. 
\ee
We are not going to provide a precise statement on the normal approximation. Instead we would like to emphasize two key facts. First, the critical scale~\eqref{eq:crica} is explained with a very simple heuristics. Second, the  normal approximation cannot apply beyond that scale. As a consequence, the scale $\Delta^{1/(1+2\gamma)}$ in Theorem~\ref{thm:lemma23} is not due to technical restrictions but is in fact sharp. 

For the heuristic derivation of the critical scale~\eqref{eq:crica}, notice 
\be \label{eq:xepslb}
	\P(X_\eps \geq x_\eps)\geq \P(Z\geq 0,\, \eps Y\geq x_\eps) 
		=\frac12\, \exp\Bigl(- \bigl(\frac{x_\eps}{\eps}\bigr)^\alpha\Bigr)
\ee 
but also, since $Y$ is non-negative,
$$
\P(X_\eps \geq x_\eps) \geq	\P(Z \geq x_\eps)=(1+o(1)) \frac{\exp(-x_\eps^2/2)}{x_\eps \sqrt{2\pi}},
$$
where we have used the well-known asymptotic behavior of Gaussian tails, see Eq.~\eqref{eq:mills} below.
The two lower bounds correspond to two different ways of realizing the unlikely event that $X_\eps = Z+ \eps Y\geq x_\eps$---either $Z$ stays small but $\eps Y$ is very large, or $\eps Y$ stays small but $Z$ is large. Which of the two effects dominates the other? The answer depends on how large $x_\eps$ is. In view of the equivalence 
$$
	\exp\Bigl( -  \frac{x_\eps^2}{2}\Bigr) \geq \exp\Bigl(- \Bigl(\frac{x_\eps}{\eps}\Bigr)^\alpha\Bigr)\, \Leftrightarrow\, x_\eps \leq \Bigl(\frac{2^{1/\alpha}}{\eps}\Bigr)^ {\alpha/(2-\alpha)}
$$
we should expect the probability $\P(X_\eps\geq x_\eps)$ to be similar to $\exp( - x_\eps^2/2)$ when $x_\eps \ll \eps^{- \alpha/(2-\alpha)}$ and similar to $\exp( - (x_\eps/\eps)^\alpha)$ when $x_\eps \gg \eps^{- \alpha/(2-\alpha)}$. Because of $\alpha = 1/(1+\gamma)$, we have 
\[
	\eps^{- \alpha/(2-\alpha)} = \eps^{- 1/(1+2\gamma)} ,
\] 
which is exactly the right-hand side of~\eqref{eq:crica}. Thus we have recovered, heuristically, the critical scale from Theorem~\ref{thm:lemma23}. 

In addition, for $x_\eps \gg \eps^{- \alpha/(2-\alpha)}$, the lower bound~\eqref{eq:xepslb} yields the  rigorous asymptotic lower bound 
\[
	\P( X_\eps \geq x_\eps)\geq \frac12 \exp \Bigl( - \Bigl(\frac{x_\eps}{\eps}\Bigr)^\alpha\Bigr) 
	\gg \P(Z\geq x_\eps). 
\] 
Hence, the normal approximation cannot be good beyond the scale~\eqref{eq:crica}. 

\subsection{ \texorpdfstring{On the Cram{\'e}r, Linnik, and Statulevi{\v{c}}ius conditions}{On the Cramer, Linnik, and Statulevicius conditions}} \label{sec:cralistat}

For $\gamma>0$, the Statulevi{\v{c}}ius condition~\eqref{condition-sgamma} seems technical and not immediately accessible to probabilistic intuition. For $\gamma =0$, the situation is simpler: If $|\kappa_j| \leq j!/\Delta^{j-2}$ for some $\Delta >0$ and all $j\geq 3$, then $\sum_{j\geq 1} \kappa_j t^j /j!$ is absolutely convergent on $(-\Delta, \Delta)$ and $\E[\exp( t|X|)]<\infty$ for all $t\in (-\Delta,\Delta)$. Thus Cram{\'e}r's condition is satisfied and the distribution of $X$ has exponentially decaying tails. 

The question arises if there is a similar intuition for the condition~\eqref{condition-sgamma} when $\gamma>0$. The answer is yes, if we replace Cram{\'e}r's condition by \emph{Linnik's condition}, which reads 
\[
	\E\bigl[ \exp( \delta |X|^\alpha)\bigr] <\infty 
\]  
for some $\alpha\in (0,1)$ and $\delta >0$. The correct choice turns out to be $\alpha = 1/(1+\gamma)$, which should not surprise us after Section~\ref{sec:weibull-ex}. In addition, conditions on cumulants may be replaced by conditions on moments.

\begin{lemma} \label{lem:conditionsequiv} 
	Let $X$ be a random variable with $\E[X]=0$, $\mathbb V(X) =1$, and $\E[|X|^j]<\infty$ for all $j\geq 3$. Fix $\gamma \geq 0$. Then, the following three statements are equivalent: 
	\begin{enumerate} 
		\item[(i)] $X$ satisfies condition~\eqref{condition-sgamma}.
		\item[(ii)] There exists $H\geq 1$ such that the moments of $X$ satisfy 
				\be \label{eq:mgamma} \tag{$\mathcal M_\gamma$}
					\bigl|\E[X^j]\bigr|\leq j!^{1+\gamma} H^{j-2} \quad (j\geq 3).
				\ee	
		\item[(iii)] There exists $\delta>0$ such that 
		$ 
			\E\bigl[ \exp\bigl( \delta |X|^{1/(1+\gamma)}\bigr) \bigr] <\infty. 
		$
	\end{enumerate} 
\end{lemma} 

Similar relations, with explicit control on constants, are proven in \cite[Chapter 3.1]{SS91}. Condition (ii), for $\gamma =0$, is a variant of the condition $|\E[X^j]| \leq \frac 1 2 C j! H^{j-2}$, sometimes called \emph{Bernstein condition} \cite{SS91} because it allows for a Bernstein inequality with unbounded random variables \cite[Chapter 7.5]{ibragimov-linnik}. Linnik's condition is discussed in depth in the context of ``monomial zones of local normal attraction'' in \cite[Chapter 9]{ibragimov-linnik}. The name Linnik condition is not used in the book \cite{ibragimov-linnik}, it is used for example by Saulis and Statulevi{\v{c}}ius \cite{SS91} or Amosova~\cite{amosova99a}.

Large deviation theorems for sums of i.i.d.\ variables under conditions of the type $\E\exp[h(X)]<\infty$ are available as well, see Chapter 11 in Ibragimov and Linnik \cite{ibragimov-linnik} and Nagaev \cite{nagaev1969} for sums of i.i.d.\ variables and Heinrich \cite[Section 4]{heinrich1987survey} for sums of Markov chains. 
However to the best of our knowledge there is no analogue of Lemma~\ref{lem:conditionsequiv} for such more general conditions.

\begin{proof} 
	We start with the equivalence of (ii) and (iii). 
	
 	``$(ii)\Rightarrow (iii)$'' First we note that condition~\eqref{eq:mgamma} implies a similar condition for the moments of $|X|$.  For even powers this is immediate. For odd powers, we use $2j+1 = j + (j+1)$, the Cauchy-Schwarz inequality, and condition~\eqref{eq:mgamma}. This gives 
 	\[
 		\E\bigl[ |X|^{2j+1}\bigr] \leq \Bigl( (2j)!^{1+\gamma} H^{2j-2}\, (2j+2)!^{1+\gamma} H^{2j}\Bigr)^{1/2} = \Bigl(\frac{2j+2}{2j+1}\Bigr)^{(1+\gamma)/2} (2j+1)!^{1+\gamma} H^{2j-1}.
 	\] 
 	The ratio $(2j+2)/(2j+1)$ is not smaller than $1$ but is bounded by $4/3$ (for $j\geq 1$). Therefore, choosing $H' \geq H$ large enough, we get 
 	\[
 		\E\bigl[ |X|^{2j+1}\bigr] \leq  (2j+1)!^{1+\gamma} (H')^{2j+1 - 2}.
 	\]  	
 	We conclude with an argument by Mason and Zhou \cite[Appendix~B]{mason-zhou2012}. Set $\alpha:= 1/(1+\gamma)$. The function $x\mapsto x^\alpha$ is concave on $\R_+$, therefore $\E[Y^\alpha] \leq \E[Y]^\alpha$ for every non-negative random variables $Y$. In particular, for all $j\geq 3$, 
 	\[
 		\E\bigl[ |X|^{j\alpha}\bigr] \leq \Bigl(\E\bigl[|X|^j\bigr]\Bigr)^{\alpha} 
 		\leq j!^{(1+\gamma)\alpha} (H')^{(j-2)\alpha} = j! (H')^{(j-2)\alpha}, 
 	\] 
	which gives 
	\[
		\E\Bigl[\exp( \delta |X|^\alpha)\Bigr]  = 1+ \sum_{j=1}^\infty \frac{\delta^j}{j!}\E\bigl[|X|^{j\alpha}\bigr] <\infty 
	\] 
	for $\delta < 1/H'$. 
	
	The implication ``$(iii)\Rightarrow (ii)$'' is proven by Amosova~\cite[Lemma~3]{amosova99a}, see also Mason and Zhou~\cite[Appendix~B]{mason-zhou2012}. We sketch the argument for the reader's convenience, following~\cite{mason-zhou2012}. Because of $\exp(x) \geq x^k /k!$ for all $x\geq 0$, we have 
	\[
		\E\bigl[ |X|^{k\alpha}\bigr] \leq \frac{k!}{\delta^k}\, \E\bigl[\exp(\delta |X|^\alpha)\bigr] 
		=: \frac{k!}{\delta^k}\, C(\delta).
	\] 
	Given $m\in \N$, define $k = \lceil m/\alpha \rceil = \lceil m(1+\gamma)\rceil$. 
	Thus $\beta:= m/(k\alpha) \leq 1$ and $y\mapsto y^\beta$ is concave. Therefore 
	\[
		\E\bigl[ |X|^m\bigr]\leq \Bigl(\E\bigl[ |X|^{k\alpha}\bigr]\Bigr)^{m/(k\alpha)}. 
	\] 	
	Because of $y^\beta \leq \max (1, y)$ for all $y\geq 0$ and $\beta\in (0,1)$, we conclude 
	\[
		\E\bigl[ |X|^m\bigr]\leq \max \Bigl(1, \frac{k!}{\delta^k}\, C(\delta) \Bigr). 
	\] 	
	The proof is completed by comparing $k!$ with $m!^{1+\gamma}$, aided by Stirling's formula, see \cite[Appendix B]{mason-zhou2012} for details.

	Next we address the equivalence of moment and cumulant conditions. 
	
	``$(ii)\Rightarrow (i)$'' We follow Rudzkis, Saulis, and Statulevi{\v{c}}ius \cite[Lemma 2]{rudzkis-saulis-statulevicius78}.
	Set $m_j:=\E\bigl[ X^j]$. The moment-cumulant relation yields 
	$$
		\frac{\kappa_j}{j!} = [z^j] 	\log \Bigl( 1+ \sum_{k=1}^j\frac{ m_k}{k!} z^k\Bigr)
	$$
	meaning that $\kappa_j / j!$ is equal to the coefficient of $z^j$ in the power series obtained by expanding the logarithm on the right-hand side. Let $r>0$ small enough so that 
	$$
		\sum_{k=1}^j \frac{|m_k|}{k!} r^k < 1.
	$$
	Then $z\mapsto \log \Bigl( 1+ \sum_{k=1}^j \frac1{k!} m_k z^k\Bigr)$ is analytic in $|z|<r$ and continuous in $|z|\leq r$. 	By Cauchy's integral formula, the coefficient can be represented by a contour integral over the contour $|z|=r$ and we find
	$$
		\Bigl| \frac{\kappa_j}{j!} \Bigr| = \Biggl| \frac{1}{2\pi \mathrm i} \oint \log \Bigl( 1+ \sum_{k=1}^j \frac{m_k}{k!} z^k\Bigr) \frac{\dd z}{z^{j+1}}\Biggr| 
			\leq r^{-j} \Biggl| \log  \Bigl( 1 - \sum_{k=1}^j \frac{|m_k|}{k!}\, r^k \Bigr)\Biggr|,
	$$
	where we have used $|\log (1+z)|\leq - \log (1- |z|)$ for $z\in \C$ with $|z|<1$. For $\delta \in (0,1)$ small enough we check that the choice 
	$$
		r:= \frac{\delta}{j!^{\gamma/j}H^{(j-2)/j}} 
	$$
	is admissible. Notice $r^{-j} = \delta^{-j} j!^\gamma H^{j-2}$ resp.~$r^k=\delta^k H^{-k(j-2)/j}(j!)^{-\gamma k/j}$ for $k \leq j$. We have 
	$$
			\sum_{k=1}^j \frac{|m_k|}{k!} r^k \leq \frac12 r^2 + \sum_{k=3}^j  k!^\gamma H^{k-2} r^k = \frac12 r^2 + \sum_{k=3}^j \frac{k!^\gamma}{j!^{\gamma k/j}} \, \frac{H^{k-2}}{H^{k(j-2)/j}} \, \delta^k. 
	$$
	For $k\leq j$ we have $k!^j \leq j!^k$ (this can be proven by induction over $j\geq k$ at fixed $k$). 
	In addition $j(k-2) \leq k(j-2)$, hence  $H^{k-2}\leq H^{k(j-2)/j}$ and 
	$$
			\sum_{k=1}^j \frac{|m_k|}{k!} r^k \leq \frac12 \delta^2 + \sum_{k=3}^j \delta^k  \leq \frac12 \delta^2 + \frac{\delta^3}{1 - \delta}=:C_\delta.
	$$
	Clearly $C_\delta <1$ for small $\delta$. Setting $C'_\delta: = - \log(1-C_\delta)$ we get
	and
	$$
		|\kappa_j| \leq C'_\delta\,  j!^{1+\gamma}\, \Bigl(\frac1\delta H\Bigr)^{j-2}.
	$$
	Set $\Delta:= \frac \delta H \min (1,\frac{1}{C'_\delta} )$, then $|\kappa_j|\leq j!^{1+\gamma} / \Delta^{j-2}$ for all $j\geq 3$. 
	
	``$(i) \Rightarrow (ii)$'' 
	Suppose that $X$ satisfies~\eqref{condition-sgamma}. From the moment-cumulant relations, we get 
	\[
		\frac{m_j}{j!} = [z^j] \exp\Bigl( \sum_{\ell=2}^j \frac{\kappa_\ell}{\ell!} z^\ell\Bigr),
	\] 
	meaning that $m_j$ is equal to the coefficient of $z^j$ on the right-hand side. The Cauchy inequality yields 
	\[
		\frac{|m_j|}{j!}\leq \sup_{r>0} \frac{1}{r^j} \exp\Bigl( \sum_{\ell=2}^j \frac{|\kappa_\ell|}{\ell!} r^\ell\Bigr). 
	\]
	From here on the proof is similar to the proof of the implication $(ii)\Rightarrow (i)$ and therefore omitted. 	
\end{proof} 

\section{Related techniques and applications} 

\subsection{Moderate deviations vs.\ heavy-tailed behavior} \label{sec:heavy-tailed}

Let $(X_n)_{n\in \N}$ be a sequence of normalized real-valued random variables such that each $X_n$ satisfies condition~\eqref{condition-sgamma} for some $n$-dependent $\Delta_n>0$. Further assume that $\Delta_n\to \infty$ as $n\to \infty$. Let us evaluate $\P(X_n\geq x_n)$ for sequences $x_n$ with $x_n\to \infty$ and $x_n=o(\Delta_n^{1/(1+2\gamma)})$. Theorem~\ref{thm:lemma23}, combined with the Gaussian tail estimate~\eqref{eq:mills}, yields
\[
	\P(X_n\geq x_n) = \bigl( 1+o(1)\bigr) \frac{1}{\sqrt {2\pi} \, x_n} \exp\Biggl( - \frac{x_n^2}{2}+ O\Bigl(\frac{x_n^3}{\Delta_n}\Bigr)\Biggr).  
\] 
In general the correction term $O(x_n^3/\Delta_n)$ from the Cram{\'e}r-Petrov series does not go to zero, however for $x_n = o(\Delta_n)$ it is negligible compared to $x_n^2/2$. Hence, if we are only interested in a rough asymptotics on the exponential scale, we may drop it and write 
\[
	\P(X_n \geq x_n) = \exp\Bigl( - \frac{x_n^2}{2} \bigl(1+ o(1)\bigr)\Bigr).
\] 
This asymptotic statement can be lifted to a full moderate deviation principle \cite[Chapter 3.7]{dembo-zeitouni}, where probabilities of more general sets are examined. 

\begin{theorem}\label{thm:moderate}\cite[Theorem 1.1]{DEi3} 
Let $(X_n)_{n\in \N}$ be a sequence of random variables with $\E[ X_n]=0$, $\V(X_n)=1$ that satisfy condition~\eqref{condition-sgamma} with $n$-dependent $\Delta_n$ but fixed $\gamma \geq 0$. 
Suppose that $\Delta_n\to \infty$. Then, for every sequence $a_n\to \infty $ with 
\[
	a_n = o(\Delta_n^{1/(1+2\gamma)}),
\]
the sequence $(X_n/a_n)_{n\in \N}$  satisfies the \emph{moderate deviation principle} with speed $a_n^2$ and rate function $I(x)=x^2/2$. 
\end{theorem}

Thus for every Borel set $A\subset \R$, the lower and upper bounds 
\begin{align*}
\liminf_{n\to\infty}  \frac{1}{a_n^2} \log \P\bigl( \frac{X_n}{a_n}\in A\bigr)
 & \geq - \inf_{x\in \operatorname{int}(A)} \frac{x^2}{2}\\
\limsup_{n\to\infty}  \frac{1}{a_n^2} \log \P\bigl( \frac{X_n}{a_n} \in A\bigr)
 & \leq - \inf_{x\in \operatorname{cl}(A)} \frac{x^2}{2}
\end{align*}
hold, where the infimum and supremum are taken over the interior respectively closure of $A$.

The scale $\Delta_n^{1/(1+2\gamma)}$ is not merely technical. For sums of i.i.d.\ random variables that satisfy Cram{\'e}r's condition or have Weibull tails, the critical scale $\Delta_n^{1/(1+2\gamma)}$ corresponds to the scale at which the tail behavior switches to Cram{\'e}r large deviations or heavy-tailed behavior. Precisely, let $Y_i$, $i\in \N$, be i.i.d.\ random variables with $\E[Y_i] = 0$, $\mathbb V(Y_i) =1$. Set
\[
   S_n := Y_1+\cdots + Y_n, \quad X_n =\frac{1}{\sqrt n} S_n.
\] 
If $\E[\exp( t Y_i)]<\infty$ for $|t|<\delta$, then $X_n$ satisfies condition~\eqref{condition-sgamma} with $\gamma =0$ and $\Delta_n = c \delta\sqrt n$ for some suitable constant $c>0$. Furthermore by the Cram{\'e}r LDP, for $x\geq 0$,  
\begin{align*} 
 	\liminf_{n\to \infty} \frac{1}{\Delta_n^2} \log \P\bigl( \frac{X_n}{\Delta_n} \geq x \bigr) & =\liminf_{n\to \infty} \frac{1}{(c\delta)^2} \frac1 n \log \P( \frac{S_n}{n} \geq x c \delta ) \\
 	&  \geq -  \frac{I( c \delta x)}{(c\delta)^2},\quad I(x) = \sup_{t\in \R}\bigl( t x - \log\E[\e^{t Y_1}]\bigr),
 \end{align*} 
similarly for the upper bound. Thus $(X_n/\Delta_n)_{n\in \N}$ satisfies a large deviation principle with speed $\Delta_n^2$, however the rate function $I(x)$ is in general different from $x^2/2$. 

If $Y_i$ is integer-valued and $\P(Y_i = k) = (1+o(1))c \exp( - k ^\alpha)$ with $\alpha = 1/(1+\gamma) \in (0,1)$, then on the one hand, by Lemma~\ref{lem:conditionsequiv} applied to the $Y_i$'s and $\kappa_j(X_n)= n^{-j/2} \sum_{i=1}^n\kappa_j(Y_i)$, the variable $X_n$ satisfies condition~\eqref{condition-sgamma} with $\Delta_n$ of the order of $\sqrt n$, and on the other hand, for $s_n \gg n^{1/(2-\alpha)}$ 
\[
	\P(S_n = s_n) = n c \exp( - s_n^\alpha + o(s_n^\alpha)\bigr)
\]
as $n\to \infty$, see Nagaev~\cite{nagaev68}.  In the previous asymptotics, the unlikely event that $S_n = s_n$ is best realized by making one out of the $n$ summands very large---this is the typical heavy-tailed behavior \cite{embrechts-klueppelberg-mikosch}, different from the collective behavior underpinning large and moderate deviations.  The scale $s_n^* = n^{1/(2-\alpha)}$ naturally appears when solving for $s_n^\alpha = s_n^2/n$.
As a consequence, for 
\[
	a_n \gg \frac{1}{\sqrt n} n ^{1/(2-\alpha)} = n ^{1 / [2(1+2\gamma)]} = \mathrm{const}\, \Delta_n^{1/(1+2\gamma)}, 
\] 
(meaning $a_n/ \Delta_n^{1/(1+2\gamma)}\to \infty$) and $x>0$, we have 
\[
	\frac{1}{a_n^2} \log \P\bigl( \frac{X_n}{a_n}\geq x\bigr) =	\frac{1}{a_n^2}  \log \P(S_n \geq \sqrt n a_n x)  = (1+o(1))\frac{1}{a_n^2} \bigl(\sqrt n a_n x\bigr)^\alpha \to 0
\] 
which is again different from $- x^2/2$. 

A more involved example illustrating the role of the critical scale for functionals in a fixed Wiener chaos is given by Schulte and Th{\"a}le \cite[Proposition 3]{SchulteThaele:2014}.

More generally, the domain of validity of the moderate deviation principle should be related to the \emph{small steps sequence} for subexponential random variables studied by Denisov, Dieker and Shneer \cite{denisov-dieker-shneer}, in turn related to the sequence $\Lambda(n)$ in Ibragimov and Linnik \cite[Chapter 11]{ibragimov-linnik} and in Nagaev \cite{nagaev1969} $(N_n^{**})$ in \cite[Lemma 2.5]{ercolani-jansen-ueltschi2019}. The small steps sequence in general is smaller than the boundary of the big-jump domain. The latter corresponds to  sequences $b_n$ for which $\P(S_n = b_n) \sim n \P(Y_1 = b_n)$, i.e.\ the dominant effect is having one large summand. 

The connection with heavy-tailed variables suggests that different bounds on cumulants---reflecting the behavior of other heavy-tailed laws, e.g.\ log-normal---might lead to generalizations of Theorem~\ref{thm:moderate}. This corresponds to a generalization of Linnik's condition (see Section~\ref{sec:cralistat}) of the form $\E[\exp( h (X))]<\infty$ with functions $h(x)$ different from $c x^\alpha$. Such generalized Linnik conditions are treated by Ibragimov and Linnik in a chapter on ``narrow zones of normal attraction'' \cite[Chapter 11]{ibragimov-linnik}, the class of functions $h$ and the domains of attraction were further improved by Nagaev \cite{nagaev1969}. However we are not aware of a corresponding generalized Statulevi{\v{c}}ius condition.

\subsection{Mod-phi convergence}\label{sec:modphi}

Let $(X_n)_{n\in \N}$ be a sequence of normalized real-valued random variables such that each $X_n$ satisfies condition~\eqref{condition-sgamma} for some $n$-dependent $\Delta_n>0$ as in the previous subsection. Then we may define a function $R_n$ on the axis of purely imaginary numbers by 
\be \label{eq:mphi-factor}
	\E\bigl[\e^{\mathrm i t X_n}\bigr] = \E\bigl[\e^{\mathrm i t Z}\bigr] R_n(\mathrm i t).  
\ee
If $\gamma =0$, then for $|t|<\Delta_n$ and some $\theta = \theta_n(t)\in [-1,1]$, 
\[
	R_n(\mathrm i t)  = \exp \Bigl(\frac{ \kappa_3(X_n)\Delta_n}{3!} \frac{(\mathrm i t)^3}{\Delta_n}  +  \theta \Delta_n^2 \frac{(t/\Delta_n)^4}{1- |t|/\Delta_n} \Bigr). 
\] 
If $\kappa_3(X_n) \Delta_n$ converges to some constant $c_3\in \R$, then it is natural to rescale variables as $t = \Delta_n^{1/3} s$. In view of $\Delta_n^2 (\Delta_n^{1/3}/\Delta_n)^4 = \Delta_n^{2 - 8/3}\to 0$  we obtain
\[
	\lim_{n\to \infty} R_n \bigl( \mathrm i \Delta_n^{1/3} s \bigr) = \exp\Bigl( \frac{c_3 }{6} (\mathrm i s)^3 \Bigr) 
\] 
uniformly on compact sets. Let us define $\eta(z):= z^2/2$, $\psi(z):= \exp(c_3 z^3/6)$, 
and $Y_n: = \Delta_n^{1/3} X_n$, then 
\be \label{eq:modphi}
	\lim_{n\to \infty} \exp\bigl( - \Delta_n^{2/3} \eta(\mathrm i s)\bigr)\, 	\E\bigl[\e^{\mathrm i s Y_n}\bigr] =\psi(\mathrm is)
\ee
uniformly on compact sets. The convergence~\eqref{eq:modphi} is a key ingredient to the notion of \emph{mod-phi convergence}, here mod-Gaussian convergence of the sequence $(Y_n)_{n\in \N}$ with speed $\Delta_n^{2/3}$ and limiting function $\psi(z)$.  Different full definitions of  mod-phi convergence, given functions $\eta$, $\psi$, and a speed sequence, impose slightly different additional conditions  \cite{JacodModPhi, DelbaenModPhi, FMNbook}. For example, one may or may not impose that the convergence~\eqref{eq:modphi} extends from the purely imaginary axis to a complex strip (which would require $\gamma =0$ in the Statulevi{\v{c}}ius condition), that $\eta(z)$ is the L{\'e}vy exponent of some infinitely divisible law, or that $\psi(z)$ is non-zero on a complex strip. 

Mod-Gaussian convergence was introduced by Jacod, Kowalski, and Nikeghbali~\cite{JacodModPhi}. The original motivation was in random matrix theory and analytic number theory. Concretely, a result by Keating and Snaith~\cite{keating-snaith2000} as summarized in \cite{JacodModPhi} says that the determinant $Z_N$ of a random matrix in $U(N)$, distributed according to the uniform measure (Haar measure) on $U(N)$, satisfies for all $\lambda \in \C$ with $\Re \lambda >-1$, 
\be \label{eq:modphi2}
	\lim_{N\to \infty}\frac{1}{N^{\lambda^2}\,}\E\bigl[ |Z_N|^{2\lambda}\bigr] = \frac{(G(1+\lambda))^2}{G(1+2\lambda)},
\ee
with $G$ some special function. Eq.~\eqref{eq:modphi2}  is clearly in the spirit of~\eqref{eq:modphi}---set $\lambda= \mathrm i s$, $Y_N = \log|Z_N|^2$, and make adequate choices of the speed and limiting function. Subsequent developments include mod-Poisson convergence for random combinatorial structures~\cite{barbour-kowalski-nikeghbali2014}, random vectors, proving mod-Gaussian convergence with dependency graphs, and a systematic study of asymptotic statements on probabilities when mod-phi convergence holds true, see the monograph \cite{FMNbook}. The asymptotic bounds and their proofs share some similarities with the bounds on which we focus in this survey, see the discussion~\cite[Section~5.3]{FMNbook}. 

\subsection{Analytic combinatorics. Singularity analysis} 

Discrete probability and analytic combinatorics~\cite{flajolet-sedgewick-book} share a common complex-analytic toolbox. It is often of advantage to work with generating functions $G(z)= \sum_{n=0}^\infty g_n z^n$. In discrete probability, the coefficient $g_n$ represent a probability measure on $\N_0$. In analytic combinatorics, the coefficients are instead related to counting problems, e.g.\ counting the number of trees on $n$ vertices. When the generating function is well-understood, probabilities or cardinalities can be recovered by complex contour integrals, using Cauchy's formula. These formulas are similar to inversion formulas that express a probability density function or cumulative distribution function in terms of the characteristic function (Fourier transform). 

Understanding probabilities then boils down to understanding parameter-dependent contour integrals, for which a plethora of methods are available, for example saddle-point and steepest descent methods, and singularity analysis \cite{flajolet-sedgewick-book}. \emph{Singularity} refers to the singularities of the function $G(z)$ in the complex plane,  
among which the dominant singularity $z=R$, the radius of convergence of $G(z)$. \emph{Transfer theorems} go from asymptotic expansions of $G(z)$ near its dominant singularity to asymptotic behavior of the coefficients as $n\to \infty$, a process related to Tauberian theorems for inverse Laplace transforms \cite[Chapter VI]{flajolet-sedgewick-book}. 

Cumulants fit in very naturally:  for an integer-valued, heavy-tailed random variable $X$, the dominant singularity of the probability generating function $G(z) = \E[z^X]$ is at $z=1$, and 
\[
	G(\e^{t}) =  \exp\Bigl(\sum_{j\geq 1}\frac{\kappa_j}{j!}\, t^j \Bigr)\quad (z = \e^t \to 1, \, t\to 0).
\] 
However, even though singularity analysis deals with Taylor expansions with zero radius of convergence, Weibull-like variables do not belong to the class amenable to singularity analysis \cite[Chapter VI.6]{flajolet-sedgewick-book} and therefore the methods described in the book by Flajolet and Sedgewick \cite{flajolet-sedgewick-book} are not directly applicable (see nevertheless \cite{ercolani-jansen-ueltschi2019} and the references therein).

The methods extend to multivariate generating functions \cite{pemantle-wilson} and to sequences of generating functions. The latter enter the stage naturally when working with sequences of random variables, and allow for a derivation of limit laws in random combinatorial structures. The simplest setting is when generating functions can be approximated by powers of simpler generating functions (think of independent random variables!), leading to the framework of \emph{quasi-powers}, see Hwang~ \cite{hwang1998} and \cite[Chapter IX.5]{flajolet-sedgewick-book}. The relation between Hwang's quasi-powers and mod-phi convergence is commented upon in \cite[Remark 1.2]{FMNbook}. 

\subsection{Dependency graphs}  \label{sec:dependency}
Let $Y_\alpha$, $\alpha\in I$, be real-valued random variables indexed by some set $I$ of cardinality $N\in \N$. The \emph{mixed cumulants} of the $Y_\alpha$'s are given by an inverse M{\"o}bius transform of mixed moments as
\be \label{eq:mixed-cumulants}
	\kappa(Y_{\alpha_1},\ldots, Y_{\alpha_r})= \sum_{\{B_1,\ldots, B_m\}} (-1)^{m-1} (m-1)! \prod_{\ell=1}^m \E\Bigl[ \prod_{i \in B_\ell} Y_{\alpha_\ell} \Bigr]
\ee
with summation over set partitions $\{B_1,\ldots, B_m\}$ of $\{1,\ldots, r\}$ of variable number of blocks $m\in \{1,\ldots, r\}$. A classical reference for~\eqref{eq:mixed-cumulants} is Leonov and Shiryaev \cite{LeonovS1959}, an early mention of M{\"o}bius inversion is found in Sch{\"u}tzenberger \cite{schuetzenberger1954}; a detailed historical discussion is given by Speicher\footnote{Blog entry from 2 July 2020: \url{https://rolandspeicher.com/tag/moment-cumulant-formula/}. Last consulted on 1 February 2021.}. Mixed cumulants of independent variables vanish. The cumulants of $X: = \sum_{\alpha \in I} Y_\alpha$  are 
\[
	\kappa_r(X) = \sum_{(\alpha_1,\ldots,\alpha_r)\in I^r}\kappa(Y_{\alpha_1},\ldots, Y_{\alpha_r}).
\] 
Now assume that the $Y_\alpha$'s have a dependency structure encoded by a \emph{dependency graph} $G$. The latter is a graph $G= (I, E(G))$ with vertex set $I$ with the following property: if $I_1$ and $I_2$ are disjoint subsets of $I$ not linked by an edge $\{\alpha,\beta\}$ in $G$, then $Y_\alpha, \alpha \in I_1$ and $Y_\beta, \beta\in I_2$ are independent. F{\'e}ray, M{\'e}liot, and Nikeghbali prove a beautiful tree bound for cumulants. 

\begin{lemma}\label{lem:treebound}  \cite[Section 9.3]{FMNbook}
	Suppose that $G= (I,E(G))$ is a dependency graph for $Y_\alpha, \alpha \in I$. Assume in addition that $|Y_\alpha|\leq A$ almost surely. For $r\in \N$, let $\mathcal T_r$ be the set of tree graphs with vertex set $\{1,\ldots, r\}$. Then for all $\alpha_1,\ldots, \alpha_r\in I$, 
	\be \label{eq:cumulant-tree-bound}
		\bigl| \kappa (Y_{\alpha_1},\ldots, Y_{\alpha_r})\bigr|\leq 2 ^{r-1} A^{r} \sum_{T\in \mathcal T_r} \prod_{\{i,j\}\in E(T)}\bigl( \1_{\{\alpha_i = \alpha_j\}}+ \1_{\{\alpha_i \neq \alpha_j,\, \{\alpha_i,\alpha_j\}\in E(G)\}}\bigr). 
	\ee
\end{lemma} 

The sum over trees is equal to the number of spanning trees of the graph $H$ with vertex set $\{1,\ldots, r\}$ for which $\{i,j\}$ is an edge if and only if either $\alpha_i = \alpha_j$ or $\alpha_i \neq \alpha_j$ and $\{\alpha_i,\alpha_j\}$ is an edge of the dependency graph. 

When the dependency graph has maximum degree $D$, it is easily checked that for each fixed tree $T\in \mathcal T_r$ and all $\alpha_1\in I$, 
\[
	\sum_{\alpha_2,\ldots, \alpha_r\in I}  \prod_{\{i,j\}\in E(T)}\bigl( \1_{\{\alpha_i = \alpha_j\}}+ \1_{\{\alpha_i \neq \alpha_j,\, \{\alpha_i,\alpha_j\}\in E(G)\}}\bigr) \leq (D+1)^{r-1}. 
\] 
Summing over $\alpha_1$ gives an additional factor $N= \#I$. Combining with Cayley's formula $\#\mathcal T_r = r^{r-2}$, one finds 
\be \label{eq:maxdeg-bound}
	\bigl|\kappa_r(X)\bigr| \leq r^{r-2}\, 2 N A  \bigl(2 A (D+1)\bigr)^{r-1},
\ee
see \cite[Theorem 9.8]{FMNbook}. Stirling's formula implies that the sum $X = \sum_\alpha Y_\alpha$, suitably normalized, satisfies the Statulevi{\v{c}}ius condition~\eqref{condition-sgamma} with $\gamma =0$. 

Applications of the bound~\eqref{eq:maxdeg-bound} are found in \cite{FMNbook} and~\cite{DEi3}. We make two additional remarks. The first remark concerns the appearance of trees. Tree bounds as in Lemma~\ref{lem:treebound}---with functions $u(\alpha,\beta)\geq 0$ rather than indicators that $\{\alpha,\beta\}$ is in some dependency graph---come up naturally in random fields and statistical mechanics, see for instance Duneau, Iagolnitzer, and Souillard \cite{duneau-iagolnitzer-souillard1973}. The existence of tree bounds is sometimes called \emph{strong mixing}. Tree bounds feature prominently in the framework of \emph{complete analyticity} for Gibbs measures, see condition IIb in Dobrushin and Shlosman \cite{dobrushin-shlosman1987}. The bound~\eqref{eq:maxdeg-bound} extends to such soft tree bounds because
\[
	\sum_{\alpha_1,\ldots,\alpha_r\in I}\, \sum_{T\in \mathcal T_r}  \prod_{\{i,j\} \in E(\mathcal T_r)} u(\alpha_i,\alpha_j) \leq N r^{r-2} \Bigl( \max_{\alpha \in I} \sum_{\beta \in I} u(\alpha,\beta)\Bigr)^{r-1}. 
\] 
Similar considerations appear with \emph{weighted dependency graphs} introduced by F{\'e}ray \cite{Feray:2018} and in particular \emph{uniform weighted dependency graphs} \cite[Definition 43]{feray-meliot-nikeghbali2019}. Weighted dependency graphs have been applied, for example, to the Ising model \cite{DousseF:2019}, and to many other examples (not covered by statistical mechanics) \cite{Feray:2018}.

Our second remark is of a speculative nature: Lemma~\ref{lem:treebound}, and its proof in \cite{FMNbook} is intriguing because of many similarities with the theory of cluster expansions. In fact the proof of Lemma~\ref{lem:treebound} brings up, for a connected graph $H$, the quantity $\sum_{G\subset H} (-1)^{\#E(G)- 1}$ (summation over spanning subgraphs) \cite[Lemma 9.12]{FMNbook}. This object is centerstage in the theory of cluster expansions and bounding it by trees is fairly standard, see Scott and Sokal \cite{scott-sokal2005} and the references therein. Connections between cluster expansions, dependency graphs, and combinatorics have been studied intensely in the context of the Lov{\'a}sz local lemma \cite{scott-sokal2005}. The combinatorial proofs by F{\'e}ray, M{\'e}liot, and Nikeghbali open up the intriguing perspective of yet another fruitful connection between cluster expansions and (weighted) dependency graphs.

\section{Toolbox}

Here we collect a few general lemmas that are of independent interest. 

\subsection{Characteristic functions, Kolmogorov distance, and smoothing inequality} 

One key ingredient in the proof is a bound on the Kolmogorov distance of two measures in terms of an integral involving the characteristic functions. Estimates of this type are fairly classical and enter proofs of the Berry-Esseen inequality following Berry's strategy~\cite{berry41}, see for example~\cite[Chapter XVI]{feller-vol2}, \cite[Chapter 5.1]{petrov-book} and the survey on smoothing inequalities by Bobkov~\cite{bobkov2016}. For asymptotic expansions, e.g.~Edgeworth expansions that capture correction terms in normal approximations, it is customary to deal not only with probability measures but also with signed measures whose density is a Gaussian multiplied by a  polynomial~\cite[Chapter XVI.4]{feller-vol2}. 

 If $\mu$ is  a finite signed measure on $\R$, write $\mu=\mu_+-\mu_-$ for the Jordan decomposition of $\mu$. The cumulative distribution function is $F_\mu(x)=\mu((-\infty,x])$ and the characteristic function is $\chi_\mu(t)=\int_\R \exp(\mathrm i t x) \mu(\dd x)$.

\begin{lemma} \label{lem:zolotarev2}
	Let $\mu$ and $\nu$ be two finite signed measures on $\R$ with total mass $1$. Let $Y$ be an auxiliary continuous random variable whose probability density function $\rho_Y$ is even, i.e., $\rho_Y(y) = \rho_Y(-y)$ for all $y\in \R$. 
	 Assume that $\mu$ has a Radon-Nikodym derivative with respect to Lebesgue measure bounded in absolute value by $q>0$, and that the negative part of $\nu$, if non-zero, satisfies $\nu_-(\R) \leq \eta$. Then 
	\begin{multline*}
	\sup_{x\in \R} |F_\mu(x) - F_\nu(x) |\\
			\leq \frac{1}{1- 2 \P(|Y|\geq y_0)} \Bigl(\bigl[ \eps q y_0+\eta\bigr] \, \P(|Y|\leq y_0)+ \frac{1}{2\pi} \int_{-\infty}^\infty \bigl|\chi_{\eps Y}(t)\bigr|\, \bigl|\chi_\mu(t) - \chi_\nu(t)\bigr|\frac{\dd t}{|t|}\Bigr),
	\end{multline*} 
	for all $\eps\geq 0$ and every $y_0>0$ with $\P(|Y|\geq y_0) < \frac12$. 
\end{lemma}

\begin{remark}
	If $\nu$ is a probability measure, i.e., $\nu_-=0$ and $\eta=0$, then Lemma~\ref{lem:zolotarev2} is due to Zolotarev~\cite{zolotarev65} as cited in ~\cite[Lemma 2.5]{SS91}. Our extension to signed measures $\nu$ is needed to fix an erroneous application of Zolotarev's lemma to the normal law $\mu \sim\mathcal N(0,1)$ and a signed measure $\nu$ that is not necessarily absolutely continuous. Another extension to signed measures is found in~\cite[Theorem 2]{zolotarev67}, under the condition that the atoms of the signed measures form a discrete subset of $\R$. 
\end{remark}

\noindent Lemma~\ref{lem:zolotarev2} is applied to a normal law $\mu =\mathcal N(0,1)$---thus $q=1/\sqrt{2\pi}$---and the random variable $Y$ with 
probability density function and characteristic  function given by
\be \label{eq:tentvariable} 
	\rho_Y(y) = \frac{1-\cos y}{\pi y^2}, \quad \chi_Y(t ) = (1-|t|)\1_{[-1,1]}(t). 
\ee
The choice $\rho_Y$ of smoothing density was already made by Berry~\cite{berry41}. Write $\eps = 1/T$, then the integral error term  becomes 
\be \label{eq:zolotent}
	\frac{1}{2\pi} \int_{-\infty}^\infty \bigl|\chi_{\eps Y}(t)\bigr|\, \bigl|\chi_\mu(t) - \chi_\nu(t)\bigr|\frac{\dd t}{|t|}
		 = \frac{1}{2\pi}\int_{-T}^{T} \Bigl(1- \frac{|t|}{T}\Bigr) |\chi_X(t) - \chi_Z(t)|\frac{\dd t}{|t|}.
\ee
In the proof of Theorem~\ref{thm:lemma22} we choose $T= 1/\eps$ of the order of $\sqrt{s}$, see Section~\ref{sec:normaltilt}. We follow~\cite{SS91} and choose $y_0=3.55$. A numerical evaluation yields 
$\P(|Y|\leq y_0) \approx 0.819717$ and
\be \label{eq:numtent}
	C_1:= \frac{y_0 \P(|Y|\leq y_0)}{1- 2 \P(|Y|\geq y_0)} \simeq 4.5509,\quad 
	C_2:= \frac{1}{1- 2 \P(|Y|\geq y_0)} \simeq 1.5639.
\ee
The numerical values are better than the values appearing in the smoothing inequality in~\cite[Chapter XVI.3, Lemma~1]{feller-vol2}. Assume that $\nu$ is a probability measure and write $J_\eps$ for the integral term. Then 
\be \label{eq:smoothing-feller}
	\sup_{x\in \R} |F_\mu(x) - F_\nu(x) | \leq  \frac{24}{\pi} \, q \eps  + 2 J_\eps.
\ee
We note $24/\pi \simeq 7.34 > 4.5509$ and $2> 1.5639$, hence Lemma~\ref{lem:zolotarev2} with $y_0=3.55$ yields better constants than Lemma~1 in~\cite[Chapter XVI.3]{feller-vol2}.

 The proof of Lemma~\ref{lem:zolotarev2} is based on classical inversion formulas.  Recall that if the characteristic function $\chi$ of a random variable $X$ is integrable, then the variable has a 
probability density function given by 
$$
	\rho(x) = \frac{1}{2\pi} \int_{-\infty}^\infty \e^{- \mathrm i t x} \chi(t) \dd t
$$
and the  cumulative distribution function $F$ is given by 
\begin{align*}
	F(x) & = \lim_{a\to - \infty} \int_a^x \rho(y) \dd y = \lim_{a\to -\infty} \frac{1}{2\pi} \int_{-\infty}^\infty \Bigl( \int_a^x \e^{- \mathrm i t y} \dd y\Bigr) \chi(t) \dd t \\
		& =  \lim_{a\to - \infty} \frac{1}{2\pi} \int_{-\infty}^\infty  \frac{\e^{-\mathrm i t x} - \e^{\mathrm i t a}}{- \mathrm i t }\, \chi(t) \dd t.
\end{align*} 
For general random variables, the previous formula for $F(x)$ holds true in every point $x$ of continuity of $F$. 

The proof of Lemma~\ref{lem:zolotarev2} is adapted from the proof of Theorem~2 in~\cite{zolotarev67}, see also Lemma~1 in~\cite[Chapter XVI.3]{feller-vol2}.  To help  the reader grasp the probabilistic content, we first prove the lemma when $\mu$ and $\nu$ are probability measures on $\R$. 

\begin{proof} [Proof of Lemma~\ref{lem:zolotarev2} when $\mu$ and $\nu$ are probability measures]
Let $X$ and $Z$ be two random variables with respective distributions $\mu$ and $\nu$. 
We may assume without loss of generality that $X,Y,Z$ are defined on a common probability space $(\Omega,\mathcal F, \P)$ and that $Y$ is independent from $X$ and $Z$. We have, in every point of continuity $x$ of $F_{X+\eps Y}$ and $F_{Z+\eps Y}$, 
$$
	F_{X+\eps Y}(x) - F_{Z+\eps Y}(x) =  \lim_{a\to - \infty} \frac{1}{2\pi} \int_{-\infty}^\infty  \frac{\e^{-\mathrm i t x} - \e^{\mathrm i t a}}{- \mathrm i t }\, \chi_Y(\eps t) \bigl(\chi_X(t) - \chi_Z(t)\bigr) \dd t.
$$
If $\int_\R |\frac1t\chi_Y(\eps t) (\chi_X (t) - \chi_Z(t))| \dd t = \infty$, the lemma is trivial, so we only need to treat the case where $t\mapsto\frac1t\chi_Y(\eps t) (\chi_X (t) - \chi_Z(t))$ is integrable. The Riemann-Lebesgue lemma then yields the simplified expression 
$$
	F_{X+\eps Y}(x) - F_{Z+\eps Y}(x) =   \frac{1}{2\pi} \int_{-\infty}^\infty  \e^{-\mathrm i t x} \, \chi_Y(\eps t) \bigl(\chi_X(t) - \chi_Z(t)\bigr) \frac{\dd t}{- \mathrm i t},
$$
from which we deduce the bound 
\be \label{eq:mollifiederror}
	\sup_{x\in \R}\bigl|F_{X+\eps Y}(x) - F_{Z+\eps Y}(x) \bigr|\leq   \frac{1}{2\pi} \int_{-\infty}^\infty \bigl| \chi_Y(\eps t)\bigr|\, \bigl|\chi_X(t) - \chi_Z(t)\bigr| \frac{\dd t}{|t|} =:J_\eps.
\ee
This proves the lemma in the case $\eps =0$. 
For $\eps >0$, we need to bound the Kolmogorov distance of $X$ and $Z$ by that of $X+\eps Y$ and $Z+\eps Y$. The relevant inequalities are called  \emph{smoothing inequalities} \cite{feller-vol2,bobkov2016}. We condition on values of $Y$ and distinguish cases according to $|Y|\geq y_0$ or $|Y|\leq y_0$. We start with $|Y|\leq y_0$. We wish to exploit $\sup_\R |F'_X(x)|\leq q$ and the monotonicity of $F_Z$. Notice that, for every $y'\in [-y_0,y_0]$, 
\be \label{eq:keymonotone}
\begin{aligned}
	F_Z\bigl(x_0 + \eps (y_0-y')\bigr) - F_X\bigl(x_0 + \eps(y_0-y')\bigr)
		& \geq F_Z(x_0) - F_X(x_0) - q\eps(y_0-y'), \\
	F_Z\bigl(x_0 - \eps (y_0+y')\bigr) - F_X\bigl(x_0 - \eps(y_0+y')\bigr)
		& \leq F_Z(x_0) - F_X(x_0) + q\eps(y_0+y').
\end{aligned} 
\ee
Because of the independence of $Y$ from $X$ and from $Z$, we can reinterpret the inequalities as almost sure inequalities conditioned on $Y=y'$. The second inequality yields 
\begin{multline*}
	\1_{\{|Y|\leq y_0\}}\E\Bigl[\bigl(\1_{\{Z+\eps Y\leq x_0- \eps y_0\}} - \1_{\{X +\eps Y\leq x_0- \eps y_0\}} \bigr)\, \Big|\, Y \Bigr] \\
	\leq \1_{\{|Y|\leq y_0\}} \Bigl(F_Z(x_0) - F_X(x_0) + q\eps y_0 + q\eps Y\Bigr) \quad \text{a.s.}
\end{multline*}
We take expectations on both sides, use $\E[Y\1_{\{|Y|\leq y\}}]=0$ from the parity of $Y$, and deduce 
\begin{multline*}
	\P(Z + \eps Y \leq x_0-\eps y_0, |Y|\leq y_0) - \P(X+\eps Y \leq x_0-\eps y_0, |Y|\leq y_0)\\
		\leq \Bigl(F_Z(x_0) - F_X(x_0) + q\eps y_0\Bigr) \P(|Y|\leq y_0).
\end{multline*}
This implies
\begin{align}
	F_Z(x_0) - F_X(x_0)
	&\geq 
	\frac{\P(Z + \eps Y \leq x_0-\eps y_0, |Y|\leq y_0) - \P(X+\eps Y \leq x_0-\eps y_0, |Y|\leq y_0)}{\P(|Y|\leq y_0)} -q\eps y_0
	\nonumber\\
	&\geq  -\frac{\sup_{x\in \R}\bigl|F_{X+\eps Y}(x) - F_{Z+\eps Y}(x) \bigr|}{\P(|Y|\leq y_0)} -q\eps y_0.
	\label{zol1}
\end{align}
By the same arguments the first inequality in \eqref{eq:keymonotone} yields
\begin{multline*}
	\P(Z + \eps Y \leq x_0+\eps y_0, |Y|\leq y_0) - \P(X+\eps Y \leq x_0+\eps y_0, |Y|\leq y_0)\\
		\geq \Bigl(F_Z(x_0) - F_X(x_0) - q\eps y_0\Bigr) \P(|Y|\leq y_0)
\end{multline*}
and
\be
	F_Z(x_0) - F_X(x_0)
	\leq 
	\frac{\P(Z + \eps Y \leq x_0+\eps y_0, |Y|\leq y_0) - \P(X+\eps Y \leq x_0+\eps y_0, |Y|\leq y_0)}{\P(|Y|\leq y_0)} +q\eps y_0.
	\label{zol2}
\ee
With \eqref{zol1} and \eqref{zol2} we get
\be  \label{eq:mollicomp}
	\bigl|F_Z(x_0) - F_X(x_0)|
	\leq \frac{\sup_{x\in \R}\bigl|F_{X+\eps Y}(x) - F_{Z+\eps Y}(x) \bigr|}{\P(|Y|\leq y_0)} +q\eps y_0. 
\ee 
If $|Y|\leq y_0$ almost surely, the lemma follows by first using~\eqref{eq:mollifiederror} and then taking the sup over $x_0$. 

If $|Y|$ is larger than $y_0$ with positive probability, we need an additional estimate. From the tower property of conditional expectations and the independence of $Y$ from $X$ and from $Z$ we get 
\begin{align*} 
	& \bigl|\P(Z+ \eps Y \leq x_0, |Y|\geq y_0) - \P(X+\eps Y \leq x_0, |Y|\geq y_0)\bigr| \\
	& \quad = \Bigl| \E\Bigl[ \1_{\{|Y|\geq y_0\}} \E\bigl[ \1_{\{Z \leq x_0 - \eps Y\}} -  \1_{\{X \leq x_0 - \eps Y\}} \mid Y\bigr] \Bigr] \Bigr|\\
	&\quad =\Bigl| \E\Bigl[\1_{\{|Y|\geq y_0\}} (F_Z(x_0 - \eps Y) - F_X(x_0-\eps Y)\Bigr]\Bigr| \\
	&\quad\leq D \P(|Y|\geq y_0),
\end{align*}
where we used 
\begin{equation*}
D:= \sup_{x \in \mathbb R}|F_Z(x)-F_X(x)|.
\end{equation*}
The triangle inequality thus shows 
\begin{multline*}
	\bigl|	\P(Z + \eps Y \leq x_0, |Y|\leq y_0) - \P(X+\eps Y \leq x_0, |Y|\leq y_0)\bigr| \\
	\leq \bigl|	\P(Z + \eps Y \leq x_0) - \P(X+\eps Y \leq x_0)\bigr| + D \P(|Y|\geq y_0).
\end{multline*} 
In combination with~\eqref{eq:mollicomp} and~\eqref{eq:mollifiederror}, this yields 
$$
	\bigl|F_Z(x_0) - F_X(x_0)|\  \P(|Y|\leq y_0) \leq J_\eps + D \P(|Y|\geq y_0) + q\eps y_0 \P(|Y|\leq y_0).
$$
We take the sup over $x_0\in \R$, use $\P(|Y|\leq y_0) - \P(|Y|\geq y_0) = 1- 2\P(|Y|\geq y_0)>0$ by the assumption on $y_0$, and obtain 
\begin{equation*}
	D\leq \frac{q \eps y_0 \P(|Y|\leq y_0) + J_\eps}{1 - 2 \P(|Y|>y_0)}.  
\end{equation*}
This concludes the proof of the lemma (remember~\eqref{eq:mollifiederror}).
\end{proof}

\begin{proof}[Proof of Lemma~\ref{lem:zolotarev2} for signed measures $\mu$ and $\nu$]
Let 
\be \label{eq:signedD}
	D:= \sup_{x\in \R} \bigl| F_\mu(x) - F_\nu(x)\bigr|
\ee
and 
\be \label{eq:deps}
	D_\eps:= \sup_{x\in \R}   \Bigl|   \nu*P_{\eps Y} \bigl( (- \infty, x ]\bigr) -   \mu*P_{\eps Y} \bigl( (- \infty, x]\bigr)\Bigr|, 
\ee
where $P_{\eps Y}$ is the law of $\eps Y$ and $\mu*P_{\eps Y}$ the convolution of $\mu$ and $P_{\eps Y}$.  Arguments similar to the proof of Eq.~\eqref{eq:mollifiederror} yield
\be \label{eq:mollifiederror2}
	D_\eps \leq \frac{1}{2\pi} \int_{-\infty}^\infty \bigl| \chi_Y(\eps t)\bigr|\, \bigl|\chi_\mu(t) - \chi_\nu(t)\bigr| \frac{\dd t}{|t|}.
\ee
It remains  to bound $D$ in terms of $D_\eps$. For signed measures the cumulative distribution function is no longer monotone increasing however 
$$
	F_\nu(x_0+u) - F_\nu(x_0) =\nu((x_0,x_0+u])\geq - \nu_-(\R)=- \eta
$$
for all $x_0\in \R$ and $u\geq 0$. Therefore~\eqref{eq:keymonotone} becomes 
\be \label{eq:keymonotone2}
\begin{aligned}
	F_\nu\bigl(x_0 + \eps (y_0-y')\bigr) - F_\mu\bigl(x_0 + \eps(y_0-y')\bigr)
		& \geq F_\nu(x_0) - F_\mu(x_0) - q\eps(y_0-y') - \eta, \\
	F_\nu\bigl(x_0 - \eps (y_0+y')\bigr) - F_\mu\bigl(x_0 - \eps(y_0+y')\bigr)
		& \leq F_\nu(x_0)  - F_\mu(x_0) + q\eps(y_0+y')+ \eta
\end{aligned} 
\ee
for all $y'\in [-y_0,y_0]$. 
We integrate over $y'$ with respect to the law $P_Y$ of $Y$, use $\int_{-y}^{y} y' P_Y(\dd y') =0$ because of the parity of $Y$, and obtain 
\be \label{eq:keymon3}
\begin{aligned} 
	&\int_{-y_0}^{y_0}\Bigl( F_\nu \bigl(x_0 + \eps (y_0-y')\bigr) - F_\mu\bigl(x_0 + \eps(y_0-y')\bigr)\Bigr) P_Y(\dd y') \\
	&\qquad \qquad \qquad \geq \Bigl( F_\nu(x_0) - F_\mu(x_0) - [q \eps y_0+\eta]\Bigr) \P(|Y|\leq y_0),\\
	&\int_{-y_0}^{y_0}\Bigl( F_\nu \bigl(x_0 - \eps (y_0+y')\bigr) - F_\mu\bigl(x_0 - \eps(y_0+y')\bigr)\Bigr) P_Y(\dd y') \\
	&\qquad \qquad \qquad \leq \Bigl( F_\nu(x_0) - F_\mu(x_0) + q \eps y_0+ \eta \Bigr) \P(|Y|\leq y_0),
\end{aligned}
\ee
which replaces~\eqref{zol1} and~\eqref{zol2}. The left-hand sides can be rewritten with the help of convolutions. Let $P_{\eps Y}$ be the distribution of $\eps Y$, then 
\begin{align*}
	\int_{-\infty}^\infty F_\mu \bigl(x_0 + \eps (y_0-y')\bigr)P_Y(\dd y') 
		& = \int_\R\Biggl(\int_\R \1_{(-\infty,x_0 + \eps y_0]}(x+ \eps y')  \mu(\dd x)  \Biggr) P_{Y}(\dd y') \\
		& = \bigl( \mu*P_{\eps Y}\bigr) \bigl( (- \infty, x_0 + \eps y_0]\bigr).
\end{align*} 
A similar identity holds true with $F_\nu$ and $\nu$ instead of $F_\mu$ and $\mu$. 
For the integral over $\R\setminus [-y_0,y_0]$, we note 
$$
	\Bigl|\int_{\R\setminus [-y_0,y_0]}\Bigl( F_\mu\bigl(x_0 + \eps (y_0-y')\bigr) - F_\nu \bigl(x_0 + \eps(y_0-y')\bigr)\Bigr) P_Y(\dd y') \Bigr| \leq D\, \P(|Y|>y_0).
$$
We write the integral over $[-y_0,y_0]$ as the difference of the integral over $\R$ and $\R \setminus [-y_0,y_0]$, apply the triangle inequality, and deduce 
\begin{equation*}
	\int_{-y_0}^{y_0}\Bigl( F_\nu\bigl(x_0 + \eps (y_0-y')\bigr) - F_\mu\bigl(x_0 + \eps(y_0-y')\bigr)\Bigr) P_Y(\dd y')	 \leq  D_\eps + D\, \P(|Y|>y_0)
\end{equation*} 
with $D_\eps$ given by~\eqref{eq:deps}. Similarly,
\begin{equation*}
	\int_{-y_0}^{y_0}\Bigl( F_\nu\bigl(x_0 -  \eps (y_0+y')\bigr) - F_\mu\bigl(x_0 + \eps(y_0+y')\bigr)\Bigr) P_Y(\dd y')	 \geq -   D_\eps - D\, \P(|Y|>y_0).
\end{equation*} 
Combining this with~\eqref{eq:keymon3}, we find 
\be \label{eq:keymon4}
\begin{aligned} 
	D_\eps + D\, \P(|Y|>y_0)  &\geq \Bigl( F_\nu(x_0) - F_\mu(x_0) - \bigl[q \eps y_0+\eta ]\Bigr) \P(|Y|\leq y_0),\\
 	 - D_\eps - D\, \P(|Y|>y_0)  &\leq \Bigl( F_\nu (x_0) - F_\mu(x_0) + q \eps y_0+\eta \Bigr) \P(|Y|\leq y_0)
\end{aligned}
\ee
hence 
$$
	\P(|Y|\leq y_0)\, |F_\nu(x_0) - F_\mu(x_0) |\leq D\, \P(|Y|>y_0) + D_\eps + ( q \eps y_0+\eta) \, \P(|Y|\leq y_0).
$$
We take the supremum over $x_0\in \R$, obtain an inequality with $D$ on both sides from which we deduce 
$$
 	D \leq \frac{1}{\P(|Y|\leq y_0)- \P(|Y|> y_0)}\Bigl( D_\eps +  (q \eps y_0+\eta)\, \P(|Y|\leq y_0)\Bigr). 
$$
To conclude, we note that the denominator on the right-hand side is equal to $1- 2 \P(|Y|>y_0)$ and bound $D_\eps$ by~\eqref{eq:mollifiederror2}.
\end{proof}

\subsection{Tails and exponential moments of the standard Gaussian} \label{sec:gaussiantails} 

The \emph{Mills ratio} of a continuous random variable is the ratio of its survival function and its probability density function. For standard normal laws, the asymptotic behavior of the Mills ratio \cite {gordon} is 
\be \label{eq:mills}
	\frac{\P(Z \geq x)}{\sqrt{2\pi}^{-1}\exp(- x^2/2)} 
		=(1+o(1)) \frac1x 
	 \qquad (x\to \infty).
\ee
 Eq.~\eqref{eq:mills} is complemented by the following lemma. 

 \begin{lemma}  \label{lem:normal}
	Let $Z$ be a standard normal variable. Then we have, for all $\beta \geq 0$, 
	$$
		\E\bigl[ \e^{ - \beta Z} \1_{\{Z\geq 0\}}\bigr] =  \e^{\beta^2/2} \P(Z\geq \beta) \geq \frac{1}{\sqrt{2\pi}(\beta+1)}.
	$$
	Moreover for all $\beta>0$ and $\eta\in (-\beta,\beta)$, 
	$$
		\E\bigl[ \e^{ - (\beta+\eta) Z} \1_{\{Z\geq 0\}}\bigr]= \frac{\beta}{\beta+\theta \eta}\ \E\bigl[ \e^{ - \beta Z} \1_{\{Z\geq 0\}}\bigr]
	$$
	for some $\theta \in [-1,1]$.
\end{lemma} 
\begin{proof} 
We compute
\begin{align*}
	\E\bigl[ \e^{ - \beta Z} \1_{\{Z\geq 0\}}\bigr] & = \frac{1}{\sqrt{2\pi}} \int_0^\infty \e^{ - \beta u - u^2/2} \dd u = \frac{1}{\sqrt{2\pi}} \, \e^{\beta^2/2} \int_{0}^\infty \e^{ - (u+\beta)^2/2} \dd u \\
		& = 
		\frac{1}{\sqrt{2\pi}} \, \e^{\beta^2/2} \int_\beta^\infty\e^{-y^2/2} \dd y  =  \e^{\beta^2/2} \P(Z\geq \beta).  
\end{align*}
Next, we note 
\begin{align*}
	\P(Z\geq \beta) & = \frac{1}{\sqrt{2\pi}}\int_\beta^\infty \e^{- y^2/2} \dd y\geq \frac{1}{\sqrt{2\pi}}\int_\beta^\infty \e^{- y^2/2}\Bigl( 1- \frac{y}{(1+y)^2}\Bigr) \dd y\\
		&= \frac{1}{\sqrt{2\pi}} \Bigl[- \frac{\exp( - y^2/2)}{1+y}\Bigr]_{y=\beta}^{y=\infty} = \frac{1}{\sqrt{2\pi}} \frac{\exp(- \beta^2/2)}{\beta +1}.		
\end{align*}
We set $\psi(\beta) := \exp( \beta^2/2) \P(Z\geq \beta)= \E[\exp(- \beta Z)\1_{\{Z\geq 0\}}]$ and $q(\beta) := \beta \psi(\beta)$. Clearly $\psi$ is monotone decreasing. We check that $q(\beta)$ is monotone increasing. Indeed, 
$$
	q'(\beta) = (1+\beta^2) \e^{\beta^2/2} \P(Z\geq \beta) - \frac{\beta}{\sqrt{2\pi}}.
$$
We compute 
\begin{align*}
	\frac{1}{\sqrt{2\pi}}\frac{\beta}{1+\beta^2} \e^{-  \beta^2/2}
		& = \frac{1}{\sqrt{2\pi}}\int_\beta^ \infty \Bigl( - \frac{\dd}{\dd y} \frac{y}{1+y^2}\,  \e^{-y^2/2}\Bigr) \dd y 
		 = \frac{1}{\sqrt{2\pi}}\int_\beta^ \infty   \e^{-y^2/2} \frac{y^4 + 2y^2 - 1}{(y^2+1)^2} \dd y  \\
		&\leq \frac{1}{\sqrt{2\pi}}\int_\beta^\infty \e^{-y^2/2} \dd y = \P(Z\geq \beta)
\end{align*} 
and deduce $q'(\beta) \geq 0$ for all $\beta\geq 0$, so $q(\beta)$ is indeed monotone increasing. 
By the monotonicity of $\psi$ and $q$,  if $\eta\in (-\beta,0)$, then 
$$
	\psi(\beta) \leq \psi(\beta+\eta) \leq \frac{1}{\beta + \eta} \beta \psi(\beta).
$$
Similarly, if $\eta\in (0,\beta)$, then 
$$
	\frac{\beta}{\beta+\eta}\psi(\beta) 	\leq \psi(\beta+ \eta) \leq \psi(\beta)
$$
In both cases $\psi(\beta + \eta) / \psi(\beta) = \beta/(\beta + \theta \eta)$ for some $\theta\in [0,1]$. 
\end{proof}  

\subsection{Integrals of monotone functions and Kolmogorov distance}

\begin{lemma}  \label{lem:kolmo-mono} 
	Let $F$ and $G$ be two cumulative distribution functions of some probability measures and $f:[0,\infty) \to \R_+$ a monotone decreasing, continuous function with $\lim_{y\to \infty} f(y) =0$. Then 
	$$
		\bigl|\int_{-\infty}^\infty f \1_{[0,\infty)} \dd F - \int_{-\infty}^\infty f \1_{[0,\infty)} \dd G\bigr| 
			\leq 2 f(0) \sup_{y\in \R}|F(y) - G(y)|.
	$$
	The lemma extends to cumulative distribution functions of finite signed measures with total mass~$1$. 
\end{lemma} 

\noindent If $F$ and $G$ are continuous at $0$ (i.e., the associated measures have no atom at $0$), we may write $\int_0^\infty f\dd F$ and $\int_0^\infty f\dd G$ without creating ambiguities.

\begin{proof} 
	The lemma is a straightforward consequence of an integration by parts for Riemann-Stieltjes integrals (roughly, $\int_0^\infty f \dd F = - \int_0^\infty F \dd f + \text{boundary term}$). 
	For $n\in \N$, set 
	$$
		f_n := f(0) \1_{\{0\}}+\sum_{k=0}^\infty f(\tfrac{k+1}{n}) \1_{(k/n, (k+1)/n]}, 
	$$
 	further define $a_{k+1}^n :=  f(\tfrac{k}{n}) - f(\tfrac{k+1}{n}) \geq 0$. Let $\mu$ be the measure on $\R$ with $\mu((-\infty,x]) = F(x)$. 	Summing by parts, we get 
 	\begin{align*}
		\int_{-\infty}^\infty f_n \1_{[0,\infty)} \dd  F &= f(0) \mu(\{0\})  +  \sum_{k=0}^\infty \bigl(f(0)-a_1^n-\cdots - a_{k+1}^n\bigr) \Bigl( F\bigl( \tfrac{k+1}{n}) - F\bigl( \tfrac{k}{n}\bigr) \Bigr)\\
		& = f(0) \bigl( \mu(\{0\}) + 1- F(0)\bigr) -  \sum_{\ell=0}^\infty a_{\ell+1}^n \sum_{k=\ell}^\infty \Bigl( F\bigl( \tfrac{k+1}{n}) - F\bigl( \tfrac{k}{n}\bigr) \Bigr) \\
		& = f(0)  \lim_{\eps \searrow 0} \bigl(1-F(-\eps)\bigr) -  \sum_{\ell=0}^\infty a_{\ell+1}^n \bigl(1 - F\bigl(\tfrac \ell n)\bigr). 
	\end{align*} 
	A similar representation holds true for the integral against $G$. Since $0\leq a_{\ell+1}^n \leq f(0)$ for all $\ell$, we deduce 
	$$
		\Bigl|\int_{-\infty}^\infty f_n \1_{[0,\infty)} \dd  F - \int_{-\infty}^\infty f_n \1_{[0,\infty)} \dd  G \Bigr| \leq 2 f(0) \sup_y |F(y)- G(y)|.
	$$
	We pass to the limit $n\to \infty$, note $f_n\nearrow f$ because of the continuity and monotonicity of $f$, and obtain the lemma. 
\end{proof} 

\subsection{Positivity of truncated exponentials} 

Let 
$$
	\exp_{m}(u):= \sum_{k=0}^m \frac{u^k}{k!}
$$
denote the truncated exponential function. 
\begin{lemma} \label{lem:truncated-exp}
	We have $\exp_{2n}(u) > 0$ for all $u\in \R$ and $n\in \N_0$. 
\end{lemma} 

\noindent The analogous statement for the exponential series truncated after odd integers $2n+1$ is false, since  $\exp_{2n+1}(u)$ is a polynomial that goes to $-\infty$ as $u\to -\infty$. 

\begin{proof} 
	We follow~\cite[pp.~37--38]{SS91}.
	For $u\geq 0$ the inequality is obvious, so we need only treat $u<0$. Set 
	$$
		a_k = \frac{u^{2k-1}}{(2k-1)!} + \frac{u^{2k}}{(2k)!} = \frac{u^{2k-1}}{(2k-1)!}\Bigl( 1+ \frac{u}{2k}\Bigr) \qquad (k\geq 1).
	$$
	If  $u\leq -2k$, then $a_k\geq 0$, so if $u\leq - 2n$, then $a_k\geq 0$ for all $k=1,\ldots,n$ and $\exp_{2n}(u) = 1+ a_1+\cdots + a_n > 0$. If $- 2n < u <0$, then for all $k \geq n+1$, we have $u >- 2k$  hence $a_k <0$. It follows that 
	\begin{equation*}
		\exp_{2n}(u) = \exp(u) - \sum_{k=n+1}^\infty a_k > \exp(u) >0. \qedhere
	\end{equation*} 
\end{proof} 

\noindent An alternative proof is based on Taylor's theorem with integral remainder: for $u<0$ and even $m$, 
\[
	\exp(u) - \exp_{m}(u) = \frac{1}{m!} \int_0^{u} (u-t)^{m+1} \e^t \, \mathrm d t 
		= - \frac{1}{m!}\int_u^0 s^{m+1} \e^{u-s} \dd s >0
\] 
hence $\exp_m(u) \geq \exp(u) \geq 0$. 

\section{Concentration inequality. Proof of Theorem~\ref{thm:concentration}}\label{sec:concentration} 

For the proof of Theorem~\ref{thm:concentration} we follow the proof presented in~\cite{SS91} for the first part (up until~\eqref{eq:markov4} below), but then take a short-cut in order to avoid cumbersome numerical evaluations; instead we exploit a relation between truncated exponentials and Poisson distributions, which we learnt about from Rudzkis and Bakshaev \cite{Rudzkis_Bakshaev}.

The first step of the proof consists in applying Markov's inequality to the truncated exponential $\exp_{2n}(x)$, which is monotone increasing on $\R_+$. This yields, for every $h>0$ and $n\in \N$, and $x\geq 0$, 
\be \label{eq:markov0}
	\P(X\geq x) \leq \frac{\E[\exp_{2n}(hX)]}{\exp_{2n}(hx)} = \frac{1}{\exp_{2n}(hx)}\sum_{k=0}^{2n} \frac{h^k}{k!} m_k.
\ee
We check that 
\be \label{eq:markov1}
	\sum_{k=0}^{2n} \frac{h^k}{k!} m_k\leq \exp_n\Bigl( \sum_{j=2}^{2n} \frac{h^j}{j!}|\kappa_j|\Bigr).
\ee
It is enough to show that for each $k=2,\ldots, 2n$, the moment $m_k$ is smaller than the coefficient of $h^k$ of the series obtained by expanding the right-hand side of the inequality. By definition of the cumulants, the moment $m_k$ is equal to the coefficient of $h^k$ in the formal power series $\exp(\sum_{j=2}^\infty h^j \kappa_j /j!)$, which gives
\[
	m_k = \sum_{r=1}^\infty \frac{1}{r!}\sum_{\substack{ j_1,\ldots,j_r\geq 2:\\j_1+\cdots + j_r = k}} \frac{\kappa_{j_1}}{j_1!}\cdots \frac{\kappa_{j_r}}{j_r!}.
\] 
For $k \leq 2n$, the only non-zero contributions are from $r \leq n$ and $2\leq j_\ell \leq 2n$, therefore
\[
	|m_k| \leq \sum_{r=1}^n \frac{1}{r!}\sum_{\substack{2 \leq j_1,\ldots,j_r\leq 2n:\\j_1+\cdots + j_r = k}} \frac{|\kappa_{j_1}|}{j_1!}\cdots \frac{|\kappa_{j_r}|}{j_r!}.
\] 
The right-hand side is equal to the coefficient of $h^k$ in $\exp_n\bigl( \sum_{j=2}^n |\kappa_j|\, h^j/j!)$. This completes the proof of~\eqref{eq:markov1}. 

Next we show that for a specific $x$-dependent choice of $h$ and $n$, we may bound
\be \label{eq:markov3}
	\sum_{j=2}^{2n} \frac{|\kappa_j|}{j!} h^j \leq \frac{h x}{2}. 
\ee
Consider first the case $H=1$. Choose $h:= x(x\overline \Delta)^{1/(1+\gamma)} / (x^2 + (x\overline{\Delta})^{1/(1+\gamma)})$ so that
\be \label{eq:hdef}
	\frac{1}{hx}= \frac{1}{x^2} + \frac{1}{(x\overline{\Delta})^{1/(1+\gamma)}}. 
\ee
By the assumptions on the cumulants, $|\kappa_j|/j! \leq j!^{\gamma} /2$ and hence 
\[
	\sum_{j=2}^{2n} \frac{|\kappa_j|}{j!} ^j  \leq \frac{h^2}{2} \sum_{j=2}^{2n} \bigl(\frac{j!}{2}\bigr)^{\gamma} \bigl( \frac{h}{\overline{\Delta}}\bigr)^{j-2}. 
\] 
A straightforward induction over $n$ shows that for $j \leq 2n$, we have $j!/2 \leq (2n)^{j-2}$, thus if we pick $n\leq hx /2$ then $j!/2 \leq (hx)^{j-2}$ for $j \leq 2n$
and 
\[
	\sum_{j=2}^{2n} \frac{|\kappa_j|}{j!} h^j \leq \frac{h^2}{2}\sum_{j=2}^{2n} \Bigl( (hx)^\gamma \frac{h}{\overline{\Delta}}\Bigr)^{j-2} =  \frac{h^2}{2}\sum_{j=2}^{2n} q^{j-2},\quad q: = (hx)^\gamma \frac{h}{\overline{\Delta}}.
\] 
Noticing $ q = (hx)^{1+\gamma} / (\overline{\Delta} x)<1$, 
we deduce 
\[
	\sum_{j=2}^{2n} \frac{|\kappa_j|}{j!} h^j \leq \frac{h^2}{2} \frac{1}{1-q}.
\] 
Set $\tilde q = q^{1/(1+\gamma)}$, then $q < \tilde q < 1$ and by~\eqref{eq:hdef}, $1 = \frac hx + \tilde q$ hence $x = h/(1-\tilde q)$ and 
\[
	\sum_{j=2}^{2n} \frac{|\kappa_j|}{j!} h^j \leq \frac{h^2}{2} \frac{1}{1-\tilde q} = \frac{hx}{2}. 
\] 
This completes the proof of~\eqref{eq:markov3}. 

The bounds~\eqref{eq:markov0}, \eqref{eq:markov1} and~\eqref{eq:markov3} yield 
\be \label{eq:markov4}
	\P(X\geq x) \leq \frac{\exp_n(hx/2)}{\exp_{2n}(hx)}
\ee
for $2 n \leq h x$. If we had exponentials instead of truncated exponentials, then the right-hand side of the previous inequality would be $\exp( - h x/2)$ and we would be done, in the case $H=1$. 

Let us choose $n:= \lfloor (hx)/2\rfloor$, then $2n \leq hx < 2n+2$. Let $N_{2n+2}$ be a Poisson random variable with parameter $2n+2$. Then 
\[
	\e^{- h x}\exp_{2n}(hx) \leq  \e^{-2n} \exp_{2n}(2n+2)= \e^{2}\P(N_{2n+2}\leq 2n).
\]
The variable $N_{2n+2}$ is equal in distribution to the sum of $2n+2$ i.i.d.\ Poisson variables with parameter $1$. Applying the central limit theorem, one finds that 
\[
	\lim_{n\to \infty} \P(N_{2n+2} \leq 2 n) = \frac 12. 
\] 
As a consequence, there exist $c,m>0$ such that for all sufficiently large $hx\geq m$ (hence large $n$), 
\[
	\exp_{2n}(hx) \geq c\, \e^{hx} \geq c\, \e^{hx/2}\exp_n(h x/2).
\] 
Together with~\eqref{eq:markov4} this gives 
\[
	\P(X\geq x) \leq \frac 1c\, \e^{- h x /2}
\]  
for all $hx \geq m$. Replacing $1/c$ by $C:=\max (1/c,\exp(m/2))$ we obtain an inequality that is true for all values of $hx$. This completes the proof of the theorem when $H=1$. 

For general $H>0$, define $\tilde X:= X/\sqrt H$. Then $\tilde X$ satisfies the assumption of the theorem with $H=1$ and $\overline{\Delta}$ replaced with $\tilde \Delta = \overline{\Delta}\, \sqrt{H}$ and the proof is easily concluded. \hfill $\qed$

\section{Bounds under Cram{\'e}r's condition. Proof of Theorem~\ref{thm:easy}} \label{sec:easy} 

Here we prove Theorem~\ref{thm:easy} on random variables that have exponential moments~\eqref{condition1}, which corresponds  to $s=\infty$ in condition~\eqref{condition-s}.  This helps explain the strategy for finite $s$, and some of the estimates  are reused for finite $s$. The Cram{\'e}r-Petrov series $\sum_{j=3}^\infty \lambda_j x^j$ is defined in Appendix~\ref{app:cramer}. Proposition~\ref{prop:cramer-concrete1} shows that under the conditions of Theorem~\ref{thm:easy}, we have
$$
	|\lambda_j|\leq \frac{\Delta^2/2}{(3\Delta/10)^j}.
$$

\begin{proof}[Proof of Theorem~\ref{thm:easy}]
The proof  of Theorem~\ref{thm:easy} comes in several steps: 
\begin{enumerate} 
	\item Introduce a tilted variable $X_h$, as often done in  the proof of the Cram{\'e}r's large deviation principle, and a centered normalized version $\widehat X_h$ with $\E[\widehat X_h]=0$ and $\mathbb V[\widehat X_h]=1$. 
	\item Estimate the Kolmogorov distance between the law of $\widehat X_h$ and the standard normal distribution by comparing characteristic functions. 
	\item Undo the tilt: express $\P(X\geq x)$ in terms of $\widehat X_h$, conclude with the help of step (2). 
\end{enumerate} 
The  exponential moments $\E[\exp( t X)] $ are finite for all $t\in (-\Delta, \Delta)$ (or $t\in \C$ with $|\Re t|\leq \Delta$), and the Taylor series of the cumulant generating function $\varphi(t) = \log \E[\exp( t X)]$ has radius of convergence at least $\Delta$,
$$
	\varphi(t) = \log \E[\e^{t X}]= \frac{t^2}{2} + \sum_{j=3}^\infty \frac{\kappa_j}{j!} t^j \qquad (|t|<\Delta).
$$
Let $I(x) = \sup_{t\in \R} (t x - \varphi(t))$ be the Cram{\'e}r rate function.
In order to estimate $\P(X\geq x)$ with $x>0$  we work with a tilted random variable $X_h$, given by 
$$
	\P(X_h \in B) = \e^{- \varphi(h)} \E[\exp( h X) \1_B(X)],
$$
assuming the equation 
$$
	x = \varphi'(h) 
$$
admits a solution $h \in (0,\Delta)$. The use of tilted variables is fairly standard in the proof of Cram{\'e}r large deviation principle.  Now, the key idea is to approximate the distribution of $X_h$ by that of a normal random variable with mean $x$ and variance $\sigma^2(h):= \varphi''(h)$. Equivalently, defining 
$$
	\widehat X_h : = \frac{X_h - x}{\sigma(h)}, 
$$
we approximate $\mathscr{L} (\widehat X_h) \approx \mathcal N(0,1)$. The error in the normal approximation is quantified by looking first at characteristic functions. We have 
\begin{align} 
	\chi_h(t)&:= \E[\exp(\mathrm i t \widehat X_h)] = \E\Bigl[\exp\Bigl( \mathrm i t \frac{X_h -  x}{\sigma(h)} \Bigr)\Bigr] \notag \\
		& = \exp\Bigl( - \varphi(h) - \mathrm i t \frac{x}{\sigma(h)}\Bigr) \E\Bigl[ \exp\Bigl( h X + \mathrm {i} t \frac{X}{\sigma(h)}\Bigr] \notag \\
		& = \exp\Bigl( \varphi\bigl( h + \mathrm i \tfrac{t}{\sigma(h)}  \bigr) - \varphi(h)  - \mathrm i t \tfrac{x}{\sigma(h)} \Bigr),\label{eq:sichi}
\end{align} 
for all $t\in \R$ that are small enough so that $|h + \mathrm i t/\sigma(h) |\leq T <\Delta$. A Taylor approximation for $\varphi$ at $h$ showsfor some $\theta \in[-1,1]$
$$
	 \varphi\bigl( h + \mathrm i \tfrac{t}{\sigma(h)}  \bigr) - \varphi(h)  - \mathrm i t \tfrac{x}{\sigma(h)} = \tfrac12 \varphi''(h) \bigl(  \tfrac{\mathrm i t  }{\sigma(h)}\bigr) ^2+ \frac16 \varphi'''( t + \mathrm i \theta \tfrac{t}{\tilde \sigma (h)} ) t ^3  = - \tfrac{t^2}{2} +O(t^3).
$$
Our cumulant bounds allow for an easy bound on the third derivative: if $z\in \C$  satisfies $|z|\leq T<\Delta $, then 
\be \label{eq:thirdder}
	|\varphi'''(z)|\leq \sum_{j=3}^\infty |\kappa_j| \frac{|z|^{j-3}}{(j-3)!} \leq \frac{1}{\Delta} \sum_{n=1}^\infty n \bigl( \frac{|z|}{\Delta}\bigr)^{n -1} \leq \frac{1}{\Delta} \frac{T/\Delta}{(1- T/\Delta)^2}. 
\ee
For $T$ bounded away from $\Delta$, we deduce that 
$$
	\chi_h(t) = \exp\Bigl( - \frac{t^2}{2}\bigl( 1+ O(\frac{t}{\Delta})\bigr) \Bigr)
$$
which, for large $\Delta$, is close to the characteristic function $\exp(- \frac{t^2}{2})$ of the standard normal variable. Careful estimates based on the smoothing inequality~\cite[Lemma~2, Chapter~XVI.3]{feller-vol2}
$$
	\sup_{y\in \R}\bigl| \P(\widehat X_h \leq y) - \P(Z\leq y)\bigr|\leq \frac1\pi \int_{-T}^T \Bigl| \frac{\chi_h(t) - \exp(- t^2/2)}{t}\Bigr|\dd t + \frac {24}{\pi T}\frac{1}{\sqrt{2\pi}}
$$
show that the Kolmogorov distance between $\mathscr L(\widehat X_h)$ and the normal distribution is of order $O(\frac{1}{\Delta})$: for all $\delta\in (0,1)$, there exists a constant $C_\delta>0$ such that whenever the tilt parameter $h=h(x)$ exists and is in $[0,\delta \Delta]$, then 
$$
	D(\widehat X_h, Z) := \sup_{y\in \R}\bigl| \P(\widehat X_h \leq y) - \P(Z\leq y)\bigr|\leq \frac{C_\delta}{\Delta}. 
$$
The next step consists in undoing the tilt: we express the probability we are after in terms of the tilted  recentered variable $\widehat X_h$ by  
$$
		\P(X\geq x) = \e^{\varphi(h)} \E\Bigl[\e^{- hX_h} \1_{\{X_h \geq x\}}\Bigr] = \e^{\varphi(h) - h x} \E\Bigl[\e^{- h \sigma(h)  \widehat X_h } \1_{\{\widehat X_h \geq 0\}}\Bigr].
$$
An easy lemma on Kolmogorov distances (Lemma~\ref{lem:kolmo-mono}) shows 
$$
	\Bigl| \E\Bigl[\e^{- h \sigma(h)  \widehat X_h } \1_{\{\widehat X_h \geq 0\}}\Bigr] - \E\Bigl[\e^{-h \sigma (h) Z} \1_{\{Z\geq 0\}}\Bigr]\Bigr| \leq 2 D(\widehat X_h, Z) \leq \frac{2C_\delta}{\Delta}.
$$
Further noting $\varphi(h)  - h x=  - I(x)$, we get 
\be \label{eq:sik}
	\P(X\geq x) = \e^{- I(x)}\Bigl(  \E\Bigl[\e^{-h \sigma (h) Z} \1_{\{Z\geq 0\}}\Bigr] + \theta \frac{2C_\delta}{\Delta}\Bigr)
\ee
for some $\theta \in [-1,1]$.
The tilted variance $\sigma(h)$ is close to $1$ because
$$
	|\sigma^2(h)-1| = |\varphi''(h) - \varphi''(0)| \leq \frac{h}{\Delta}
$$
by the bound~\eqref{eq:thirdder} on the third derivative. The tilt parameter $h$ is close to $x$ because 
\be \label{eq:sihx}
	|x-h|= |\varphi'(h)-h|\leq \sum_{j=3}^\infty \frac{|h|^{j-1}}{(j-1)!} |\kappa_j|\leq \Delta \frac{(h/\Delta)^2}{1- h/\Delta},
\ee
which yields $x -h = h(1+O(h/\Delta))$. Altogether we get
$$
	h\sigma(h) = x \Bigl( 1+O\Bigl(\tfrac h\Delta\Bigr)\Bigr).
$$
Bounds on the standard normal distribution and a completion of squares given in Lemma~\ref{lem:normal} yield
$$
	 \E\Bigl[\e^{-h \sigma (h) Z} \1_{\{Z\geq 0\}}\Bigr] = \frac{1}{1+O (h/\Delta)} \E\Bigl[\e^{-x Z} \1_{\{Z\geq 0\}}\Bigr] = \bigl(1+O(h/\Delta)\bigr) \e^{x^2/2} \P(Z\geq x).  
$$
Inserting this expression into~\eqref{eq:sik}, together with $\frac{1}{1+O (h/\Delta)}= 1+O (h/\Delta)$,  we get
$$
	\P(X\geq x) = \e^{- I(x) +x^2/2}\Bigl( \bigl(1+ O(\tfrac{h}{\Delta})\bigr) \P(Z\geq x) + \theta \frac{C_\delta}{\Delta} \e^{-x^2/2} \Bigr). 
$$
We factor out $\P(Z\geq x)$, use the lower bound for $\P(Z\geq x)$ from Lemma~\ref{lem:normal}, remember~\eqref{eq:sihx}, and obtain 
$$
	\P(X\geq x) = \e^{- I(x) +x^2/2}  \P(Z\geq x) \Bigl( 1+ O\bigl(\tfrac{x+1}{\Delta}\bigr) \Bigr). 
$$
By Appendix~\ref{app:cramer}, given $\delta\in (0,1)$, there exists a constant $c_\delta>0$ such that for all $x\in[0,c_\delta \Delta]$, the Cram{\'e}r-Petrov series converges, and the equation $\varphi'(h)=x$ has a   unique solution $h=h(x)$ and this solution is in $[0,\delta \Delta]$.
Going through the previous estimates carefully, we see that there is a constant $C_\delta>0$ such that for all $x\in [0,c_\delta \Delta]$, 
\begin{equation*} 
	\P(X\geq x) = \P(Z\geq x) \Bigl( 1+ \theta C_\delta \tfrac{x+1}{\Delta} \Bigr) \exp\Bigl( \sum_{j=3}^\infty \lambda_j x^j\Bigr).\qedhere 
\end{equation*}
\end{proof}

\section{Bounds with finitely many moments. Proof of Theorem~\ref{thm:lemma22}} \label{sec:sstar}

Now we turn to condition~\eqref{condition-s}  that 
$|\kappa_j|\leq (j-2)! / \Delta^{j-2}$ for $j=3,\ldots, s+2$ with $3 \leq s \leq 2 \Delta^2$. Remember that the random variables are centered and normalized as $\E[X]=0$, $\mathbb V(X) =1$.

\subsection{Introducing a modified tilted measure} \label{sec:tildedefs}
 The Taylor series of $\varphi(t)$ may have radius of convergence zero, so we work instead with the truncated functions
\begin{equation}\label{def_m_k_tilde}
	\tilde \varphi(t) = \sum_{j=2}^s \frac{\kappa_j}{j!} t^j,\qquad \e^{\tilde \varphi(t)} =: \sum_{j=0}^\infty \frac{\tilde m_j}{j!} t^j.
\end{equation}
Notice $\tilde m_j =\E[X^j]=:m_j$ for $j\leq s+2$.
The random variable
$X$ may have infinite exponential moments, so the exponential tilt is no longer possible; we replace the exponential $\exp(t x)$  in the tilt by 
\begin{equation}\label{eq:def_g_h}
	g_t(x) := \exp_s(tx) + x^2 \tilde r(t) 
\end{equation}
with
\begin{equation}\label{def_r_tilde}
	\exp_s(tx):=\sum_{j=0}^s \frac{(t x)^j}{j!} ,\quad  \tilde r(t):=\sum_{j=s+1}^\infty \frac{1}{j!} \tilde m_j t^j.
\end{equation}
The truncated exponential is fairly natural, the additional term $x^2 \tilde r(t)$ in $g_t(x)$ ensures that the tilted measure $\mu_h$ defined below has total mass $1$. For small $h=h(x)$ such that
$$x = \tilde m(h) = \tilde \varphi'(h)\quad\text{and}\quad \tilde \sigma(h)^2 := \tilde \varphi''(h), $$
we introduce a signed measure $\mu_h$ on $\R$  by 
\be \label{eq:mod-tilt}
	\mu_h(B) = \e^{-\tilde \varphi(h)} \E\Bigl[ g_h(X) \1_{\{(X- x)/\tilde \sigma_h \in B\}} \Bigr].
\ee
Because of 
\be \label{eq:egh}
	\E[g_h(X)] = \sum_{j=0}^s \frac{m_j}{j!} h^j + \E[X^2] \sum_{j=s+1}^\infty \frac{\tilde m_j}{j!}\, h^j = \e^{\tilde \varphi(h)}, 
\ee
the signed measure is normalized, i.e.,  $\mu_h(\R)=1$, however the function $g_h$ may take negative values so $\mu_h$ is not necessarily a probability measure. Nevertheless, we will see that $\mu_h$ is close to a normal distribution. For later purposes we note the inverse relation to~\eqref{eq:mod-tilt}: for all Borel sets $B\subset \R$, 
\be \label{eq:mod-tilt-inversion} 
	\P(X\in B) = \e^{\tilde \varphi(h)} \int_\R \frac{1}{g_h( \tilde \sigma(h)y+ x)} \1_B\bigl( \tilde \sigma(h)y+ x\bigr) \dd \mu_h(y).
\ee
Later we evaluate the function $\tilde \varphi(z)$ for complex parameters $z$. A visual summary of various quantities that are introduced for proofs is given in Figure~\ref{fig:params} below.

\subsection{Moment estimates} 
The key step of the proof will be, just as in the case $s=\infty$, to show that the Fourier transform $\chi_h(t)$ of $\mu_h$  is close to $\exp(-t^2/2)$, for small $t$, so that we can apply Zolotarev's lemma.  A new feature compared to $s=\infty$ is that we have to deal with truncation errors. In order to estimate them, it is helpful to have bounds on the quantities $\tilde m_j$ and on truncated moments. 

\begin{lemma} \label{lem:tildemk-bound}
Let $\tilde m_k$ be given by \eqref{def_m_k_tilde}. 	We have for all $k\in \N$, 
	$$
		\frac{1}{k!}|\tilde m_k|\leq \max \Bigl( \frac{1}{\sqrt{k/(2\e)}^{k}},  \frac{1}{(\Delta/\sqrt{\mathrm e})^k}\Bigr)=\begin{cases}
			1/\sqrt{k/(2\e)}^{k},&\quad k \leq 2 \Delta^2,\\
			1/(\Delta/\sqrt{\mathrm e})^k,&\quad k> 2 \Delta^2.
		\end{cases} 
	$$
	Moreover the moments $m_k:=\E[X^k]$ satisfy $m_k= \tilde m_k$ for $k=0,1,\ldots, s+2$. 
\end{lemma} 

\begin{remark}
For small $k$, the $\tilde m_k$'s satisfy a bound inherited from the moments of Gaussian variables.  For large $k$, the $\tilde m_k$'s have a behavior closer to the moments of a random variable that has exponential moments $\mathbb E[\exp(tY)]$ up to order $|t|<\Delta$, e.g., an exponential variable $Y\sim \mathrm{Exp}(\Delta)$. Indeed, as $k\to \infty$, 
\begin{align*}
	\frac{1}{k!} \mathbb E\bigl[|Y|^k\bigr] = \frac{1}{k!}\frac{2^{k/2}\Gamma(\frac{k+1}{2})}{\sqrt{\pi}} \sim \frac{\sqrt{2}}{k!} \Bigl(\frac{k}{\mathrm e}\Bigr)^{k/2} \sim \frac{1}{\sqrt{\pi k}(k/\e)^{k/2}}  \leq  \frac{1}{\sqrt{k/(2\e)}^k},
\end{align*}
where we have used Stirling's formula $\Gamma(x+1)\sim \sqrt{2\pi x} (x/\mathrm e)^x$ for the Gamma function and for the factorial $k! = \Gamma(k+1)$. 
On the other hand the moments of an exponential random variable $Y\sim \mathrm{Exp}(\Delta)$ are given by 
$$
	\frac{1}{k!}\, \E\bigl[Y^k\bigr] =  \frac{1}{\Delta^k} \leq \frac{1}{(\Delta/\sqrt{\e})^k}. 
$$
\end{remark}

\begin{proof} [Proof of Lemma~\ref{lem:tildemk-bound} ]
	We have 
	\be \label{eq:tildemkrep}
		\frac{1}{k!}\, \tilde m_k = \sum_{r=1}^{\lfloor k /2\rfloor } \frac{1}{r!} \sum_{\substack{j_1+\cdots + j_r = k \\ j_\ell = 2,\ldots, s }} \frac{\kappa_{j_1}\cdots\kappa_{j_r}}{j_1!\cdots j_r !}. 
	\ee
	Hence 
	$$
		\Bigl| \frac{1}{k!} \tilde m_k\Bigr|  \leq \sum_{r=1}^{\lfloor k /2\rfloor } \frac{1}{r!} \sum_{\substack{j_1+\cdots + j_r = k \\ j_\ell = 2,\ldots, s }} \Biggl( \prod_{\ell=1}^r \frac{1}{j_\ell(j_\ell -1)} \Biggr) \frac{1}{\Delta^{k - 2 r}} 
			 \leq \sum_{r=1}^{\lfloor k/2\rfloor} \frac{1}{r!} \Biggl( \sum_{j=2}^s \frac{1}{j(j-1)} \Biggr)^r  \frac{1}{\Delta^{k-2r}}.  
	$$
	Using 
	$
		\sum_{j=2}^\infty \frac{1}{j(j-1)} = \int_0^1 \bigl( - \log (1- t) \bigr) \dd t = 1,
	$
	we obtain for $\Delta^2 \leq k/2$
	$$
		\Bigl| \frac{1}{k!} \tilde m_k\Bigr| \leq \frac{1}{\Delta^k}\sum_{r = 1}^{\lfloor k /2\rfloor} \frac{1}{r!} \Delta^{2r} \leq \frac{1}{\Delta^k}\, \e^{\Delta^2} \leq \frac{1}{\Delta^k}\, \e^{k/2}.
	$$
	For $\Delta^2 \geq k/2$, we estimate instead
	\begin{equation*}
		\Bigl| \frac{1}{k!} \tilde m_k\Bigr| \leq \sum_{r=1}^{\lfloor k/2\rfloor} \frac{1}{r!}   \frac{1}{\sqrt{k/2}^{k-2r}} \leq \frac{1}{\sqrt{k/2}^{k}}\,  \e^{k/2}.
	\end{equation*} 
	The moments $\E[X^k]$, $k\leq s+2$,  are given by a formula similar to~\eqref{eq:tildemkrep}, but in theory the range of summation indices is now $j_\ell = 2,\ldots, s+2$. However, from $j_1+\cdots + j_r =k$ and $j_\ell \geq 2$ we get that all indices must be smaller or equal to $k-2 \leq s$, so we are back to Eq.~\eqref{eq:tildemkrep}.  
\end{proof}

Let us quickly explain how Lemma~\ref{lem:tildemk-bound} affects the parameter choices and the bounds on truncation errors. For $z\in\mathbb C$ bounded away from $\Delta/\sqrt{\mathrm e}$ and $\sqrt{s/(2\mathrm e)}$,  the remainder term $\tilde r(z)$ from \eqref{def_r_tilde} is bounded by 
$$
	|\tilde r(z) | \leq \sum_{j=s+1}^\infty \frac{|z|^j}{(\Delta/\sqrt{\e})^j} + \sum_{j=s+1}^\infty \frac{|z|^j}{\sqrt{s/(2\e)}^j} = O \Bigl( \Bigl( \frac{|z|}{\Delta/\sqrt{\e}}\Bigr)^{s+1}\Bigr) + O \Bigl( \Bigl( \frac{|z|}{\sqrt{s/(2\mathrm{e})}}\Bigr)^{s+1}\Bigr),
$$
where we have estimated $\max(x,y) \leq x+y$ for $x,y\geq 0$. In the bound there is nothing to be gained from having $s$ larger than $2 \Delta^2$. Indeed, if $s>2\Delta^2$, then it is the first term that dominates, and it corresponds to the bound obtained for $s = 2\Delta^2$. This is why, in assumption $(S^*)$, we do not bother with $s>2\Delta^2$. 
It will be convenient to work with 
$$
	s\leq 2 \Delta^2,\qquad |z|<a := \sqrt{\frac{s}{4\mathrm{e}}}. 
$$
Then for all $k\geq s$, we have $\sqrt{\frac{k}{2\e}}\geq \sqrt{\frac{s}{2\e}}=\sqrt{2}\,a$ and $\frac{\Delta}{\sqrt{2}}\geq\frac{1}{\sqrt{\e}}\sqrt{\frac s2} = \sqrt2 a$ by the assumption $s\leq 2 \Delta^2$, hence 
$$ 
	\min \Bigl(\sqrt{\frac{k}{2\e}}, \frac{\Delta}{\sqrt 2}\Bigr) \geq \sqrt{2} a\quad \text{for all }k\geq s
$$
and
\be \label{eq:tildemk-klarge}
	\frac{1}{k!} |\tilde m_k|\leq \frac{1}{(\sqrt 2 a)^k}\quad \text{for all }k\geq s.
\ee
It is in this form that Lemma~\ref{lem:tildemk-bound} is used later on.
	
In Lemma~\ref{lem:chierror1} below we need to estimate sums of truncated moments, which are of the type $\sum_{k=0}^s \frac{z^k}{k!} \E[ |X|^k \1_{\{|X|\geq b\}}]$, for some additional truncation parameter $b\geq 1$. 
Because of $\E[X^k] = \tilde m_k$ for $k=2,\ldots,s+2$, the bound from Lemma~\ref{lem:tildemk-bound} extends to the moments $\E[X^k]$, $k \leq s$, and we get for $k \leq s/2$:
\be  	\label{eq:mkbeasy0} 
\E\bigl[ |X|^k \1_{\{|X|\geq b\}}\bigr] \leq \frac{1}{b^k}\E\bigl[X^{2k}\bigr] \leq \frac{1}{b^k} \frac{(2k)!}{\min\bigl( k/\mathrm{e}, \Delta^2/\e\bigr)^k} = \frac{1}{b^k}\frac{(2k)!}{ (k/\e)^k},
\ee
where we have used $2k \leq s \leq 2 \Delta^2$. By Stirling's formula,
$$
	 \frac{(2k)!}{k!(k/\e)^k} = (1+o(1)) 4^k \sqrt2 \qquad (k\to \infty),
$$
so we deduce 
\be \label{eq:mkbeasy}
	\frac{1}{k!}\, \E\bigl[ |X|^k \1_{\{|X|\geq b\}}\bigr] \leq C \frac{4^k}{b^k} 
\ee
for $k \leq s/2$, with $C\geq \sqrt{2}$ some constant independent of $s$. This bound has the drawback of not being small for $k=0$, moreover we need to proceed differently for $k >s/2$. Therefore the bound~\eqref{eq:mkbeasy} is complemented by the following lemma. 

\begin{lemma} \label{lem:mkb-bound}
	For $30\leq s\leq 2\Delta^2$, $a:=\sqrt{s/(4\e)}$ and $b:=4a$, we have  
	$$
		 \frac{1}{k!}\, \E\bigl[|X|^k\1_{\{|X|\geq b\}}\bigr] \leq \begin{cases}
			\sqrt{2} / a^{\lfloor b\rfloor}   , & \quad 0 \leq k \leq b,\\
			\sqrt{2} / a^k, &\quad b <k \leq s. 
		\end{cases} 
	$$
\end{lemma} 

\noindent Note that $b = \sqrt{4 s/\e}\geq 1$ and $b= \sqrt{s}\frac{2}{\sqrt{\e}}<\frac{s}{2} $ for $s\geq 30$. Set 
$$
	m_k(b):=  \E\bigl[|X|^k\1_{\{|X|\geq b\}}\bigr] .
$$

\begin{proof}
For the  proof, we distinguish the cases $0\leq k\leq s/2$ and $  k> s/2$.	For $1\leq k\leq s/2$, we proceed as in~\eqref{eq:mkbeasy0}. With Stirling's formula \cite{Robbins} 
$
	 \e^{1/(12n+1)} \sqrt{2\pi n} (n/\e)^n \leq n! \leq 	\e^{1/(12n)} \sqrt{2\pi n} (n/\e)^n, 
$ 
we obtain
	$$
		 \frac{1}{b^k}\frac{(2k)!}{k! (k/\mathrm e)^k} 
		\leq \frac{1}{a^k}\sqrt{2}\mathrm e^{\frac{1}{24 k}- \frac{1}{12 k+1}}  \leq \frac{\sqrt{2 }}{a^k}  $$
since for $k \geq 1$ we have $\frac{1}{24 k}- \frac{1}{12 k+1} \leq 0$. 
It follows that~\eqref{eq:mkbeasy} holds true with $C= \sqrt{2}$, for $k=1,\ldots, s/2$. This proves in particular the assertion of  Lemma~\ref{lem:mkb-bound} for $b<k \leq \frac{s}{2}$. 
	For $0 \leq k\leq {\lfloor b\rfloor}$, using $b \geq 1$, we have
	$$ 
	        \frac{1}{k!} m_k(b) 
	        \leq \frac{1}{k!} \E\Bigl[|X|^{k} \frac{|X|^{{\lfloor b\rfloor}-k}}{b^{{\lfloor b\rfloor}-k}}\1_{\{|X|\geq b\}}\Bigr] \leq \frac{1}{{\lfloor b\rfloor}!} \E\Bigl[|X|^{\lfloor b\rfloor} \1_{\{|X|\geq b\}}\Big]
	        = \frac{1}{{\lfloor b\rfloor}!} m_{\lfloor b\rfloor}(b).
	$$
	We combine this with the bound~\eqref{eq:mkbeasy} with $C = \sqrt{2}$---which is applicable because $b \leq s/2$---and obtain the first inequality in the lemma. 
	
	For $s/2 \leq k \leq s$ and $k$ even, we estimate 
		$m_k(b) \leq \E[|X|^k] = \E[X^k]= \tilde m_k$   and by Lemma~\ref{lem:tildemk-bound}, 
	$$
	\frac{1}{k!} m_k(b) \leq \frac{1}{[\min(\sqrt{k/(2 \mathrm e)}, \Delta^2/\sqrt{\mathrm e})]^k}
		= \frac{1}{\sqrt{k/(2 \mathrm e)}^k}  \leq \frac{1}{\sqrt{s/(4\mathrm e)}^k} = \frac{1}{a^k}. 
	$$
	If $k$ is odd, 
	 we handle $\E[|X|^k]\neq \E[X^k]$ by bounding the expectation by an even moment: 
	$$
		m_k(b)= \E\bigl[|X|^k \1_{\{|X|\geq b\}}\bigr]\leq \frac{1}{b} \E\bigl[|X|^{k+1}\bigr] = \frac{1}{b} \E\bigl[X^{k+1}\bigr] = \frac 1b \tilde m_{k+1}
	$$
	hence 
	$$
		\frac{1}{k!} m_k(b) \leq \frac{k+1}{b}\frac{\tilde m_{k+1}}{(k+1)!}  \leq \frac{k+1}{b} \frac{1}{\sqrt{(k+1)/(2\e)}^{k+1}}.
	$$
	Now 
	$$
		\frac{k+1}{\sqrt{k+1}^{k+1}} = (k+1) ^{-\frac{k-1}{2}} \leq \bigl(\frac s 2 +1\bigr)^{-(k-1)/2}
	$$
	hence
	\begin{align*}
		\frac{1}{k!} m_k(b) &\leq \frac{s/2+1}{4 \sqrt{s/(2\e)}} \, \frac{1}{\sqrt{(s/2+1)/(2\e)}^{k+1}} = \frac{s+2}{8 \sqrt{s/(4\e)}} \frac{1}{\sqrt{(s+2)/(4\e)}} \frac{1}{\sqrt{(s+2)/(4\e)}^k}\\
		& = \frac{4e\sqrt{s+2}}{8 \sqrt{s}} \frac{1}{\sqrt{(s+2)/(4\e)}^k} \leq \frac{e}{2}\sqrt{1+\frac{2}{s}} \, \frac{1}{a^k} \leq \frac{\sqrt{2}}{a^k} \text{ for }s\geq 30,
	\end{align*}
	since $\frac{e}{2}\sqrt{1+ \frac{2}{s}} \leq 1.41 \leq \sqrt{2}$ for $s\geq 30$. 
\end{proof} 

\subsection{Bounds on tilt parameters}
Remember the choice $s \leq 2 \Delta^2$, $a = \sqrt{s/(4\e)}\leq \Delta/\sqrt{2\e}$. 
From now on $\theta$ designates a generic constant in $[-1,1]$, not always the same, possibly dependent on other constants in statements; for example, $a = b+\theta c $ means ``there exists $\theta =\theta(a,b,c)\in [-1,1]$ such that $a= b+\theta c$. When the quantities involved are complex, $\theta$ is a complex number of modulus smaller than or equal to $1$.

\begin{lemma} \label{lem:tiltexist} 
	Assume that $s\leq 2 \Delta^2$. 
	Then for all $x\in [0,\frac23 a]$, the equation $\tilde \varphi'(h) = x$  has a unique solution $h \in [0,a]$. 
	In addition, if $x=\frac23 a\delta$ with $\delta \in [0,1]$, then the solution $h$ is in $[0,\delta a]$ and 
	$$
		x = h \bigl( 1+  \theta \frac{\delta }{3}\bigr), \quad  \tilde{\sigma}(h)^2= \varphi''(h) = 1 + \theta\, 0.751\, \delta >0.
	$$
\end{lemma}
\noindent Note that the polynomial equation $\tilde \varphi'(h) = x$ may have additional solutions $h>\frac23 a$. The constant $0.751$ replaces the constant $0.75$ from~\cite[Eq.~(2.21)]{SS91}, which we were not able to reproduce. 

\begin{proof}
	 Set $q:= 1/\sqrt{2\e}$ so that $a\leq q \Delta < \Delta /2$. Fix $\delta \in [0,1]$. For $h \in [0,\delta a]$, we have $h/\Delta \leq q \delta$ and
	\be \label{eq:tiltexist1}
		| \tilde \varphi''(h) - 1| =\Bigl|\sum_{j=3}^s \frac{\kappa_j}{(j-2)!} h^ {j-2}\Bigr| 
			\leq \sum_{j=3}^\infty \Bigl(\frac{h}{\Delta}\Bigr)^{j-2}  = \frac{h/\Delta}{1- h /\Delta} \leq \frac{q}{1-q} \delta \leq 0.751 \delta. 
	\ee
 It follows that $\tilde \varphi''(h) >0$ on $[0, \delta a]$ and $\tilde \varphi'(\cdot)$ is a monotone increasing bijection from $[0,\delta a]$ onto $[0, \tilde \varphi'(\delta a)]$. Next we bound 
	\begin{align}
		|\tilde \varphi'(h) - h | &\leq h \sum_{k=3}^s \frac{1}{k-1} \Bigl(\frac h \Delta\Bigr)^{k-2}
			\leq h \Bigl( \frac{h/\Delta}{2} + \frac13 \frac{(h/\Delta)^2}{1- h /\Delta}\Bigr) = h \frac{h}{\Delta}\Bigl( \frac12 +\frac13 \frac{h/\Delta}{1- h /\Delta}\Bigr)  \notag \\
			& \leq h \delta q \Bigl( \frac12 + \frac13 \frac{q}{1-q}\Bigr) \leq \frac13 \delta h. \label{eq:tiltexist2}
	\end{align} 
	It follows in particular that $\tilde \varphi'(\delta a) \geq \frac23 \delta a$. Therefore, for $x$ to be in the image $[0,\tilde \varphi'(\delta a)]$ it is sufficient that $0\leq x \leq \frac23 \delta a$. Thus $\tilde \varphi'$ is a monotone increasing bijection from $[0,\delta a]$  onto an interval containing $\frac23 \delta a$. Specializing to $\delta=1$ we see that for every $x\in [0, \frac23 a]$ there is a unique $h\in [0,a]$ such that $\tilde \varphi'(h)=x$. If in addition $x=\frac23 \delta a $ with $\delta \in [0,1]$, then the solution $h$ must be in $[0,\delta a]$ and the proof is concluded with~\eqref{eq:tiltexist1} and~\eqref{eq:tiltexist2}.
\end{proof} 

\begin{lemma} \label{lem:phipositive}
We have	$\tilde \varphi(h)\geq 0$ for all $h\in [0,a]$.
\end{lemma}	

\begin{proof}
	We have 
	$$
		h\leq a =\sqrt{\frac{s}{4\e}}\leq \frac{\Delta}{\sqrt{2\e}} < \frac{\Delta}{2}
	$$
	and
	\begin{align*}
		\tilde \varphi(h) &= h^2 \Bigl( \frac{1}{2}+  \sum_{j=3}^s \frac{\kappa_j}{j!} h^{j-2}\Bigr) 
			\geq h^2 \Bigl(\frac{1}{2} - \sum_{j=3}^\infty \frac{1}{j(j-1)} \bigl(\frac{h}{\Delta}\bigr)^{j-2} \Bigr) \\
&				\geq h^2 \Bigl( \frac{1}{2}- \frac{1}{6}\frac{h/\Delta}{1- h /\Delta}\Bigr) = h^2 \frac{3- 4 h/\Delta}{6(1- h/\Delta)} \geq 0. \qedhere
	\end{align*}
\end{proof}

\begin{lemma} \label{lem:rest}
	For $0\leq h \leq \delta a$ with $\delta \leq 1$, we have 
	\[
		|\tilde r(h)|\leq \frac{(\delta/\sqrt{2})^{s+1}}{1-1/\sqrt 2}. 
	\] 
\end{lemma}

\begin{proof}
	By the definition of $\tilde r(h)$ in \eqref{def_r_tilde}, the bound~\eqref{eq:tildemk-klarge} yields 
	\[
		|\tilde r(h)|\leq \sum_{k=s+1}^\infty \frac{h^k}{(\sqrt 2 a)^k} \leq \frac{(\delta/\sqrt{2})^{s+1}}{1- \delta/\sqrt 2}.  \qedhere
	\] 
\end{proof}

\begin{lemma}\label{lem:decreasing}
	For $s\geq 30$ and $h \in [0,\delta a]$, the function $g_h$ given in \eqref{eq:def_g_h} is monotone increasing  on $[0,\infty)$ and satisfies $g_h(u)\geq 1$ for all $u\in \R_+$. 
\end{lemma}

\begin{proof} 
Since the assertion of the lemma is trivially true for $h=0$, we assume $h>0$. We start from
	$$
		g'_h(u) = h\sum_{k=0}^{s-1} \frac{(hu)^k}{k!} +2 u \tilde r(h)
			= h \Bigl( \exp_{s-1}(hu) +2 hu \frac{\tilde r(h)}{h^2}\Bigr).
	$$
	By the definition of $\tilde r(h)$ in \eqref{def_r_tilde} together with \eqref{eq:tildemk-klarge}, i.e.~$|\tilde m_k| \leq (\sqrt{2} a)^{-k}$ for $k \geq s$, and $h\in (0,\delta a]$
	\be \label{eq:rh2} 
		\frac{|\tilde r(h)|}{h^2} \leq \sum_{k=s+1}^\infty \frac{h^{k-2}}{(\sqrt 2 a)^k} 
		 \leq \frac{1}{2 a^2} \sum_{k=s-1}^\infty \frac{\delta^k}{\sqrt{2}^k} = \frac{1}{2a^2}\frac{1}{1- \delta/\sqrt{2}} \Bigl(\frac{\delta}{\sqrt{2}}\Bigr)^{s-1} \leq 2 \Bigl(\frac{\delta}{\sqrt{2}}\Bigr)^{s-1}. 
	\ee
	For the last inequality we have used that for $s\geq 30$, we have 
	$$
		2 a^2 \Bigl(1- \frac{\delta}{\sqrt{2}}\Bigr) \geq  \frac{s}{2\e} \Bigl( 1 - \frac{1}{\sqrt{2}}
			 \Bigr)\geq \frac12. 
	$$
	Therefore 
	$$
		\frac{|\tilde r(h)|}{h^2} \leq 2 \Bigl( \frac{1}{\sqrt{2}}\Bigr)^{s-1} \leq \frac12. 
	$$
	We conclude
	$$
		g'_h(u)\geq h \Bigl( \exp_{s-1}(hu) - hu \Bigr) = h\Bigl( 1+ \sum_{k=2}^{s-1} \frac{(hu)^{k}}{k!}	\Bigr) \geq 0
	$$
	and $g_h$ is monotone increasing. In particular, $g_h(u) \geq g_h(0) = 1$ for all $u\in \R_+$. 
\end{proof}

\subsection{Fourier transform of the tilted measure: a first bound} \label{sec:chafu1}
The Fourier transform of the signed measure $\mu_h$ is 
$$
	\chi_h(t)   =\e^{- \tilde \varphi(h)} \E\Bigl[ g_h(X) \e^{\mathrm i t (X-x)/\tilde \sigma_h}\Bigr].
$$
Motivated by Eq.~\eqref{eq:sichi}, we set 
$$
	\tilde \chi_h(t) := \exp\Bigl( \tilde \varphi(h+\mathrm i t /\tilde\sigma_h) - \tilde \varphi(h) - \mathrm i t x/\tilde \sigma_h\Bigr). 
$$
With the analogue of~\eqref{eq:egh} for $z= h +\mathrm i t /\tilde\sigma_h$ instead of $h$ we see 
\be \label{eq:chideco} 
	\chi_h(t)  = \tilde \chi_h(t) + \e^{- \tilde \varphi(h) - \mathrm i t x/\tilde \sigma(h)}\E\Bigl[ g_h(X) \e^{\mathrm i t X/\tilde\sigma(h)} - g_{h+\mathrm i t /\tilde\sigma_h} (X) \Bigr]. 
\ee
Eq.~\eqref{eq:chideco} is a substitute for Eq.~\eqref{eq:sichi}. 
The term $\tilde \chi_h(t)$ is easily treated with the Taylor series of $\tilde \varphi$. Remember the choice $s \leq 2 \Delta^2$, $a = \sqrt{s/(4\e)}\leq \Delta/\sqrt{2\e}$. 

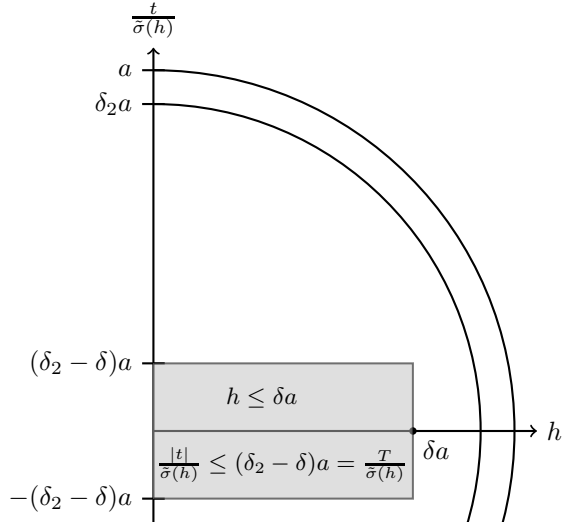
\begin{figure}
\begin{tikzpicture}[x=.3cm, y=.3cm, thick]
  
  \clip (-7,-4) rectangle (20,20);
  
  \draw[->] (0,0) -- (17,0) node[right] {$h$};
  \draw[->] (0,-17) -- (0,17) node[above] {$\frac{t}{\tilde{\sigma}(h)}$};
  \draw (0.5,16) -- (-0.5,16) node[left] {$a$};
  \draw (0,-16) arc (-90:90:16);

  \draw (0.5,14.5) -- (-0.5,14.5) node[left] {$\delta_2 a$};
  \draw (0,-14.5) arc (-90:90:14.5);

  \draw (0.5,3) -- (-0.5,3) node[left] {$(\delta_2-\delta )a$};
  \draw (0.5,-3) -- (-0.5,-3) node[left] {$-(\delta_2-\delta )a$};
\filldraw[black] (11.5,0) circle (1pt) node[anchor=north west] {$\delta a$};

 \filldraw[fill=lightgray, draw=black, opacity=0.5] (0,3) -- (11.5,3) -- (11.5,-3) -- (0,-3);

\node  at (4.8,1.5)   {\small $h \leq \delta a$};
\node  at (5.7,-1.5)   {\small $\frac{|t|}{\tilde{\sigma}(h)}\leq (\delta_2-\delta)a= \frac{T}{\tilde{\sigma}(h)}$};
\end{tikzpicture}
\caption{\label{fig:params}
A summary of parameter choices in Lemma~\ref{lem:chih1}.  The function $\tilde \varphi(z)$ is evaluated at $z= h + \mathrm i t / \tilde \sigma(h)$ with $0\leq h \leq \delta a$ and $|t| \leq (\delta_2 - \delta) a \tilde \sigma(h) = T$ (shaded rectangle), which is contained in the disk $|z|\leq \delta_2 a$. The parameter $a = \sqrt{s/(4\e)}$ enters the moment bound~\eqref{eq:tildemk-klarge}. }
\end{figure}

\begin{lemma} \label{lem:chih1}
	Assume that $s \leq 2 \Delta^2$. Fix $0<\delta <\delta_2<1$, set $a:= \sqrt{s/(4 \mathrm e)}$ and $T:=(\delta_2 - \delta) a \tilde \sigma(h)$, see Figure~\ref{fig:params}. Then for all $h\in [0,\delta a]$ and $t\in[-T,T]$ we have 
	$$
		\bigl|\tilde \chi_h(t) - \exp( - t^2/2) \big| \leq \frac{|t|}{T} \bigl(\e^{- t^2/4} - \e^{- t^2 /2}\bigr). 
	$$
\end{lemma} 

\begin{proof} 
	Let 
	$z:= h + \mathrm i t /\tilde \sigma(h)$. Then $|z|\leq \delta_2 a$. 
		A third order Taylor expansion of $\tilde \varphi$ yields, for some $\theta\in [0,1]$, 
	$$
	\tilde \varphi\bigl(h+\mathrm i \tfrac{t}{\tilde\sigma(h)} \bigr)- \tilde \varphi(h) - \mathrm i \tfrac{t x}{\tilde \sigma(h)} = - \frac12 t^2+ \frac{1}{3!} {\tilde \varphi}'''\bigl(h + \mathrm i \theta \tfrac{t}{\tilde\sigma(h)}\bigr) \bigl( \mathrm i \tfrac{t}{\tilde \sigma(h)} \bigr) ^3.
	$$
	In $|z|<\Delta$, the third derivative is bounded as 
	$$ 
		\Bigl|{\tilde \varphi}'''(z) \Bigr|  \leq \sum_{j=3}^s \frac{|\kappa_j|}{(j-3)!} |z|^{j-3}
			 \leq \sum_{j=3}^s \frac{(j-2)}{\Delta^{j-2}} |z|^{j-3} \leq \frac{\dd}{\dd u} \left. \Bigl(\sum_{k=0}^\infty \bigl( \frac u \Delta \bigr)^k\Bigr) \right|_{u=|z|} 
			  = \frac{1}{\Delta} \frac{1}{(1- |z|/\Delta)^2}.
	$$
	Set $L:= 3 \Delta (1- |z|/\Delta)^2 \tilde \sigma(h)^3$, then 
	$$
	\tilde \chi_h(t) = \exp\Bigl( - \frac12 t^2\Bigl(1 +\frac{\theta}{3(1- |z|/\Delta)^2 \tilde \sigma(h)^3} \frac{t}{\Delta }\Bigr)\Bigr)= \exp\Bigl( - \frac12 t^2\Bigl(1 +\theta  \frac{t }{L} \Bigr)\Bigr)
$$	
	for some $\theta \in \C$ with $|\theta|\leq 1$. Assuming that $|t|\leq T \leq L/2$,  we have 
	\begin{align*} 
		\Bigl|\e^{- \frac{1}{2}t^2 (1+\theta \frac{t}{L})} - \e^{- \frac{1}{2}t^2}\Bigr|
			&\leq \e^{- \frac{1}{2}t^2} \Bigl( \e^{|t|^3/(2L)} - 1 \Bigr) = \e^{- \frac{1}{2}t^2}\sum_{k=1}^\infty \frac{1}{k!} \Bigl( \frac{|t|^3}{2L}\Bigr)^k \\
			&= \e^{- \frac{1}{2}t^2} |t| \sum_{k=1}^\infty \frac{1}{k!} \Bigl( \frac{t^2}{4}\Bigr)^k \, \frac{ (4|t|)^{k-1}}{(2L)^{k-1}}  \frac{4}{2L} 
			 \leq \frac {|t|}{ T} \e^{- \frac{1}{2}t^2}\bigl( \e^{\frac{1}{4}t^2} -1\bigr) \\
			& = \frac {|t|}{ T} \bigl( \e^{- \frac{1}{4}t^2}- \e^{- \frac{1}{2}t^2}\bigr).			\end{align*} 
	It remains to check $T \leq L/2$. To that aim we use the lower bound for $\tilde{\sigma}(h)^2$ from Lemma \ref{lem:tiltexist} together with $|z| \leq \delta_2 a$ and we evaluate 
	\begin{align*}
		\frac{T}{L/2} & = \frac{(\delta_2- \delta) a \tilde \sigma(h)}{3 \Delta( 1- |z|/\Delta)^2\tilde \sigma(h)^3}	\leq \frac{\delta_2 - \delta}{1 - 0.751\, \delta} \, \frac{a/\Delta}{(1- \delta_2 a/\Delta)^2} 	\leq \frac{\delta_2 - \delta}{1 - 0.751\, \delta} <1.
	\end{align*} 
	The last inequality is equivalent to $\delta_2 <1 + (1- 0.751) \delta$, it holds true because 
	$\delta_2 <1$. 
\end{proof}

\begin{lemma} \label{lem:chierror1}
	Assume that $30\leq s \leq 2 \Delta^2$. Fix $\delta_2\in (0,1)$, set $a:= \sqrt{s/(4 \mathrm e)}$ and further assume that $a^{-1}\leq \delta_2^2$. Then for all $h\geq 0$ and $t \in \R$ with $|h+ \mathrm i \frac{t}{\tilde \sigma (h)} |\leq \delta_2 a$, we have 
	$$
		|\chi_h(t)  -\tilde \chi_h(t)|\leq 4 \sqrt{2}\, \frac{\delta_2^{\lfloor 4 a\rfloor}}{1-\delta_2}. 
	$$
\end{lemma} 

\begin{proof} We will use \eqref{eq:chideco}. 
	By the definition of $g_h$, we have
\begin{multline} \label{eq:ghfunctional-split}
	g_h(X) \e^{\mathrm i t X/\tilde\sigma_h} - g_{h+\mathrm i t /\tilde\sigma_h} (X)  \\
	=\bigl(\exp_s( hX)  \e^{\mathrm i t X/\tilde\sigma_h} - \exp_s((h+\mathrm i \tfrac{t}{\tilde\sigma_h}) X) 
\bigr) + X^2\bigl( \tilde r(h) \e^{\mathrm i t X/\tilde\sigma_h} - \tilde r(h+\mathrm i \tfrac{t}{\tilde\sigma_h})\bigr)
\end{multline} 
	By~\eqref{eq:chideco} we have to evaluate the expected value of~\eqref{eq:ghfunctional-split}. 	Set $z:= h + \mathrm i t /\tilde \sigma(h)$; notice $|h|\leq |z|$.  
	 With Lemma~\ref{lem:tildemk-bound}, the expectation of the second term in~\eqref{eq:ghfunctional-split} is bounded in absolute value by 
	\begin{align} \label{eq:i2}
		I_2&:= \E\left [X^2 \bigl| \tilde r(h) \e^{\mathrm i t X/\tilde\sigma_h} - \tilde r(h+\mathrm i t /\tilde\sigma_h)\bigr| \right]
		\leq \E[X^2] \bigl(| \tilde r(h) | +|\tilde r(h+\mathrm i t /\tilde\sigma_h)|\bigr)
		\nonumber\\
		& \leq 2 \sum_{k=s+1}^\infty \frac{1}{k!} |\tilde m_k|\, |z|^k\leq 2 \sum_{k=s+1}^ \infty \frac{(\delta_2 a)^k }{(\sqrt{2} a) ^k },
	\end{align}
	where we have used $\E[X^2] = 1$, $|z| \leq \delta_2 a$ and $\frac{1}{k!} |\tilde m_k| \leq \frac{1}{(\sqrt{2a})^k}$ for $k \geq s$ as derived in \eqref{eq:tildemk-klarge}. 
	For the first term in~\eqref{eq:ghfunctional-split}, we remember $\exp_s(y) = \exp(y) - \sum_{k=s+1}^\infty \frac{1}{k!} y^k$ and deduce 
	$$ 
	 \exp_s( hX)  \e^{\mathrm i t X/\tilde\sigma_h} - \exp_s((h+\mathrm i \tfrac{t}{\tilde\sigma_h}) X)  = - \sum_{k=s+1}^\infty \frac{X^k}{k!} \Bigl( \e^{\mathrm i t X/\tilde\sigma_h}  h^k - (h+\mathrm i \tfrac{t}{\tilde\sigma_h}) ^k \Bigr).
	$$
	The intuition is of course that this term should be small, however taking expected values we obtain moments $\E[X^k]$ with $k\geq s+1$, over which we have little control. Therefore we introduce an additional truncation parameter $b:= 4a>0$ and use the previous expression for $|X|\leq b$ only: we estimate, with the help of $k!\geq (k/\mathrm e)^k\geq (s/\e)^k$ for $k \geq s$ as well as $b/(s / \e)\leq 1/a$ and $|z|/a \leq \delta_2$ 
	\begin{align}
		I_1^{(1)} &:= \Bigl| \E\Bigl[ \1_{\{|X|\leq b\}}\Bigl(  \exp_s( hX)  \e^{\mathrm i t X/\tilde\sigma_h} - \exp_s((h+\mathrm i \tfrac{t}{\tilde\sigma_h}) X)\Bigr) \Bigr] \Bigr|
		\leq 2 \sum_{k=s+1}^\infty \frac{(b|z|)^k}{k!} \notag \\ 
		& \leq 2 \sum_{k=s+1}^\infty \frac{(b |z|)^k}{(s/\mathrm e)^k } \leq 2 \sum_{k=s+1}^\infty \delta_2^k. \label{eq:i11}
	\end{align}
	For $|X|>b$, we use instead $|\exp_s(y)| \leq \sum_{k=0}^s |y|^k/k!$, which yields 
	$$
		I_1^{(2)} := \Bigl| \E\Bigl[ \1_{\{|X|> b\}}\Bigl(  \exp_s( hX)  \e^{\mathrm i t X/\tilde\sigma_h} - \exp_s((h+\mathrm i \tfrac{t}{\tilde\sigma_h}) X)\Bigr) \Bigr] \Bigr|
		\leq 2 \sum_{k=0}^s \frac{|z|^k}{k!} \E\bigl[|X|^k\1_{\{|X|\geq b\}}\bigr].
	$$
Using the bound on the truncated moments provided by Lemma~\ref{lem:mkb-bound} and $|z|/a \leq \delta_2$, we can further estimate
	\be \label{eq:i12}
		\frac12 I_1^{(2)} \leq  \sum_{k=0}^{\lfloor b\rfloor } \sqrt{2}\frac{|z|^k}{a^{\lfloor b\rfloor}}  + \sum_{k=\lfloor b\rfloor+1}^{s}  \sqrt{2}\frac{|z|^k}{a^k}  \leq   \sum_{k=0}^{\lfloor b\rfloor } \sqrt{2}\frac{|z|^k}{a^{\lfloor b\rfloor}} + \sum_{k=\lfloor b\rfloor+1}^{s }  \sqrt{2}\delta_2^k.
	\ee
	The condition $a^{-1}\leq \delta_2^2$ allows us to treat the sum over $0\leq k\leq \lfloor b\rfloor$  using $|z|^k/a^{\lfloor b\rfloor} \leq \delta_2^k /a^{\lfloor b\rfloor -k} \leq \delta_2^{2\lfloor b\rfloor - k}$. We get
	$$
		I_1^{(2)} \leq 2 \sqrt{2}\Bigl(  \sum_{k=0}^{\lfloor b\rfloor} \delta_2^{2\lfloor b\rfloor - k} + \sum_{k=\lfloor b\rfloor+1}^s \delta_2^k\Bigr) 
				\leq 2\sqrt{2}\Bigl( \sum_{\ell=\lfloor b\rfloor}^{2 \lfloor b\rfloor} \delta_2^\ell + \sum_{k=\lfloor b\rfloor+1}^s \delta_2^k\Bigr)  \leq 4 \sqrt{2} \sum_{k=\lfloor b\rfloor}^{s} \delta_2^k,
	$$
	where we have used again $2b \leq s$. Altogether
	$$
		 I_1^{(2)}+I_1^{(1)}  + I_2 \leq 4 \sqrt{2} \sum_{k=\lfloor b\rfloor}^{s} \delta_2^k 
			+ 2 \sum_{k=s+1}^\infty \delta_2^k + 2\sum_{k=s+1}^\infty \Bigl(\frac{\delta_2}{\sqrt{2}}\Bigr)^k \leq 4 \sqrt{2}\, \frac{\delta_2^{\lfloor b\rfloor }}{1-\delta_2}. 
	$$
	The proof is concluded with $|\chi_h(t)- \tilde \chi_h(t)|\leq \e^{-\tilde \varphi(h)} (I_1^{(1)}+ I_1^{(2)}+ I_2)$ and the lower bound $\tilde \varphi(h)\geq 0$ from Lemma~\ref{lem:phipositive}.
\end{proof}

\begin{lemma} \label{lem:chierrormedium}
	Assume that $30\leq s \leq 2 \Delta^2$. Fix $0<\delta <\delta_2<1$, set $a:= \sqrt{s/(4 \mathrm e)}$ and $T:=(\delta_2 - \delta) a \tilde \sigma(h)$.  Further assume that $a^{-1}\leq \delta_2^2$. Then for all $h\in [0,\delta a]$ and $t\in \R$ with $|t|\leq T$, we have 
	$$
		\bigl|\chi_h(t ) - \exp( - t^2/2) \big| \leq \frac{|t|}{T} \bigl(\e^{- t^2/4} - \e^{-t^2 /2}\bigr) + 4 \sqrt{2}\, \frac{\delta_2^{\lfloor 4 a\rfloor}}{1-\delta_2}. 
	$$
\end{lemma} 

\begin{proof}
	The lemma is an immediate consequence from the Lemmas~\ref{lem:chih1} and~\ref{lem:chierror1}. 
\end{proof}

\noindent  We want to plug this bound into Zolotarev's lemma. The  first part $\tilde \chi_h(t)$ is treated exactly as it was for $s=\infty$. The second term is problematic: because of $|z|\geq |h|$, it does not go to zero as $t \to 0$ and the integral of the error estimate against $\dd t /t$  diverges. So another bound is needed for very small $t$.

\subsection{Fourier transform of the tilted measure: a second bound.} \label{sec:chafu2} 
An alternative bound for very small $t$ is derived with a Taylor expansion of the Fourier transform,  
\be \label{eq:claim2start}
	\chi_h(t) = 1 +\mathrm i t  \int_{\R} y\, \dd \mu_h(y) + \theta\, \frac{t^2}{2} \int_\R y^2 \dd |\mu_h|(y),
\ee
see  Eq.~(4.14) in~\cite[Chapter VI.4]{feller-vol2}. When $s=\infty$, the mean and second moment of the tilted normalized measure are $0$ and $1$, respectively. For finite $s$, we have instead the following lemma. 

\begin{lemma} \label{lem:tiltedmoments}
	Assume that $s$ is even and satisfies $30\leq s \leq 2 \Delta^2$. Further assume that 
 $x\in [0, \frac23 \delta a]$ with $\delta \in (0,1)$ and $a= \sqrt{s/(4 \mathrm e)}$.  Then 	
	$$ 
		\int_{\R} y\dd \mu_h (y) = \theta \, 0.0023\,  \delta^{s+1},\qquad 
		 \int_\R y^2 \dd |\mu_ h|(y)  = 	1 + \theta\, 0.07724\,\delta^{s+1}
	$$
\end{lemma}

\begin{proof} 
	First we check that the order of magnitude $O(\delta^{s+1})$ is correct. 	
	We evaluate 
	\be \label{eq:int1rep}
			\int_{\R} y\, \dd \mu_h (y)   = \e^{-\tilde \varphi(h)} \E\Bigl[  g_h(X) \frac{X-x}{\tilde \sigma (h)} \Bigr] = \tilde \sigma(h)^{-1}  \Bigl( \e^{-\tilde \varphi(h)}  \E\bigl[X g_h(X) \bigr] - x \Bigr).
	\ee
	By the definition of $g_h$ and in view of $\E[X^k]=m_k=\tilde m_k$ for $k\leq s+2$ (see Lemma~\ref{lem:tildemk-bound}), we have 
	\begin{align*} 
		\E\bigl[X g_h(X) \bigr]  & = \sum_{j=0}^{s} \frac{h^j}{j!} m_{j+1} + m_3 \tilde r(h) = \frac{\dd}{\dd h} \e^{\tilde \varphi(h)} - \sum_{j=s+1}^\infty \frac{h^{j}}{j!} \tilde m_{j+1} +m_3 \tilde r(h). 
	\end{align*} 
	In view of $\tilde \varphi'(h) = x$, the first term is cancelled by $-x$ in~\eqref{eq:int1rep}, we find 
	\be \label{eq:tiltmom1}
		\int_{\R} y\, \dd \mu_h (y) = \tilde \sigma(h)^{-1}  \e^{-\tilde \varphi(h)} \Bigl( - \sum_{j=s+1}^\infty \frac{h^{j}}{j!} \tilde m_{j+1} + m_3 \tilde r(h)\Bigr) =O(\delta^{s+1}),
	\ee
	where we have used~\eqref{eq:tildemk-klarge} and $h\leq \delta a$. 
	For the second moment, we need to evaluate 
	$$
		\int_{\R} y^2\, \dd |\mu_h| (y) = \e^{-\tilde \varphi(h)} \E\Bigl[  |g_h(X)|\, \Bigl( \frac{X-x}{\tilde \sigma (h)}\Bigr)^2 \Bigr].
	$$
	Since $s$ is even, the truncated exponential is positive by  Lemma~\ref{lem:truncated-exp} and the triangle inequality yields 
	$$
		\exp_s(u) -  u^2 |\tilde r(h)|	 \leq |g_h(u)| \leq \exp_s(u) + u^2 |\tilde r(h)|\qquad (u\in \R)
	$$
	hence 
	\be \label{eq:tiltmom2}
			J_1- J_2 \leq \int_{\R} y^2\, \dd |\mu_h| (y)\leq J_1+J_2
	\ee
	with
	$$
		J_1 := \frac{\exp(-\tilde \varphi(h))}{\tilde \sigma(h)^2}  
			 \E\Bigl[ \exp_s(hX) (X-x)^2\Bigr],\quad
		J_2:= \frac{\exp(-\tilde \varphi(h))}{\tilde \sigma(h)^2}  \E\Bigl[  X^2 |\tilde r(h)|\, (X-x)^2 \Bigr].
	$$ 
	We evaluate
	\begin{align*}
		J_1
		  & = \tilde \sigma(h)^{-2} \e^{-\tilde \varphi(h)}\E\Bigl[  \exp_s(hX) (X^2 - 2 x X + x^2) \Bigr]  \\
		  & = \tilde \sigma(h)^{-2} \e^{-\tilde \varphi(h)} \sum_{k=0}^s \frac{h^k}{k!} (m_{k+2} - 2 x m_{k+1} + x^2 m_k) \\
		  & = \tilde \sigma(h)^{-2} \e^{-\tilde \varphi(h)} \Bigl\{\Bigl( (\e^{\tilde \varphi})''(h) - 2 x ( \e^{\tilde \varphi})'(h) + x^2 \e^{\tilde \varphi(h)} \Bigr) -   \sum_{k=s+1}^\infty \frac{h^k}{k!} \bigl(\tilde m_{k+2} - 2 x \tilde m_{k+1} + x^2 \tilde m_k \bigr) \Bigr\}.  \
	\end{align*} 
	The terms involving $\tilde \varphi(h)$ and its derivatives combine to 
	$$
			 \tilde \sigma(h)^{-2} \e^{-\tilde \varphi(h)} \Bigl( (\e^{\tilde \varphi})''(h)- 2 x (\e^{\tilde \varphi})'(h) + x^2\e^{\tilde \varphi(h)}\Bigr) = \frac{1}{\tilde \sigma^2(h)}\bigl( \tilde \varphi''(h) + \tilde \varphi'(h)^2 - 2 x \tilde \varphi'(h) + x^2\bigr) = 1
	$$
	(remember $x = \tilde \varphi'(h)$ and $\tilde \sigma^2(h) = \tilde \varphi''(h)$), consequently 
	\be \label{eq:tiltmom3}
		J_1  =1-  \tilde \sigma(h)^{-2} \e^{-\tilde \varphi(h)} \sum_{k=s+1}^\infty \frac{h^k}{k!} \bigl(\tilde m_{k+2} - 2 x \tilde m_{k+1} + x^2 \tilde m_k \bigr) = 1+ O(\delta^{s+1}).
	\ee
	For $J_2$, we evaluate by Lemma \ref{lem:rest}
	\be \label{eq:tiltmom4}
		J_2 = \tilde \sigma(h)^{-2} \e^{-\tilde \varphi(h)} |\tilde r(h)| (m_4 - 2 x m_3+ x^2) = O(\delta^{s+1}). 
	\ee
	It follows that $\int_\R y^ 2 \dd |\mu_h|(y) = 1+ O(\delta^{s+1})$. 
	
	Next we bound the constants in front of $\delta^{s+1}$. 	 By Lemmas~\ref{lem:tiltexist} and~\ref{lem:phipositive}, we have $\tilde \sigma(h) \geq 1- 0.751=0.249$ and $\tilde \varphi(h)\geq 0$. Consequently 
	\be\label{eq:tildesigmaexp}
 0<	\tilde \sigma(h)^{-1}\exp(- \tilde\varphi(h))\leq 2.0041,\quad 0<	\tilde \sigma(h)^{-2}\exp(- \tilde\varphi(h)) \leq 4.017.
	\ee
	Moreover, we can estimate the third and fourth moment by 
	$$|m_3|=|\kappa_3|\leq \frac1\Delta \leq \frac{1}{a \sqrt{2\e}}$$ 
	by the cumulant bound and $\Delta^2\geq \frac{s}{2} = 2\e\,  a^2 $
	and 
	$$|m_4|= |\kappa_4+3\kappa_2^2|\leq \frac{2}{\Delta^2} +3 \leq 3 +\frac{2}{15} $$
	for $\Delta^2\geq \frac s2 \geq 15$. 
	For $\tilde r(h)$ and the power series in $h$, we use~\eqref{eq:tildemk-klarge} and $h\leq \delta a$, set $u=h /(\sqrt{2}a )$ and obtain 
	\begin{align*}
		|\tilde r(h)| &\leq \frac{u^{s+1}}{1- u},\\
		 \sum_{k=s+1}^\infty \frac{h^{k}}{k!} |\tilde m_{k+1}|
		 	& \leq  \frac{1}{\sqrt{2} a} \frac{\dd }{\dd u} 	 \sum_{k=s+2}^\infty u^k = \frac{u^{s+1}}{\sqrt{2} a}\Bigl( \frac{(s+2)}{1-u}+ \frac{u}{(1-u)^2}\Bigr),\\
 		 \sum_{k=s+1}^\infty \frac{h^{k}}{k!} |\tilde m_{k+2}|
	 	& \leq  \frac{1}{2 a^2} \frac{\dd^2 }{\dd u^2} 	 \sum_{k=s+3}^\infty u^k 
	 		= \frac{u^{s+1}}{2a^2} \Bigl(\frac{(s+3)(s+2)}{1-u} + 2\frac{(s+3)u}{(1-u)^2 }+2 \frac{u^{2}}{(1-u)^3}\Bigr).
	\end{align*}
	We use the notation 
$$U_j:= \sup_{u \in [0, \frac{1}{\sqrt{2}}]}\frac{u^{j-1}}{(1-u)^j} = \begin{cases}
2+\sqrt{2}  
, & j=1 \\
4+3\sqrt{2}  
,& j=2 \\
10+7\sqrt{2}  
,& j=3
\end{cases},
$$	
where the latter representation follows from monotonicity.	
	With this notation we have, in view of \eqref{eq:tiltmom1}, \eqref{eq:tiltmom4}, \eqref{eq:tiltmom3} together with $u^{s+1} \leq (\delta/\sqrt{2})^{s+1}$, \eqref{eq:tildesigmaexp} and $x \leq \frac{2}{3}a=\frac{2}{3} \sqrt{\frac{s}{4e}} \leq \frac{\sqrt{2}}{3\sqrt{e}}\Delta $ resp. $x^2 \leq \frac{s}{9e}$
\begin{align*}
\Bigl|\int_{\R} y\, \dd \mu_h (y)\Bigr| &\leq 2.0041 (\delta/\sqrt{2})^{s+1}\left( \frac{1}{\sqrt{2}a}  ((s+2)U_1 + U_2) + \frac{1}{a\sqrt{2e}} U_1\right),
\\
J_2 &\leq  4.017 (\delta/\sqrt{2})^{s+1} U_1 (3+\frac{2}{15} +2 x \frac{1}{\Delta} + x^2),
\\
&\leq  4.017 (\delta/\sqrt{2})^{s+1} U_1 (3+\frac{2}{15}+\frac{2\sqrt{2}}{3\sqrt{e}} + \frac{s}{9e})
\\
|J_1  -1| & 
\leq 4.017  (\delta/\sqrt{2})^{s+1} \Bigl ( \frac{1}{2a^2} (U_1 (s+3)(s+2) + 2 (s+3) U_2 + 2U_3)\\
 &\qquad \qquad \qquad +  \frac{4}{3\sqrt{2} } ( (s+2)U_1 + U_2)+ \frac{s}{9e} U_1 \Bigr).
\end{align*}
 From $a= \sqrt{\frac{s}{4e}} \geq \sqrt{\frac{30}{4e}}$ together with the representations of $U_j, j=1,2,3$ and estimate of the form $\frac{s+2}{\sqrt{2}^{s+1}} \leq \frac{32}{\sqrt{2}^{31}}$ resp. $\frac{(s+3)(s+2)}{\sqrt{2}^{s+1}} \leq \frac{33 \cdot 32}{\sqrt{2}^{31}}$ we immediately get 

\begin{equation*}
\Bigl|\int_{\R} y\, \dd \mu_h (y)\Bigr| \leq \delta^{s+1} 0.0023
\end{equation*}
\begin{equation*}
J_2 \leq   \delta^{s+1}  0.0015
\end{equation*}
\begin{equation*}
|J_1  -1| 
\leq   \delta^{s+1}  0.07577
\end{equation*}
and hence 
	\begin{equation*}
		\Bigl|\int_\R x^2 \dd |\mu_h|(x) - 1\Bigr|\leq |J_1 - 1|+ |J_2| \leq    0.07724 \delta^{s+1}. \qedhere
	\end{equation*}
\end{proof}

\begin{lemma} \label{lem:chierror2}
	Under the assumptions of Lemma~\ref{lem:tiltedmoments}, we have 
	$$
		\bigl|\chi_h(t) - \exp( - \tfrac12 t^2)\bigr| \leq \delta^{s+1} \,  0.0023 \, |t|+ (1+ 0.0387\, \delta^{s+1}) t^2
	$$
	for all $t\in \R$. 
\end{lemma} 

\begin{proof} 
	We combine Eq.~\eqref{eq:claim2start}, Lemma~\ref{lem:tiltedmoments} and $|\exp( - \frac12 t^2) - 1 |\leq \frac12 t^2$ 
	\begin{align*} 
		\bigl| \chi_h(t) - \exp( -\tfrac12 t^2)\bigr| &= \Bigl| 1  +  \theta t \, 0.0023\,  \delta^{s+1} + \theta \frac{t^2}{2} \bigl( 	1 + \theta\, 0.0773\,\delta^{s+1} \bigr) - \e^{- t^2/2}\Bigr| \\
		&\leq 0.0023 \, \delta^{s+1} |t| + \frac{t^2}{2} (1+0.0773 \, \delta^{s+1}) + \frac{t^2}{2}. \qedhere
	\end{align*} 
\end{proof} 

\noindent Combining Lemmas~\ref{lem:chierrormedium} and~\ref{lem:chierror2}, we get
\be \label{eq:fourierkey}
	\bigl|\chi_h(t) - \exp( - \tfrac12 t^2)\bigr| \leq \min \Bigl(\tfrac{t}{T}\bigl( \e^{-t^2/4} - \e^{-t^2/2}\bigr) + 4 \sqrt 2\frac{\delta_2^{\lfloor 4a \rfloor}}{1-\delta_2},  \delta^{s+1} \,  0.0023 \, |t|+ (1+ 0.0387\, \delta^{s+1}) t^2.
\ee
Notice that this upper bound is integrable against $\dd t / t$. The bound is valid if the assumptions of both lemmas are met. For future reference we summarize these assumptions, see Figure~\ref{fig:params}.

\begin{assumption} \label{assum} 
	The quantities $s\in \N$, $x\geq 0$, $\delta,\delta_2\in [0,1]$ satisfy the following:
	\begin{itemize} 
		\item $30 \leq s \leq 2 \Delta^2$ and $s$ is even. We set $a:= \sqrt{\frac{s}{4\e}}$.
		\item $x\in [0,\frac23 \delta a]$. 
		\item $0<\delta < \delta_2< 1$ and  $a^{-1}\leq \delta_2^ 2$. We set $T:= (\delta_2- \delta) a \tilde \sigma (h)$, with $h \in [0, \delta a]$ the solution of $\tilde \varphi'(h) = x$. 
		\item Fourier transforms are evaluated at $t\in [-T,T]$. 
	\end{itemize} 
\end{assumption} 

\subsection{Normal approximation for the tilted measure} \label{sec:normaltilt}

Armed with the bound~\eqref{eq:fourierkey} and Lemma~\ref{lem:zolotarev2} we can bound the Kolmogorov distance between the normal law and the tilted signed measure $\mu_h$. Set 
\[
	D_h:= \sup_{y\in \R}\bigl|\mu_h\bigl((-\infty,y]\bigr) - \P(Z\leq y)\bigr|.
\]

\begin{lemma} \label{lem:kolmodist1}
	Fix $\delta \in (0,1)$. Then, there exists $C(\delta) >0$ and $s_\delta \in \N$ such that for all $h\in (0,\delta a)$ and all even $s\geq s_\delta$ with $s\leq 2 \Delta^2$, we have 
	\be \label{eq:dhs}
		D_h \leq \frac{C(\delta)}{\sqrt s}.
	\ee
\end{lemma} 

\begin{proof} 
	Fix $\delta_2\in (\delta, 1)$. Clearly for sufficiently large $s$, the condition $\delta_2^2 > a^{-1}$ in Assumption~\ref{assum} is satisfied. We apply Lemma~\ref{lem:zolotarev2} to $\nu = \mu_h$ (the signed tilted measure), $\mu = \mathcal N(0,1)$ the standard normal distribution, $q = 1/\sqrt{2\pi}$, and $\eps =1/T$ with $T$ of the order of $\sqrt s$, which gives 
	\be \label{eq:kolmodist1}
		D_h \leq  \frac{C_1}{\sqrt{2\pi}} \frac{1}{T} +  \frac{C_2}{\pi}\int_{0}^T \Bigl(1- \frac{t}{T}\Bigr)\bigl |\chi_h(t) - \e^{- t^2/2}\bigr| \frac{\dd t}{t} + C_3 \mu_h^-(\R)
	\ee
	with $C_1$ and $C_2$ the numerical constants defined in Eq.~\eqref{eq:numtent} and 
	\[
		C_3 = \frac{C_1}{3.55} \approx 1.282.
	\]	
	The third term in~\eqref{eq:kolmodist1} is missing in~\cite{SS91} because of an erroneous application of Zolotarev's  lemma, 
see the remark after Lemma~\ref{lem:zolotarev2}.

	For $T$ of the order of $\sqrt{s}$, clearly the first term on the right-hand side of~\eqref{eq:kolmodist1} is of the order of $1/\sqrt{s}$.
	The measure $\mu_h$ was defined in~\eqref{eq:mod-tilt}, its negative part is given by
	\[
		\mu_h^-(\R) = \e^{- \tilde \varphi(h)} \E\bigl[ (g_h)_-(X) \Bigr] \leq \e^{-\tilde \varphi(h)} \E\bigl[ X^2 \tilde r(h)\bigr] = \e^{-\tilde \varphi(h)} \tilde r(h),
	\]
	where we used the positivity of the truncated exponential shown in Lemma \ref{lem:truncated-exp}.  	
	Lemma~\ref{lem:phipositive} and~\ref{lem:rest} give 
	\be \label{eq:kolmodist2}
		\mu_h^-(\R) \leq \frac{(\delta/\sqrt 2)^{s+1}}{1- 1/\sqrt 2} = O\Bigl( \Bigl(\frac{1}{\sqrt{2}}\Bigr)^s\Bigr),
	\ee
	which decays exponentially fast as $s\to \infty$ and is negligible compared to $1/\sqrt s$. 
	
	In the integrand in the second term in~\eqref{eq:kolmodist1} we bound $(1- |t|/T) \leq 1$, split the domain of integration into two subintervals $[0,t_0]$ and $[t_0,T]$ for some $t_0>0$, and apply~\eqref{eq:fourierkey}. This yields two contributions. The first one is bounded by
	$$
		\int_{0}^{t_0} \bigl|\chi_h(t) - \e^{- t^2/2}\bigr| \frac{\dd t}{t} 
			\leq \int_0^{t_0} \Bigl( O(\delta^{s+1}) + (1+ O(\delta^{s+1})) t\Bigr) \dd t  = t_0 O(\delta^{s+1}) + \frac{t_0^2}{2} (1+ O(\delta^{s+1})).
	$$
Using Lemma \ref{lem:chierrormedium} we find that the second contribution is bounded by 
	$$
		\int_{t_0}^{T} \bigl|\chi_h(t) - \e^{- t^2/2}\bigr| \frac{\dd t}{t} 
			\leq \frac 2T \int_{t_0}^\infty \e^{-t^2/4} \dd t  + \int_{t_0}^T O(\delta_2^{\lfloor 4a \rfloor}) \frac{\dd t}{t} \leq \frac{C}{T}+ O(\delta_2^{\lfloor 4a \rfloor}) \log \frac{T}{t_0}.
	$$
	Choose $t_0 = \delta_2^{\lfloor 4a\rfloor /2}$, then we get in the limit $s\to \infty$ at fixed $\delta_2$
	\be \label{eq:kolmodist3}
		\int_0^T  \bigl|\chi_h(t) - \e^{- t^2/2}\bigr| \frac{\dd t}{t}  
			= O\Bigl(\frac{1}{\sqrt s}\Bigr) + O\bigl(\delta_2^{\lfloor 4a \rfloor}\bigl(s+ \log \delta_2 + \log s)\bigr) = O\Bigl(\frac{1}{\sqrt s}\Bigr).
	\ee
	We combine Eqs.~\eqref{eq:kolmodist1}, \eqref{eq:kolmodist2}, \eqref{eq:kolmodist3} and arrive at $D_h = O(1/\sqrt s)$. Precise bounds depend on $\delta$ and $\delta_2$ but not on the exact value of $h\in (0,\delta)$ or $x= \tilde \varphi'(h)$ and the proof is complete. 
\end{proof}

\noindent Lemma~\ref{lem:kolmodist1} is all that is needed for the proof of Theorem~\ref{thm:lemma22}. For the reader interested in numerical values, we add two further bounds, following~\cite[Eqs.~(2.48) and (2.55)]{SS91}. Set 
\be \label{eq:t0}
		t_0^2 := \frac
{4 \sqrt{2} \delta_2^{\lfloor 4 a\rfloor}}{1-\delta_2}.
\ee
The main idea for the following lemma is then to split the domain of integration $[0, T]$ into the intervals $[0, t_0]$ and $[t_0, T]$. 
\begin{lemma} \label{lem:kolmodist2}
	Under Assumption~\ref{assum}, we have
	\be \label{eq:kolmo24}
	 	\int_{0}^T\bigl|\chi_h(t) - \e^{- t^2/2}\bigr| \frac{\dd t}{t} 
		\leq  \frac{\sqrt{2\pi}}{T} \frac{\sqrt 2 - 1}{2} + t_0^2 \Bigl( 0.5203 + \max (\log \frac{T}{t_0}, 0) \Bigr).
	 \ee
\end{lemma} 

\begin{proof} 
	We refine the bounds from the proof of Lemma~\ref{lem:kolmodist1} as in~\cite[Eq.~(2.46)]{SS91}. 
	 The fraction with $\delta_2$, which we denote by $t_0$ according to \eqref{eq:t0}, appeared in the error bound on Fourier transforms in Lemma~\ref{lem:chierror1}. The bound~\eqref{eq:fourierkey} yields 
	 \be \label{eq:kolmo21}
	 	\int_{0}^{t_0} \bigl|\chi_h(t) - \e^{- t^2/2}\bigr| \frac{\dd t}{t} 
		\leq 0.0023 \, \delta^{s+1}  t_0 + (1+ 0.0387\, \delta^{s+1}) \frac{t_0^2}{2}. 
	 \ee
	 and, if $T\geq t_0$,
	 \begin{align}
 		\int_{t_0}^{T} \bigl|\chi_h(t) - \e^{- t^2/2}\bigr| \frac{\dd t}{t} & \leq 
				\frac 1 T \int_{t_0}^T (\e^{-t^2/4} - \e^{- t^2/2})\, \dd t + t_0^2 \int_{t_0}^T \frac{\dd t}{t} \notag 	\\
				& \leq \frac{\sqrt{2\pi}}{T} \frac{\sqrt 2 - 1}{2} + t_0^2 \log \frac{T}{t_0}. \label{eq:kolmo22}
	 \end{align}
	 To that aim we note, first, that, because of $\delta< \delta_2<1$ and $4a \leq 2 \sqrt s \leq s$, we have 
	 \[
	 	\frac{\delta^{s+1}}{t_0} \leq \delta_2^{s+1-\lfloor 4a\rfloor} \sqrt{\frac{1-\delta_2}{4\sqrt 2}} \leq 2^{-5/4}
	 \] 
	 Therefore~\eqref{eq:kolmo21} is bounded by
	 \be\label{eq:kolmo23}
	 	t_0^2 \bigl( 0.0023\times 2^{-5/4} + \frac{1.0387}{2}) \leq t_0^2\,  0.5203.
	 \ee
	If $T\geq t_0$, the inequality~\eqref{eq:kolmo24} follows by adding up~\eqref{eq:kolmo21} and~\eqref{eq:kolmo22} and combining with~\eqref{eq:kolmo23}. If $T\leq t_0$, then the integral from zero to $T$ is bounded by the integral from zero to $t_0$ and~\eqref{eq:kolmo21} holds true because the right-hand side of~\eqref{eq:kolmo24} is larger than the right-hand side of~\eqref{eq:kolmo21}.
	 We insert the bound into~\eqref{eq:kolmodist1}, combine with~\eqref{eq:kolmodist2}, and obtain the lemma. 
\end{proof}

\begin{lemma} \label{lem:kolmodist3} 
	Let $s$ be an even integer with $30\leq s \leq 2 \Delta^2$. Fix $\delta \in [0,1)$ and $x\in [0,\frac23 \delta a]$ and define $h\in [0,\delta a]$ by $\tilde \varphi'(h) = x$. 
	Define 
	\[
		\delta_2 := 1- \frac{1- \delta}{2 s^{1/4}}.
	\] 
	Then Assumption~\ref{assum} is satisfied and 
	\[
		D_h \leq \frac{22.52}{(1-\delta)\sqrt{s}} \Bigl( 0.8+ s \exp\Bigl( - \frac12 (1-\delta) s^{1/4}\Bigr)\Bigr).
	\] 
\end{lemma} 

\begin{proof} 
	The only conditions to be checked in Assumption~\ref{assum} are $\delta_2^2\geq a^{-1}$ and $0< \delta < \delta_2 <1$. The second inequality follows from 
	\[
		0 < 1- \delta_2 = \frac{1- \delta}{2 s^{1/4}} < 1 -\delta.
	\] 
	The first inequality holds true because 
	\[
		\sqrt a\, \delta_2 \geq \Bigl(\frac{s}{4 \e}\Bigr)^{1/4} \Bigl( 1 - \frac{1}{2 s^{1/4}}\Bigr) 
			= (4 \e)^{-1/4}\bigl( s^{1/4} - 0.5\bigr). 
	\]
	For $s=30$, the last expression is approximately equal to 1.0135 hence in particular larger than $1$, which by monotonicity extends to all of $s\geq 30$. 
	Thus  Assumption~\ref{assum} holds true and we may apply the bounds from Lemma~\ref{lem:kolmodist1} and Lemma~\ref{lem:kolmodist2}. 
	
	We want to extract from~\eqref{eq:kolmodist1} and~\eqref{eq:kolmo24} a bound expressed directly in terms of $s$, $\delta$, and numerical constants. In order to bound $1/T$ with $T= (\delta_2 - \delta) a \tilde \sigma(h)$, we remember the bound $\tilde \sigma(h)^2 = \tilde \varphi''(h) \geq 1 - 0.751$ from Lemma~\ref{lem:tiltexist} and bound
	\be \label{eq:TTbound}
		T \geq \sqrt{0.249} \sqrt{\frac{s}{4 \e}} \Bigl(1- \frac{1}{2 s^{1/4}}\Bigr) (1-\delta)
		\geq 0.118\, (1-\delta) \sqrt s 
	\ee
	for $s\geq 30$. 
	In order to bound $t_0^2$ defined in~\eqref{eq:t0}, we note, first, $4 a = 2 \sqrt{s/\e}\geq \sqrt s$, hence $\lfloor 4 a\rfloor \geq \sqrt s - 1$ and, second, 
	\[
		\delta_2^{\sqrt s} \leq \exp\Bigl( - (1- \delta) s^{1/4} /2\Bigr),
	\] 
	where we have used the standard inequality $1- x\leq \exp( - x)$. As a consequence, 
	\[
		t_0^2 \leq \frac{4 \sqrt 2}{\delta_2 (1- \delta_2)} \delta_2^{\sqrt s} 
			\leq \frac{4 \sqrt 2}{1 - 1/(2 s^{1/4})} \frac{2 s^{1/4}}{1- \delta}  \exp\Bigl( - (1- \delta) s^{1/4} /2\Bigr).
	\] 
	It remains to take care of $t_0^2\max(\log(T/t_0),0)$ in~\eqref{eq:kolmo24}. We bound $\log (T/t_0)$ by a constant times $s^{1/4}$.  In view of Lemma~\ref {lem:tiltexist}  and $\delta_2 - \delta \leq 1$ we have 
	\[
		\log T \leq \log (a \tilde{\sigma}(h))\leq \frac12 \log 1.751 + \frac12 \log \frac{s}{4 \e}.
	\]
	It is easily checked that $x^{-1}\log x$ is maximal at $x=\e$ hence $\log x \leq x/\e$ 
   and $\log x = 4 \log x^{1/4} \leq 4 x^{1/4}/ \e$, therefore 
	\[
		\log T \leq \frac12 \log \Bigl( \frac{1.751}{4 \e} s\Bigr) \leq \frac{2}{\e} \Bigl( \frac{1.751}{4 \e} \Bigr)^{1/4} s^{1/4} \leq 0.47\, s^{1/4}. 
	\]	
	Furthermore, by the definition of $t_0$
	\[
		- \log t_0 \leq \frac12\Bigl( - 4 a \log \delta_2 - \log (4 \sqrt 2) +\log (1-\delta_2)\Bigr) \leq - 2 a \log \delta_2 = - \sqrt{\frac{s}{\e}}\, \log\delta_2.  
	\] 	
	It is easily checked that $x(\log x - 1)$ is minimal at $x=1$, from which one deduces the inequality $\log x \geq 1 - 1/x$. We apply the inequality to $\log \delta_2$, use $(1- 1/(2s^{1/4}))^{-1}\leq 1.272$ for $s\geq 30$, and find 
	\[
		- \log \delta_2 \leq \frac{1 - \delta_2}{\delta_2} \leq \frac{1-\delta}{2 s^{1/4}} \frac{1}{1 - 1/(2s^{1/4})} \leq 0.64\, s^{-1/4}. 
	\] 
	It follows that 
	\[
		\log T - \log t_0 \leq \Bigl( 0.47 + \frac{0.64}{\sqrt{\e}}\Bigr) s^{1/4}\leq 0.86\, s^{1/4}
	\]		
	and for $s\geq 30$
	\[
		0.5203+ \max\Bigl(\log \frac{T}{t_0}, 0\Bigr) 
		\leq 1.0805\, s^{1/4}. 
	\]	
	We insert this bound into~\eqref{eq:kolmo24}, combine with~\eqref{eq:TTbound} and~\eqref{eq:kolmodist1} for $C_1$ and $C_2$ defined by \eqref{eq:numtent}, and find that $D_h\leq A + B + C$ with 
	\begin{align*}
		A &= \frac 1 T \Bigl( \frac{C_1}{\sqrt {2\pi}}+ \frac{C_2}{\pi} \sqrt{2\pi}\frac{\sqrt 2 - 1}{2}\Bigr) \leq \frac{1}{(1-\delta)\sqrt s} \frac{1}{0.118}\Bigl( \frac{4.551}{\sqrt{2\pi}} + \frac{1.564}{\sqrt{2\pi}} (\sqrt 2 - 1)\Bigr) \leq \frac{18}{(1-\delta)\sqrt s}. 
	\end{align*} 
	and 
	\begin{align*}
		B & = \frac{C_2}{\pi} t_0^2 \times 1.0805\, s^{1/4} \\
		  &\leq \frac{4.551}{\pi}\times 1.0805 \times \frac{8 \sqrt 2}{1- 1/(2\times 30^{1/4})}\, \frac{\sqrt s}{1-\delta}\exp\Bigl( - \frac12 (1-\delta) s^{1/4}\Bigr) \\
		  &\leq 22.5193\, \frac{\sqrt s}{1-\delta}\exp\Bigl( - \frac12 (1-\delta) s^{1/4}\Bigr).
	\end{align*}
	and finally for $s\geq 30$
	\[
		C = C_3 \mu_h^-(\R) \leq 4.38 \Bigl( \frac{\delta}{\sqrt 2}\Bigr)^{s+1}\leq 10^{-4}\, \delta^{s+1}. 
	\] 
	Notice 
	\[
		\delta^{s+1} = \delta \bigl( 1-(1-\delta)\bigr)^s \leq \e^{- (1-\delta) s}\leq \sqrt{s}\,\e^{- (1-\delta)s^{1/4}/2}.
	\]
	The lemma follows by adding up the bounds for $A$, $B$, and $C$.
\end{proof} 

\subsection{Undoing the tilt}
The relation~\eqref{eq:mod-tilt-inversion} yields
$$
	\P(X\geq x) = \e^{\tilde \varphi(h)} \int_{0}^\infty g_h(\tilde \sigma(h) y + x)^{-1} \dd\mu_h(y),
$$
which we want to bring into the form 
$$
\P(X\geq x) = \e^{\tilde \varphi(h) - h x} \Bigl(\E\bigl[\e^{-h \tilde \sigma(h) Z}\1_{\{Z\geq 0\}}\bigr]+ \text{error term} \Bigr). 
$$
Let 
\begin{align*}
	I_1(h) &:= \int_0^\infty g_h(\tilde \sigma(h) y + x)^{-1} \dd\mu_h(y) -
	\E\Bigl[g_h(\tilde \sigma(h) Z + x)^{-1} \1_{\{Z\geq 0\}}\Bigr],\\
		I_2(h)& :=  \E\bigl[  g_h( x + \tilde \sigma (h) Z)^{-1} \1_{\{Z\geq 0\}}\bigr]
\end{align*}
so that 
\be \label{eq:undo}
	\P(X\geq x) = \e^{\tilde \varphi(h)}\bigl( I_1(h) + I_2(h)\bigr).
\ee
Lemma~\ref{lem:I1h} below says that $I_1(h)$ is small, Lemma~\ref{lem:I2h} says that $I_2(h)$ is approximately equal to $\E [ \exp( - h x - h \tilde \sigma(h) Z) \1_{\{Z\geq 0\}}]$.

\begin{lemma} \label{lem:I1h}
	For $s \geq 30$, $h \in [0,\delta a]$ and $x < \sqrt s/ (3 \sqrt \e)$
	let $D_h$ be the Kolmogorov distance between $\mu_h$ and the normal distribution $\mathcal N(0,1)$. Then 
	$$
		|I_1(h)|\leq 2.0004\, D_h \exp( - h x).
	$$
\end{lemma} 

\begin{proof} 	
	We apply Lemma~\ref{lem:kolmo-mono} to $f(y)=1/g_h(\tilde \sigma(h) y + x)$, which is monotone decreasing by Lemma \ref{lem:decreasing}, and get 
	$$
		|I_1(h)|\leq 2  g_h(x)^{-1} \sup_{y\in \R}\bigl|\mu_h\bigl((-\infty,y]\bigr) - \P(Z\leq y)\bigr|
			= 2 g_h(x)^{-1} D_h.
	$$
	We  have for all $u\geq 0$ and suitable $\theta = \theta_{h,u}\in [-1,1]$
	$$
		g_h(x) = \exp_s(hx) + \frac12 (h x)^2 \frac{2\tilde r(h)}{h^2} =\Bigl(1+ \theta  \frac{2\tilde r(h)}{h^2} \Bigr) \exp_s(hx).
	$$
	A Taylor expansion for the exponential shows $\exp(hx) = \exp_s(hx) + \theta \frac{(hx)^{s+1}}{(s+1)!}$ 
	hence 
	$$
		\exp_s(hx) = \e^{h x}\Bigl( 1 - \theta \frac{(hx)^{s+1}}{(s+1)!} \Bigr). 
	$$
	Under our assumptions on $h$ and $x$ we have 
	$$
		0\leq hx \leq \delta \sqrt{\frac{s}{4\e}} \frac{\sqrt{s}}{3 \sqrt{e}} = \frac{s}{\e}\frac{\delta}{6}.
	$$
	We estimate with the aid of Stirling
	$$
		\frac{(hx)^{s+1}}{(s+1)!}\leq \frac{(s/\e)^{s+1}}{(s+1)!}	\Bigl( \frac{\delta}{6}\Bigr)^{s+1} \leq \Bigl( \frac{\delta}{6}\Bigr)^{s+1}. 
	$$
	Altogether, combining also with~\eqref{eq:rh2} and $s\geq 30$,
	\be \label{eq:ghx1}
		 g_h(x)^{-1} \leq \frac{\exp(-hx)}{(1- (\delta/\sqrt 2)^{s-5})(1- (\delta/6)^{s+1})}  \leq 1.0002\,\e^{-h x}
	\ee
	and the lemma follows. 
\end{proof}

\begin{lemma} \label{lem:I2h}
	Under Assumption~\ref{assum}, we have
	$$
		I_2(h) = \e^{- h x}\Biggl( \frac{\E\bigl[ \exp( - h \tilde\sigma (h) Z) \1_{\{Z\geq 0\}}\bigr]}{(1+\theta (\delta/\sqrt{2})^{s-5})(1 +\theta \delta^{s+1}) } + \theta \frac{1.0002}{\sqrt {2\pi}}\,\e^{- s/8}\Biggr).	$$
\end{lemma}

\begin{proof}  
	We split the expectations in two contributions, one belonging to $Z\leq b_1$ for some well-chosen truncation parameter $b_1$ and another contribution belonging to $Z>b_1$. On the event $Z\leq b_1$, we are going to estimate $g_h(x+\tilde \sigma(h) Z)^{-1}$ in a way similar to Lemma~\ref{lem:I1h} (where we had looked at $g_h(x)^{-1}$ only). The event $Z>b_1$ is shown to give a negligible contribution. Set 
	$$
		b_1:= \frac{b - x}{\tilde \sigma(h)}\geq \frac53 a.
	$$
	The bound uses $\tilde\sigma(h)\leq 2$ (see Lemma~\ref{lem:tiltexist}), $b=4a$ and $x\leq \frac23 a$.
	For $Z= y\geq b_1$, we note $0 \leq g_h(x+\tilde \sigma(h) Z)^{-1}\leq g_h(x)^{-1}$ and bound $g_h(x)^{-1}$ as in~\eqref{eq:ghx1}. This yields 
	$$
		0 \leq \E\bigl[  g_h( x + \tilde \sigma (h) Z)^{-1} \1_{\{Z\geq b_1\}}\bigr] 
			\leq 1.0002\, \exp(-h x) \P(Z\geq b_1).
	$$
	Next consider $Z=y\in [0,b_1]$. Then $u:=x+\tilde \sigma (h) y \leq b$, moreover 
	$$
		hu \leq h b \leq (\delta a) (4 a) = \delta \frac{s}{\e}
	$$
	and 
	$$
		\frac{(hu)^{s+1}}{(s+1)!} \leq \frac{(s/\e)^{s+1}}{(s+1)!} \delta^{s+1}\leq \delta^{s+1}. 
	$$
	Proceeding as in the proof of Lemma~\ref{lem:I1h}, we see that 
	$$
		\frac{\exp( - h u)}{(1+ (\delta/\sqrt{2})^{s-5})(1+ \delta^{s+1}) }	\leq g_h\bigl( u) ^{-1}\leq \frac{\exp( - h u)}{(1- (\delta/\sqrt{2})^{s-5})(1- \delta^{s+1}) }.
	$$
	We substitute $u = x + \tilde \sigma(h) Z$, take expectations, and find 
	$$
		 \E\bigl[  g_h( x + \tilde \sigma (h) Z)^{-1} \1_{\{Z\in [0,b_1]\}}\bigr]
			 = \frac{\E\bigl[ \exp( - h x - h \tilde\sigma (h) Z) \1_{\{Z\in [0,b_1]\}}\bigr]}{(1+\theta (\delta/\sqrt{2})^{s-5})(1 +\theta \delta^{s+1}) }.
	$$
	On the right-hand side, replacing the constraint $Z\in [0,b_1]$ by $Z\geq 0$ we obtain an upper bound. For a lower bound we note 
	$$
		\E\bigl[ \exp(  - h \tilde\sigma (h) Z) \1_{\{Z\in [0,b_1]\}}\bigr]
			\geq \E\bigl[ \exp(  - h \tilde\sigma (h) Z) \1_{\{Z\geq 0\}}\bigr] - \P(Z\geq b_1).
	$$	
	Combining everything, we get 
	$$
		I_2(h) \leq \frac{\E\bigl[ \exp( - h x - h \tilde\sigma (h) Z) \1_{\{Z\geq 0\}}\bigr]}{(1+\theta (\delta/\sqrt{2})^{s-5})(1 +\theta \delta^{s+1}) } + 1.0002\, \exp(-h x)\, \P(Z\geq b_1)
	$$
	and
	$$
		I_2(h) \geq \frac{\E\bigl[ \exp( - h x - h \tilde\sigma (h) Z) \1_{\{Z\geq 0\}}\bigr]}{(1+\theta (\delta/\sqrt{2})^{s-5})(1 +\theta \delta^{s+1}) } -1.0002\, \exp(-h x) \P(Z\geq b_1).
	$$
	Finally we estimate
	$$
		\P(Z\geq b_1)\leq \frac{1}{\sqrt{2\pi}}\int_{b_1}^\infty \frac{y}{b_1}\e^{-y^2/2}\dd y \leq \frac{1}{b_1\sqrt{2\pi}}\e^{-b_1^2/2}\leq \frac{1}{\sqrt{2\pi}}\e^{-s/8}.
	$$
	In the last inequality we  have used $b_1\geq 1$ and $4 b_1^2 \geq \frac{100}{9}a^2 =\frac{100}{9} \frac{s}{4\e} \simeq 1.02\, s\geq s$.
\end{proof} 

\subsection{Cram{\'e}r-Petrov series}

The function $\tilde L(x)$ defined in~\eqref{eq:Ltildedef} is for $h=t(x)$ equal to 
\[
	\tilde L(x) = \tilde \varphi(h) - h x + \frac12 x^2. 
\]
We can now estimate the absolute value of $\tilde{L}$.
\begin{lemma} \label{lem:crabo}
	Suppose $0\leq x< \sqrt s/(3\sqrt \e)$ with $s\leq 2\Delta^2$. Then 
	$|\tilde L(x)| \leq 1.2\, x^3/\Delta$. 
\end{lemma} 

\noindent We emphasize that the lemma does not need the assumption $s\geq 30$ or $s$ even. 

\begin{proof}
	To conclude, we follow~\cite[p.~31]{SS91} to bound $\tilde L(x)$, starting from
	\begin{align*}
		\tilde L(x) & = \frac12 h^2 + \sum_{j=3}^s \frac{ \kappa_j}{j!} h^j - h x+  \frac12 x^2
			 = \frac12 (h-x)^2 + \sum_{j=3}^s \frac{\kappa_j}{j!} h^j,
	\end{align*} 	
	which gives
	\[
		|\tilde L(x)|  \leq x^2 \Bigl( \frac12 \Bigl( \frac{h}{x} - 1\Bigr)^2 + \frac{h^2}{x^2} \sum_{j=3}^s \frac{h^{j-2}}{j(j-1) \Delta^{j-2}}\Bigr).
	\] 	
	Set $s':= 2\Delta^2$, $a':= \sqrt{s'/(4\e)}$, and $\delta':= x/(\frac23 a')$. Because of 
	$\frac23 a'\geq \frac 23 a = \sqrt s / (3\sqrt \e)$,  we have $\delta'<1$. It is easily checked that Lemma~\ref{lem:tiltexist} also holds true for primed quantities. In particular, the solution $h$ of $\tilde\varphi'(h) =x$ satisfies $h \leq\delta' a'= \delta' \Delta/\sqrt{2\e}$. It follows that 
	\[
		\sum_{j=3}^s \frac{h^{j-2}}{j(j-1) \Delta^{j-2}} \leq \delta' \sum_{j=3}^\infty \frac{(1/\sqrt{2\e})^{j-2}}{j(j-1)} \leq 0.0924\, \delta',
	\]  	
where we used a numerical evaluation of the series. 	
 	By Lemma~\ref{lem:tiltexist} applied to primed variables, $x = h(1+\theta \delta'/3)$, therefore $(h/x)^2\leq 9/4$ and 
	\[
		\Bigl( \frac h x - 1\Bigr)^2 \leq \Bigl(\frac{\delta'/3}{1- \delta'/3}\Bigr)^2 \leq \delta'^2 \Bigl( \frac{1/3}{1-1/3}\Bigr)^2 \leq \frac{\delta'}{4}. 
	\] 
	Altogether 
	\[
		|\tilde L(x)| \leq C x^2 \delta'= C x^2 \frac{x}{2 \Delta /(3 \sqrt{2\e})} \leq 1.17\, x^2 \frac{x}{\Delta}
	\] 
	with $C=  0.3329$. 
\end{proof} 

\begin{lemma} \label{lem:crabof}
	Suppose $0\leq x< \Delta/(3\sqrt e)$. Then
	$\tilde L(x) = \sum_{j=3}^\infty \tilde \lambda_j x^j$ 
	with $\tilde \lambda_j$ the coefficients given after Definition~\ref{def:cramer-petrov}. The series is absolutely convergent in $|x|\leq 0.3\, \Delta$. 
\end{lemma} 

\begin{proof} 
	Let $\tilde b_k$ be the coefficients from Eq.~\eqref{eq:tildebk}. By Proposition~\ref{prop:cramer-concrete2}, the series 
	\[
		H(x) = \sum_{k=1}^\infty \tilde b_k x^k
	\] 	
	has radius of convergence larger or equal to $0.3\, \Delta$. Because of 
	\[
		\frac 23 a = \frac 23 \sqrt{\frac{s}{4\e}} \leq  \frac{\sqrt 2}{3 \sqrt{\e}}\, \Delta < 0.3\, \Delta,
\]
	it follows that $H(x)$ is absolutely convergent for all $x\in [0, \frac23 a]$. The definition of the coefficient $\tilde b_k$ ensures that $\tilde \varphi'(H(x)) = x$ whenever $H(x)$ is convergent. 
	Since $H(x)\to 0$ when $x\to 0$, Lemma~\ref{lem:tiltexist} guarantees that $H(x) = h(x)$ for all sufficiently small $x$. Because of $h=h(x) \in [0,\delta a]$ and $\tilde \varphi''(h) \neq 0$ on $[0,a]$ (again by Lemma~\ref{lem:tiltexist}), the holomorphic inverse function theorem shows that $h(x)$ is analytic in $[0, \frac23 a)$. Therefore the equality $h(x) = H(x)$ extends to all of $[0, \frac23 a)$.
	The series representation for $\tilde L(x)$ then follows from the considerations in Appendix~\ref{app:cramer}. 	 
\end{proof}

\subsection{Theorem~\ref{thm:lemma22}---Conclusion of the proof} 
In the evaluation of $\E[\exp( - h \tilde \sigma(h) Z)\1_{\{Z\geq 0\}}]$ appearing in the approximation for $I_2(h)$ from Lemma~\ref{lem:I2h} we would like to replace $h\tilde \sigma(h)$ by $x$, and control the error with Lemma~\ref{lem:normal}. To that aim we first compare $h\tilde \sigma(h)$ with $x$.

\begin{lemma} \label{lem:etax}
	Let $\delta \in [0,1]$, $x=\frac 23 \delta a$, and $h\in [0,\delta a]$ the solution of $\tilde \varphi'(h)=x$. Then
	\[
		\E\bigl[\e^{-h \tilde \sigma(h) Z}\1_{\{Z\geq 0\}}\bigr] =\frac{1}{1+ \theta\, 0.86\,\delta}\,  \e^{x^2/2} \P(Z\geq x).
	\]
\end{lemma} 

\begin{proof} 
	We show first 
	\be \label{eq:hsibound}
		\Bigl| \frac{h \tilde \sigma(h)}{x} - 1\Bigr|\leq 0.86 \delta.
	\ee
	This bound  replaces~\cite[Eq.~(2.53)]{SS91} which looks incorrect.
	We start as in~\cite{SS91} and note 
	$$	
		h \tilde \sigma(h)^2 - x = h \tilde \varphi''(h) - \tilde \varphi'(h) =\sum_{j=2}^s \Bigl(\frac{1}{(j-2)!}-\frac{1}{(j-1)!}\Bigr)\kappa_j h^{j-1}= \theta \sum_{j=3}^s \frac{j-2}{j-1} h \Bigl(\frac{h}{\Delta}\Bigr)^{j-2}.
	$$
	The sum is non-negative and can be bounded from above, with $q=1/\sqrt{2 \e}$ and $h\leq \delta a \leq  \delta q \Delta$,  by $h$ multiplied with
	\begin{align*}
		\sum_{j=3}^s \frac{j-2}{j-1} (q \delta)^{j-2}	&\leq q \delta \Bigl(\frac12 +\sum_{k=1}^{s-4} (q \delta)^k\Bigr)\leq q \delta \Bigl(\frac12 + \frac{q}{1-q}\Bigr) \leq 0.537\,\delta.
	\end{align*}
	Therefore  we have $h\tilde \sigma(h)^2- x=\theta 0.537 \delta h$. Combining with the bound $x=h(1+\theta \frac\delta 3)$ from Lemma~\ref{lem:tiltexist}, we get 
	$$
		\frac{h\tilde \sigma(h)^2}{x}=1+\theta \delta h \,\frac{ 0.537}{h(1+\theta \delta/3)} = 1+\theta\, \delta \frac{0.537}{1- 1/3}=1+\theta\, 0.8055\, \delta.
	$$
	Eq.~(2.53) in~\cite{SS91} is the same but with $\tilde \sigma(h)$ instead of $\tilde \sigma(h)^2$, which looks like a typo.  We have to work a bit more to get rid of the square of the tilted variance. 
	We deduce with the help of the inequality $\sqrt{1+u}\leq1+\frac12 u$ that 
	$$
		h\tilde \sigma(h)= \sqrt{h x (1+\theta 0.8055\delta)}
			= x \sqrt{\frac{1+\theta  0.8055\delta}{1+\theta \delta/3}}
			\leq x\Bigl(1 + \frac12 \frac{	(0.8055+1/3) \delta}{2/3}\Bigr) \leq x (1+ 0.8542\,\delta).
	$$
	For the lower bound, we exploit the concavity of $u\mapsto \sqrt{1-u}$ on $[0,1]$, which implies that the difference quotient $\frac1u (\sqrt{1- u}- 1)$ is monotone decreasing. Let
	\[
		u = \frac{(0.8055+1/3)\delta}{1+\delta/3},\quad u^* = \frac{0.8055 + 1/3}{1+1/3} =0.8542.
	\]
	Then $u\leq u^*$ and 
	\[
	\frac{h \tilde{\sigma(h)}}{x}-1=	\sqrt{\frac{1+\theta  0.8055\delta}{1+\theta \delta/3}} - 1 \geq \sqrt{1-u}-1 \geq \frac{\sqrt{1-u^*} - 1}{u^*} u.
	\] 	
	We notice 
	\[
		\frac{1-\sqrt{1-u^*}}{u^*} \leq 0.7237,\quad \frac{u}{\delta} \leq 0.8055+1/3 \leq 1.139
	\] 	
	and conclude 
	\[
		\frac{h\tilde \sigma(h)}{x} - 1\geq - 0.7237\cdot 1.139\, \delta \geq - 0.8243\, \delta.
	\] 	
	The bound~\eqref{eq:hsibound} for $h\tilde \sigma (h) /x$ follows. In other words, 	
for 	$\eta:= h\tilde \sigma(h)  - x$ we showed $|\eta/x|\leq 0.86\, \delta$.  Therefore Lemma~\ref{lem:normal}  yields 
	\[
		\E\bigl[\e^{-h \tilde \sigma(h) Z}\1_{\{Z\geq 0\}}\bigr] = \frac{x}{x+\theta \eta} \E\bigl[\e^{-  xZ}\1_{\{Z\geq 0\}}\bigr] = \frac{1}{1+\theta \eta/ x}\, \e^{x^2/2}\, \P(Z\geq x). \qedhere
	\]
\end{proof} 

Now we can turn to the proof of Theorem~\ref{thm:lemma22}. We distinguish the cases $s\geq 30$ and $s\leq 29$. The bound on $\tilde L(x)$ has already been proven in Lemma~\ref{lem:crabo}, it holds true for all $s\leq 2\Delta^2$. 

\begin{proof}[Proof of Theorem~\ref{thm:lemma22} when $s\geq 30$]
We start from $\P(X\geq x) = \e^{\tilde \varphi(h)} (I_1(h)+ I_2(h))$, see~\eqref{eq:undo}. 
Lemma~\ref{lem:I1h} on $I_1(h)$, Lemma~\ref{lem:kolmodist3} on the Kolmogorov distance $D_h$, and the lower bound for $\P(Z\geq x)$ from Lemma~\ref{lem:normal} yield
\begin{align}
	|I_1(h)| & \leq \e^{-hx}
	   \frac{45.05}{\sqrt{s}} \Bigl( 0.8+ s \exp\Bigl( - \frac12 (1-\delta) s^{1/4}\Bigr)\Bigr) \notag \\
	   &\leq \e^{- h x+ x^2/2}\P(Z\geq x) \, \frac{113\, (x+1)}{\sqrt s}\Bigl( 0.8+ s \exp\Bigl( - \frac12 (1-\delta) s^{1/4}\Bigr)\Bigr). \label{eq:i101}
\end{align} 
Lemmas~\ref{lem:I2h} and~\ref{lem:etax} yield
\be \label{eq:i21}
	I_2(h) = C\, \e^{-hx + x^2/2} \P(Z\geq x) + \theta\, \frac{\exp( - hx)}{\sqrt{2\pi}}\, 1.0002\, \e^{-s/8}
\ee
with 
\[
	C:= \Bigl((1+\theta (\delta/\sqrt{2})^{s-5})(1 +\theta \delta^{s+1})(1+\theta\, 0.86 \delta)\Bigr)^{-1}.
\] 
In $C$, we bound $1+\theta\delta^{s+1} \geq 1- \delta$ and, using $s\geq 30$
\[
	(1+\theta (\delta/\sqrt{2})^{s-5})(1+\theta\, 0.86\, \delta) \geq 1 - \delta\Bigl( 0.86 + 2^{-25/2}\Bigr) =: 1- c \delta.
\]
with $c\leq 0.8602$.
The function $\delta\mapsto (1- c\delta)^{-1}$ is convex on $[0,1]$, therefore 
\[
	\frac{1}{\delta} \Bigl( \frac{1}{1-c\delta} -1 \Bigr) \leq \left.	\frac{1}{\delta} \Bigl( \frac{1}{1-c\delta} -1 \Bigr)\right|_{\delta =1}  = \frac{c}{1-c}\leq 6.2.
\]
We deduce 
\[
	C \leq \frac{1+6.2\, \delta}{1-\delta} = 1+ \frac{7.2\, \delta}{1-\delta}.
\]
For a lower bound for $C$, we use $1/(1+x) \geq 1 - x$ which gives 
\begin{align*} 
	C & \geq (1-\delta^{s+1}) \bigl( 1- (\delta/\sqrt 2)^{s+1}\bigr) (1- \delta/6) \\
		& \geq 1 - \delta^{s+1} - (\delta/\sqrt 2)^{31} - \delta /6 - \delta^{63}\frac{1}{\sqrt{2}^{31} } \frac{1}{6} \\
		& \geq 1 - \delta (1+ 2^{-31/2} + 1/6 +\frac{1}{\sqrt{2}^{31} } \frac{1}{6}),
\end{align*} 
which is certainly larger than $1 - 7.2\, \delta/(1-\delta)$. Thus we have checked
\[
	C  = 1+ \theta \frac{7.2\, \delta}{1-\delta}.
\] 
We insert this bound into~\eqref{eq:i21}, combine with the lower bound $\P(Z\geq x)\geq \exp(- x^2) /(\sqrt{2\pi} (x+1))$ from Lemma~\ref{lem:normal}, and obtain
\be \label{eq:i23}
	I_2(h) = \e^{- hx + x^2/2}\P(Z\geq x) \Bigl( 1+ \theta\, \frac{7.2\, \delta}{1-\delta} +\theta\, 1.0002\, (x+1)\e^{-s/8} \Bigr).
\ee
Because of $x = \frac 23 \delta a  = \delta \, \sqrt{s}/(3\sqrt e)$ we have 
\[
	7.2\, \delta = \frac{7.2\, x}{\sqrt s/(3 \sqrt \e)} \leq 35.4\,  \frac{x}{\sqrt s}. 
\] 
Moreover, $1.0002\, \sqrt s\, \e^{- s /8} \leq 1$ for all $s \geq 30$. 
It follows that 
\be \label{eq:i24}
	I_2(h) = \e^{- hx + x^2/2}\P(Z\geq x) \Bigl( 1+ \theta\, \frac{36.4\, (x+1)}{(1-\delta)\sqrt s}\Bigr).
\ee
Finally we add up the estimates~\eqref{eq:i101} and~\eqref{eq:i24} for $I_1(h)$ and $I_2(h)$, remember Eq.~\eqref{eq:undo}, and obtain 
$$
	\P(X\geq x) =\e^{\tilde L(x)} \P(Z\geq x) \Bigl( 1+ \theta\, \frac{x+1}{(1-\delta)\sqrt s}\Bigl( 127+ 113\, s\, \e^{-(1-\delta) s^{1/4} /2}\Bigr)\Bigr).
$$
Because of $\delta = x/(\sqrt s/(3\sqrt \e))$, this completes the proof of the theorem when $s$ is even and $s\geq 30$. 
\end{proof}

\begin{proof}[Proof of Theorem~\ref{thm:lemma22} when $s<30$]
	If $s\leq 30$, then $x\leq \sqrt{s}/(3\sqrt{\e}) \leq 1.2$. 
	By Lemma~\ref{lem:crabo} and $s\leq 2\Delta^2$,
	\[
		|\tilde L(x)| \leq 1.2\, x^2 \frac{x}{\Delta}\leq 1.2\, x^2 \frac 13 \sqrt{\frac{s}{\Delta^2\e}}  
		\leq \frac{1.2^3  \sqrt 2}{3 \sqrt{ \e}}\leq 0.5.
	\]		
	Using $\P(Z\geq 1.2) \geq 0.115$, we get 
	\begin{align*}
		\e^{\tilde L(x)}\P(Z\geq x)\Bigl( 1+ f(\delta,s) \frac{x+1}{\sqrt s}\Bigr) 
		& \geq \e^{ - 0.5} \, 0.115\, \Bigl( 1 + \frac{127}{\sqrt{30}} \Bigr) \\
		&\geq 1.6\geq\P(X\geq x).
	\end{align*}
	Furthermore
	\[
		\frac{\P(X\geq x)}{\exp( \tilde L(x))\P(Z\geq x)}-1 \geq - 1 \geq - \frac{127}{\sqrt{30}} \geq - f(\delta,s) \frac{x+1}{\sqrt s}. 
	\]
	This completes the proof the inequality.
\end{proof} 

\section{Bounds for Weibull tails. Proof of Theorem~\ref{thm:lemma23}} \label{sec:weibull}
Recall that in Theorem~\ref{thm:lemma23} we assume condition~\eqref{condition-sgamma}, i.e.  
$ |\kappa_j(X)|\leq \frac{j!^{1+\gamma}}{\Delta^{j-2}}$ for $ j \geq 3$. 
  The main idea is to use Theorem \ref{thm:lemma22} (which uses condition~\eqref{condition-s}, i.e.~$|\kappa_j(X)|\leq \frac{(j-2)!}{\Delta^{j-2}}  , j \in \{3,\ldots, s+2\}$). However, in general~\eqref{condition-sgamma} does not imply~$\mathcal S_\gamma$, but this is true if $\Delta$ in \eqref{condition-s} is replaced by some $\Delta_s$. For this purpose, we set for $s\in \N$ 
\be\label{eq:deltas}
	\Delta_s :=\frac{\Delta}{6(s+2)^\gamma},
\ee
where we choose $s$ depending on $\Delta$ and $\gamma$ later. 

\begin{lemma} \label{lem:w1}
	Let $s\geq 4$. If $|\kappa_j|\leq j!^{1+\gamma}/\Delta^{j-2}$ for all $j\geq 3$ then $|\kappa_j|\leq (j-2)!/\Delta_s^{j-2}$ for all $j=3,\ldots, s+2$. 
\end{lemma} 

\begin{proof}
	We show $j!^{1+\gamma} \leq (j-2)! ( 6 (s+2)^\gamma)^{j-2}$ for $j=3,\ldots, s+2$ or equivalently, 
	\[
		j(j-1) j!^\gamma \leq (6(s+2)^\gamma)^{j-2}\quad (j=3,\ldots,s+2). 
	\] 
	The proof is by a finite induction over $j$. For $j=3$, the inequality reads $6\cdot 6^\gamma \leq 6(s+2)^\gamma$ and it holds true because of $s\geq 4$. For the induction step, we note that if the inequality holds true for $j-1\geq 3$ with $j-1 <s+2$ then 
	\begin{align*}
		j(j-1) j!^\gamma &= \frac{j}{j-2} j^\gamma \, \bigl( (j-1)(j-2) (j-1)!^\gamma\bigr)  \\
	&	\leq 2 (s+2)^\gamma\, ( 6 (s+2)^\gamma)^{j-3} \leq (6 (s+2)^\gamma)^{j-2}. \qedhere
	\end{align*}
\end{proof} 

Set
\be \label{eq:schoice}
	s =s_{\gamma}:= 2 \Bigl \lfloor \frac12\Bigl(\frac{\Delta^2}{18}\Bigr)^{1/(1+2\gamma)}\Bigr \rfloor - 2.
\ee
Before we can apply Theorem \ref{thm:lemma22} with this $s$ and $\Delta_s$ we note some relations between $s$ and $\Delta_s$ resp.~$\Delta$ and (below Lemma \ref{lem:w2}) a relation between $\Delta_s$ and   $\Delta_{\gamma}$, where  $\Delta_{\gamma}$ was given in \eqref{eq:dsmdef} and appears in the formulation of Theorem \ref{thm:lemma23}. Since we will see that the statement Theorem \ref{thm:lemma23} is trivially true for $\Delta$ sufficiently small, some of the following estimates are stated for large values of $\Delta$ only. 
\begin{lemma} \label{lem:w2} 
	The even integer $s$ defined in~\eqref{eq:schoice} satisfies $s\leq 2 \Delta_s^2$. If $\Delta \geq 20^{1+2\gamma}$, then in addition $s\geq 4$ and 
	\be \label{eq:sdelta6}
		\sqrt s\geq \sqrt{0.82} \Bigl(\frac{\sqrt 2 \Delta}{6}\Bigr)^{1/(1+2\gamma)}.
	\ee
\end{lemma} 

In~\cite{SS91} it is claimed that the bound $s\geq 4$ and the inequality~\eqref{eq:sdelta6} (with $0.95$ instead of $\sqrt{0.82} \approx 0.9$, see~\cite[Eq.~(2.64)]{SS91}) holds true for $\Delta> 10^{1+2\gamma}$, which looks incorrect. 

\begin{proof}
	We note
	\[
		s (s+2)^{2\gamma} \leq (s+2)^{1+2\gamma} \leq  \frac{\Delta^2}{18},
	\] 
	which gives $s\leq 2 \Delta_s^2$. For the lower bound, if $\Delta \geq 20^{1+2\gamma}$, then 
	$(\frac{\Delta^2}{18})^{\frac{1}{1+2\gamma}} \geq \frac{400}{18^{\frac{1}{1+2\gamma}}} \geq \frac{400}{18}\geq  22$ and 
	\[
		s \geq \Bigl(\frac{\Delta^2}{18}\Bigr)^{1/(1+2\gamma)} - 4 \geq 22 - 4 \geq 4
	\] 
as well as 
	\[
		s\geq \Bigl(\frac{\Delta^2}{18}\Bigr)^{1/(1+2\gamma)} \Bigl(1 - 4 \Bigl(\frac{18}{\Delta^2}\Bigr)^{1/(1+2\gamma)}\Bigr)
		\geq \Bigl(\frac{\Delta^2}{18}\Bigr)^{1/(1+2\gamma)} \Bigl(1 - \frac{4\cdot 18}{400}\Bigr) 
		= 0.82 \Bigl( \frac{\sqrt 2\Delta}{6}\Bigr)^{2/(1+2\gamma)}.
    \]
\end{proof} 
We  note that $\Delta_s$, which we introduced in order to use Theorem~\ref{thm:lemma22} and $\Delta_\gamma:=\frac{(\Delta/\sqrt{18})^{1/(1+2\gamma)}}{6}$, which appears in  Theorem~\ref{thm:lemma23}, satisfy the following  inequalities:
\be \label{eq:ddsg}
	\Delta_s \geq \frac16 \Delta \Bigl( \frac{\Delta^2}{18}\Bigr)^{- \gamma/(1+2\gamma)} 
		= \sqrt{18} \cdot \frac 16 \Bigl(  \frac{\Delta}{\sqrt{18}}\Bigr)^{1/(1+2\gamma)} \Bigr) 
		= \sqrt{18}\Delta_\gamma \geq 4 \Delta_\gamma.
\ee
We also note that in \eqref{eq:ddsg} we have $\Delta_s= \sqrt{18} \Delta_{\gamma}$ if in the definition of  $s$ in \eqref{eq:schoice}  we neglect the floor function. 
Further, we have 
\begin{align} 
\frac{\Delta_s}{\Delta_{\gamma}} &= \frac{\Delta}{6 (s+2)^{\gamma}} \frac{1}{\frac{1}{6} (\Delta / \sqrt{18} )^{\frac{1}{1+2\gamma}}} \leq \Delta^{1-\frac{1}{1+2\gamma}} \sqrt{18}^{\frac{1}{1+2\gamma}}\frac{1}{2^{\gamma} \left( \frac{1}{2} \left(\Delta^2/18 \right)^{\frac{1}{1+2\gamma}}-1 \right)^{\gamma}}
\nonumber \\
&= \sqrt{18} \frac{1}{\left(1-2\left(\frac{18}{\Delta^2}\right)^{\frac{1}{1+2\gamma}}\right)^{\gamma}}\leq \sqrt{18} \left(\frac{1}{1-\frac{9}{100}}\right)^{\gamma}
\leq \sqrt{18} (1.1)^{\gamma}
. \label{eq_dsgd}
\end{align}
Together with Lemma \ref{lem:w2}, which reads $\sqrt s\geq  6 \sqrt{0.82} \Delta_{\gamma}$, \eqref{eq_dsgd} implies that    for $\Delta\geq 20^{1+2\gamma}$, the quantity $\Delta_s^2 /s^2$ is bounded by some constant that does not depend on $\Delta$.
We further note that from the definition of $s$ we have 
\begin{equation}\label{est_s_upper}
s \leq 2 ( 1/2 (\Delta^2/18)^{1/(1+2\gamma)}) = 36 \Delta_{\gamma}^2.
\end{equation}

Next, let $\tilde \lambda_k$, $k\geq 3$, be the coefficients defined with the help of the truncated functions below Definition~\ref{def:cramer-petrov}, with $s$ given by~\eqref{eq:schoice}. Set 
\be  \label{eq:weibull-tilde-L}
	\tilde L(x):= \sum_{j=3}^\infty \tilde \lambda_j x^{j}. 
\ee 
The function appears naturally when applying Theorem~\ref{thm:lemma22} in the proof of Theorem~\ref{thm:lemma23}. A first rough estimate for $\tilde L(x)$ is obtained as follows. Because of Lemma~\ref{lem:w1} we can apply the bound for $\tilde \lambda_j$ from Proposition~\ref{prop:cramer-concrete2} (remember $\tilde \lambda_k = - \tilde b_{k-1}/k$)  with $\Delta_s$ instead of $\Delta$, which gives 
\be \label{eq:tildebo}
	|\tilde \lambda_j|\leq \frac{1}{0.6\, j(j-1)} \frac{1}{(0.3 \, \Delta_s)^{j-2}}\quad \text{ for all }j \geq 3.
\ee
Consequently for $|x|\leq 0.3\, \Delta_s$ we have
\be \label{eq:tila}
	|\tilde L(x)| \leq \frac{x^3}{0.3\, \Delta_s}\sum_{j=3}^\infty \frac{1}{0.6\, j(j-1)} = \frac{x^3}{0.36\, \Delta_s} \leq  \frac{x^3}{1.54\, \Delta_\gamma}. 
\ee
This bound should be compared to $|L(x)|\leq C' x^3 /\Delta$ in Theorem~\ref{thm:easy} and $|\tilde L(x)|\leq 1.2\, x^2/\Delta$ in Theorem~\ref{thm:lemma22}. Notice, however, that $x^3/\Delta_s$ can be fairly large. Therefore in some circumstances it can be useful to write $\tilde L(x)$ as a polynomial in $x$, obtained by truncating the Cram{\'e}r-Petrov series, plus a correction term that is small when $x$ is small compared to $\Delta_s$.

\begin{lemma} \label{lem:statucram} 
	Suppose that condition~\eqref{condition-sgamma} holds true with $\Delta > 20^{1+2\gamma}$. Let $s$ be the integer from~\eqref{eq:schoice} and 
\be \label{eq:mchoice} 
	m:= \min \Bigl( \left \lceil{ \frac{1}{\gamma}}\right \rceil +1, s\Bigr).
\ee
Then 
\be \label{eq:weibull-tilde2}
	\tilde L(x) = \sum_{j=3}^m \lambda_j x^j + \theta C_\gamma \Bigl( \frac{x}{\Delta_\gamma}\Bigr)^{m+1}
\ee
for some constant $C_\gamma>0$ and all $x\in (0, \Delta_\gamma)$. 
\end{lemma} 

\begin{proof}
We use  $\tilde \lambda_k = \lambda_k$ for $k \leq s$ (see~\eqref{eq:tilde-notilde}) and split 
\be \label{eq:lamszerl}
	\sum_{j=m+1}^\infty \tilde \lambda_j x^j = \sum_{j=m+1}^s \lambda_j x^j + \sum_{j=s+1}^\infty \tilde \lambda_j x^j. 
\ee
The first sum on the right-hand side is set to zero if $m=s$.
For the second sum we use~\eqref{eq:tildebo} and obtain 
\begin{align*} 
	&\sum_{j=s+1}^\infty |\tilde \lambda_j| x^j 
	\leq \sum_{j=s+1}^\infty \frac{1}{j(j-1)} \frac{1}{0.6} \frac{x^{j}}{(0.3\Delta_s)^{j-2}} = \frac{5 x^2}{3} \sum_{j=s+1}^\infty \frac{1}{j(j-1)} \Bigl(\frac{x}{0.3\, \Delta_s}\Bigr)^ {j-2}\\
	&\leq \frac{5 x^2}{3 s^2} \frac{(x/(0.3\Delta_s))^{s-1}}{1- x/(0.3\,\Delta_s)} 
	 \leq \frac{5 (0.3)^2 \Delta_s^2}{3 s^2} \frac{(x/(0.3\,\Delta_s))^{s+1}}{1- x/(0.3\,\Delta_s)}
	 \leq 0.15 \frac{\Delta_s^2}{s^2} \frac{1}{1- x/(0.3\,\Delta_s)}(x/\Delta_{\gamma}))^{s+1}. 
\end{align*} 
Here, we used $\Delta_s\geq 4 \Delta_\gamma$ (see~\eqref{eq:ddsg}) for the last inequality. The factor $\frac{1}{1- x/(0.3\,\Delta_s)}$ is bounded for $x\in (0, \Delta_\gamma)$ and also $\Delta_s^2 /s^2$ is bounded by some constant that does not depend on $\Delta$ (see \eqref{eq_dsgd}). 
Hence, we get 
\be \label{eq:dimsum}
	\sum_{j=s+1}^\infty |\tilde \lambda_j| x^j \leq C'_\gamma \Bigl( \frac{x}{\Delta_\gamma}\Bigr)^{s+1}
\ee
for some constant $C'_\gamma>0$ and all $x\in (0, \Delta_\gamma)$.
 
If $m=s$, then \eqref{eq:weibull-tilde2} follows, with $C_\gamma = C'_\gamma$. 

If $m <s$, we use the bound from Proposition~\ref{prop:cracob}, which yields
\begin{align}
	\sum_{j=m+1}^{s} |\lambda_j| x^j & \leq x^2 \sum_{j=m+1}^{s} (j+1)!^\gamma \Bigl( \frac{15 x}{\Delta}\Bigr)^{j-2}  \notag	\\
		&\leq (m+3)!^\gamma\,x^2 (s+2)^{-(m-1)\gamma} \sum_{j=m+1}^{s} (s+2)^{(j-2)\gamma} \Bigl( \frac{15 x}{\Delta}\Bigr)^{j-2} \notag \\
		&\leq (m+3)!^\gamma\, x^2 (s+2)^{-(m-1)\gamma}  \frac{(2.5\, x/\Delta_s)^{m-1}}{1- 2.5\, x/\Delta_s} \notag\\
		& \leq \frac{(m+3)!^\gamma}{6.25} \frac{\Delta_s^2}{(s+2)^{(m-1)\gamma}} \frac{(2.5\, x/\Delta_s)^{m+1}}{1- 2.5\, x/\Delta_s}. \label{eq:domsum}
\end{align} 
Notice that $\Delta_\gamma\leq 0.25\, \Delta_s\leq \Delta_s/2.5$, hence the geometric series in the previous inequality are indeed absolutely convergent for $x\in (0,\Delta_\gamma)$. 

By Lemma~\ref{lem:w2}, for large $s$ (i.e.~large $\Delta$, see~\eqref{eq:schoice}), $s$ is of the order of $\Delta_s^2$ and if in addition $x$ is of the order of $\Delta_s$ (and strictly smaller than $\Delta_s/2.5$) then~\eqref{eq:domsum} is of the order of 
\[
	\frac{\Delta_s^2}{\Delta_s^{2\gamma(m-1)}},
\] 
which stays bounded because $m\geq 1+ 1/\gamma$. It follows that for some constant $C''_\gamma>0$ and all $x\in (0,\Delta_\gamma)$, 
\[
	\sum_{j=m+1}^{s} |\lambda_j| x^j\leq C''_\gamma \Bigl(\frac{x}{\Delta_\gamma}\Bigr)^{m+1}. 
\] 
We combine this bound with~\eqref{eq:dimsum} and obtain~\eqref{eq:weibull-tilde2} for $m<s$. The case $m=s$ was already addressed above, the proof is complete. 
\end{proof}

\begin{proof}[Proof of Theorem~\ref{thm:lemma23}]
We start with the proof for  $\Delta \geq 20^{1+2\gamma}$.	Set $\Delta_\gamma:= \frac 16 (\Delta\sqrt 2/6)^{1/(1+2\gamma)}$ and let $\Delta_s$ and $s$ be as in Eqs.~\eqref{eq:deltas} and~\eqref{eq:schoice}. By Lemma~\ref{lem:w2}, for $\Delta \geq 20^{1+2\gamma}$, we have
	\[
		\Delta_\gamma \leq \frac{\sqrt{s}}{6 \sqrt{0.82}}\leq \frac{\sqrt s}{3 \sqrt \e}.
	\]
For $\delta:= x/\Delta_\gamma$ we have $x = \delta \Delta_\gamma \leq \delta \sqrt{s}/(3\sqrt \e)$. 	Theorem~\ref{thm:lemma22} is stated for equality $x= \delta \sqrt s / (3\sqrt\e)$ but it extends to $x\leq \delta \sqrt s/ (3 \sqrt \e)$ because $f(s,\delta)$ is monotone increasing in $\delta$. Therefore we obtain
	\[
		\P(X\geq x) = \e^{\tilde L(x)}\P(Z\geq x) \Bigl( 1+ \theta f(\delta,s) \frac{x+1}{\sqrt s}\Bigr)
	\]   
	 $\tilde L$ given in~\eqref{eq:weibull-tilde-L}. Then  
	\begin{align*}
		f(\delta,s) & = \frac{1}{1-\delta}\Bigl( 127+113\, s\, \e^{- (1-\delta) s^{1/4}/2} \Bigr) \\
			& \leq  \frac{1}{1-\delta}\Bigl( 127+ 113 \cdot 36 \Delta_\gamma^2 \exp\bigl(- \frac12 (1-\delta) (3 \sqrt \e\, \Delta_\gamma)^{1/2}	\bigr) \Bigr)\\
			&\leq  \frac{1}{1-\delta}\Bigl( 127+ 113 \cdot 36 \Delta_\gamma^2 \exp\bigl(- (1-\delta)  \sqrt{\Delta_\gamma} \bigr) \Bigr),
	\end{align*} 
	where we used \eqref{est_s_upper} for the first inequality. Hence we have 
	\begin{align*}
	\frac{1}{\sqrt{s}} f(\delta,s) &\leq \frac{1}{\Delta_{\gamma}}\frac{1}{6 \sqrt{0.82}}\frac{1}{1-\delta} \Bigl( 127+ 113 \cdot 36 \Delta_\gamma^2 \exp\bigl(- (1-\delta)  \sqrt {\Delta_\gamma}	\bigr) \Bigr)
	\\
	& \leq  \frac{1}{\Delta_{\gamma}}\frac{1}{1-\delta} \Bigl( 24+ 749  \Delta_\gamma^2 \exp\bigl(- (1-\delta)  \sqrt {\Delta_\gamma}	\bigr) \Bigr) = \frac{1}{\Delta_{\gamma} }g(\delta, \Delta_{\gamma})
	\end{align*}
	Then we obtain for $\Delta \geq 20^{1+2\gamma}$
	\begin{equation}\label{est_the_2_3}
		\P(X\geq x) = \e^{\tilde L(x)}\P(Z\geq x) \Bigl( 1+ \theta g(\delta,\Delta_\gamma) \frac{x+1}{\Delta_\gamma}\Bigr), \quad x\in [0, \Delta_\gamma).
	\end{equation}   
We check that \eqref{est_the_2_3} is trivial for $\Delta \leq 20^{1+2\gamma}$ as follows.
For $0<\Delta<20^{1+2\gamma}$ we have 
\[
 0 \leq x < \Delta_{\gamma}\leq \frac{1}{6} \left( \frac{{20^{1+2\gamma}} \sqrt{2}}{6}\right)^{\frac{1}{1+2 \gamma}}
\leq \frac{5\sqrt{2}}{9} \leq 0.8. 
\]
By \eqref{eq:tila} we have for $ x\in [0, \Delta_\gamma)$
\begin{equation*}
|L_{\gamma} (x)| \leq \frac{x^3}{1.54 \Delta_{\gamma}}\leq \frac{x^2}{1.54} \leq 0.42
\end{equation*}
 We supplement this with the inequality
 \begin{equation*}
   g(\delta,\Delta_\gamma)\frac{x+1}{\Delta_{\gamma}}=\frac{1}{1-\delta} \Bigl( 24+ 749 \Delta_\gamma^2 \exp\bigl(- (1-\delta)  \sqrt {\Delta_\gamma}	\bigr) \Bigr)\frac{x+1}{\Delta_{\gamma}} \geq 24.
 \end{equation*}
 This implies
 \begin{align*}
  \bigl(1- \Phi(x)\bigr) e^{L_{\gamma} (x)} \left( 1+   g(\delta,\Delta_\gamma) \frac{x+1}{\Delta_{\gamma}} \right)
  &\geq 0.79\cdot  0.65 \cdot 25 \geq 12 \geq  1 \geq 1- F_X(x)
 \end{align*}
as well as
\begin{align*}
\frac{1- F_X(x)}{e^{L_{\gamma} (x)}\bigl(1- \Phi(x)\bigr)} -1 \geq -1 \geq -24  \geq - g(\delta,\Delta_\gamma) \frac{x+1}{\Delta_{\gamma}}.
\end{align*}
This proves that \eqref{est_the_2_3} is also true for $\Delta \leq 20^{1+2\gamma}$.  
The statement on $\tilde L_\gamma = \tilde L$ in the theorem follows from Eq.~\eqref{eq:tila} and Lemma~\ref{lem:statucram}.
\end{proof}

\section{Berry-Esseen bound. Proof of Theorem~\ref{thm:besseen}} \label{sec:besseen}

The proof of Theorem~\ref{thm:besseen} is similar to the proof of the normal approximation of the tilted measure for the proof of Theorem~\ref{thm:lemma22}, see Section~\ref{sec:normaltilt}. The primary ingredients are a smoothing inequality, e.g.\ Lemma~\ref{lem:zolotarev2}, and bounds on the characteristic function of $X$. The bounds on characteristic functions are similar to the bounds from Sections~\ref{sec:chafu1} and~\ref{sec:chafu2}, the estimates are slightly easier because we do not need to take tilt parameters into account.

\begin{proof}[Proof of Theorem~\ref{thm:besseen}]
    Replacing the exponentials by their second-order Taylor approximations, we immediately get
	\[
			\bigl| \E\bigl[\e^{\mathrm i t X}\bigr] - \e^{- t^2/2}\bigr|  \leq t^2
	\] 
	for all $t\in \R$, compare Section~\ref{sec:chafu2}. 
	But now suppose first that $\Delta > 20^{1+2\gamma}$. Define $\Delta_\gamma$, $s:=s_\gamma$, and $m_\gamma$ as in~\eqref{eq:dsmdef} and $\Delta_s = \Delta/ (6 (s+2)^\gamma)$ as in~\eqref{eq:deltas}. In Section~\ref{sec:weibull} we showed that $s\leq 2 \Delta_s^2$ and $s\geq 4$, moreover $X$ satisfies condition~\eqref{condition-s} with $\Delta_s$ instead of $\Delta$, i.e., $|\kappa_j|\leq (j-2)!/\Delta_s^{j-2}$ for $j=3,\ldots, s+2$. 	Let $\tilde \varphi(t) = \frac{t^2}{2}+\sum_{j=1}^s \kappa_j t^j/j!$. We split 
	\be \label{eq:esseen1}
		\bigl| \E\bigl[\e^{\mathrm i t X}\bigr] - \e^{- t^2/2}\bigr| 
		\leq \bigl| \E\bigl[\e^{\mathrm i t X}\bigr] - \e^{\tilde \varphi(\mathrm i t)}\bigr| + \bigl|\e^{\tilde \varphi(\mathrm i t)}- \e^{- t^2/2}\bigl|. 
	\ee
	For $|t|\leq \Delta_s$, we have 
	\[
		\Bigl|\tilde \varphi(\mathrm i t) +\frac{t^2}{2}\Bigr| 
		\leq t^2 \sum_{j=3}^s  \frac{(t/\Delta_s)^{j-2}}{j(j-1)} \leq \frac{t^3}{2\Delta_s}.
	\] 
	In the last inequality we have used $\sum_{j=3}^\infty 1/[j(j-1)] =\sum_{j=2}^\infty \left( 1/j - 1/(j+1)\right)=  1/2$.
	For the first term on the right-hand side of~\eqref{eq:esseen1}, we use a couple of relations from Section~\ref{sec:tildedefs}, notably~\eqref{eq:egh}, involving the truncated moments $\tilde m_j$ and the functions $g_t(x) = \exp_s(tx) + x^2 \tilde r(t)$, which give
	\[
		\bigl| \E\bigl[\e^{\mathrm i t X}\bigr] - \e^{\tilde \varphi(\mathrm i t)}\bigr| 
		 = \bigl| \E\bigl[\e^{\mathrm i t X} - g_{\mathrm it}(X) \bigr]\bigr|
		 \leq \E\bigl|\e^{\mathrm i t X} - \exp_s(\mathrm i t X) \bigr| + |\tilde r(\mathrm i t)|, 
	\]
	which is similar to~\eqref{eq:ghfunctional-split} with $h=0$ (note $g_h(x) = 1$ when $h=0$) and added expected values. Let us assume that $s\geq 30$, which is the case when $\Delta$ is large enough, so that the bounds from Section~\ref{sec:chafu1} and~\ref{sec:chafu2} are applicable. Let $a:= \sqrt{s/(4\e)}$ and $\delta_2\in (0,1)$ with $\delta_2^{-2}\leq a$. Then, proceeding as in Lemma~\ref{lem:chierror1}, we obtain the upper bound 
	\[ 
		\bigl| \E\bigl[\e^{\mathrm i t X}\bigr] - \e^{\tilde \varphi(\mathrm i t)}\bigr| 
			\leq 4 \sqrt 2\, \frac{\delta_2^{\lfloor 4a\rfloor}}{1 - \delta_2} \quad (|t|\leq \delta_2 a).	
	\]
	We deduce 
	\[
			\bigl| \E\bigl[\e^{\mathrm i t X}\bigr] - \e^{- t^2/2}\bigr| 
 \leq 4 \sqrt 2\, \frac{\delta_2^{\lfloor 4a\rfloor}}{1 - \delta_2} + \e^{- t^2/2} \bigl| \e^{ t^3/(2\Delta_s)} - 1\bigr| \qquad (|t|<\delta_2 a). 
	\] 
	A reasoning similar to Lemma~\ref{lem:kolmodist1}  shows that the Kolmogorov distance between the normal law and the law of $X$ is bounded by some constant times $1/\sqrt s$ hence also by some constant times $1/\Delta_\gamma$,
 \[
 	\sup_{x\in \R}\Bigl| \P(X\geq x) - \P(Z\geq x) \bigr|\leq \frac{C}{\Delta_\gamma}. 
 \] 	
 This holds true if $\Delta$ is large enough so that $s=s_\gamma$ is larger than $30$, say $\Delta \geq \Delta^*$. For smaller $\Delta$, the bound is trivially true if we choose $C\geq 2 \sup_{\Delta \leq \Delta^*} \Delta_\gamma$. 
\end{proof}

\appendix
\section{Cram{\'e}r-Petrov series} \label{app:cramer}

Here we recall some facts on the Cram{\'e}r-Petrov series, also called Cram{\'e}r series, for the reader's convenience. The series was introduced by Cram{\'e}r~\cite{cramer38} and appeared in a limit theorem subsequently improved by Petrov~\cite{petrov54}, see~\cite[Theorem 5.23]{petrov-book} or ~\cite[Chapter~8]{ibragimov-linnik}. The recurrence relation~\eqref{eq:brecurrence} below can be found in~\cite{SS91}.

\subsection{When Cram{\'e}r's condition is satisfied.} 
Even though we are primarily interested in heavy-tailed variables,  the definition of the Cram{\'e}r-Petrov series is best understood by looking first at a random variable that satisfies Cram{\'e}r's condition. Thus let $X$ be a real-valued random variable $X$ such that 
$$
	\mathbb E[\exp(tX)]<\infty 
$$
for all $t\in (-\Delta,\Delta)$, for some $\Delta>0$. Further assume that $X$ is not almost surely constant so that the variance $\sigma^2$ is non-zero. Then the cumulant generating function $\varphi(t) = \log \mathbb E[\exp( tX)]$ is analytic in some neighborhood of the origin and the Taylor expansion 
$$
	\varphi(t) = \mu t + \frac12 \sigma^2 t^2 + \sum_{j=3}^\infty \frac{\kappa_j}{j!} t^ j
$$
has a strictly positive radius of convergence.  Recall $I(x) = \sup_{t\in \R} (tx- \varphi(t))$.

\begin{prop} \label{prop:cramer-petrov1}
	Let $X$ be a real-valued random variable that is not almost surely constant. Assume $\mathbb E[\exp(tX)]<\infty$ for all $t\in (-\Delta,\Delta)$ for some $\eps>0$. Let $\mu = \E[X]$ and $\sigma^2 =\mathbb V(X)$. 
	Then the 
 Taylor expansion of $I$  at $\mu=\E[X]$ is of the form 
	\be \label{eq:itaylor}
		I(\mu + \tau) =  \frac{\tau^2}{2\sigma^2} - \sum_{j=3}^\infty \lambda_j \tau^j
	\ee
	and has non-zero radius of convergence. The coefficients $(\lambda_j)_{j\geq 3}$ are given by $\lambda_k = - b_{k-1}/k$ with coefficients $(b_k)_{k\geq 2}$ computed recursively as follows: $b_1 = 1/\sigma^2$ and for all $k\geq 2$, 
	\be \label{eq:brecurrence}
		b_k = - \frac{1}{\sigma^2}  \sum_{r=2}^k \frac{\kappa_{r+1}}{r!} \sum_{\substack{1\leq j_1,\ldots,j_r\leq k-1:\\ j_1+\cdots + j_r = k}} b_{j_1}\cdots b_{j_r}.
	\ee
\end{prop} 

\noindent The coefficients $b_k$ have a significance of their own: for small $t$, 
$$
	\varphi'(t)=\mu+\tau \ \Leftrightarrow\ t=\frac{\tau}{\sigma^2}+ \sum_{k=2}^\infty b_k \tau^k.
$$

\begin{proof}
	The restriction of the cumulant generating function to $(-\Delta,\Delta)$ is strictly convex and in $C^\infty((-\Delta,\Delta))$, its derivative $\varphi'$ is a strictly increasing smooth bijection from $(-\Delta,\Delta)$ onto some open interval $(a,b)\subset \R$, and $\varphi''(t) >0$ on $(-\Delta,\Delta)$. Therefore the inverse map $(\varphi')^{-1}:(a,b)\to (-\Delta,\Delta)$ is smooth as well, i.e., $(\varphi')^{-1}\in C^\infty((a;b))$.  By standard facts on Legendre transforms, setting $t_x := (\varphi')^{-1}(x)$, we have for all $x\in (a,b)$ 
	$$
		I(x) = t_x x - \varphi(t_x),\quad\varphi'(t_x)=x, \quad I'(x) = t_x,\quad I''(x) = \frac{1}{\varphi''(t_x)}. 
	$$
	In view of $\varphi'(0) = \mu$ and $\varphi''(0) = \sigma^2$, we have $t_\mu=0$ and 
	$$
		I(\mu) = 0,\quad I'(\mu) = 0,\quad I''(\mu) = \frac{1}{\sigma^2}, 
	$$
	a well-known property of the Cram{\'e}r rate function $I(x)$. Consequently the Taylor series of $I(x)$ at $x= \mu$ is of the form~\eqref{eq:itaylor} with 
	$$
		\lambda_j = -  \frac{1}{j!} \frac{\dd^{j}}{\dd x^{j}} I(x)\Big|_{x=\mu}. 
	$$ 
From the analyticity of the cumulant generating function $\varphi(t)$ at $t=0$, the fact $\varphi''(0) = 1 \neq 0$, and the holomorphic inverse function theorem, we know that there exist \emph{complex} open neighborhoods $U,V\subset \mathbb C$ of $t=0$ and $\mu$ respectively such that $\varphi'$ is a bijection from $U$ onto $V$ with holomorphic inverse $(\varphi')^{-1}:V\to U$. In particular, its Taylor series around $\mu$ has non-zero radius convergence. Thus 
	$$
		t(\tau):= (\varphi')^{-1}(\mu+\tau) = \sum_{k=0}^\infty b_k \tau^k
	$$
	for suitable coefficients $b_k\in \R$, and the series is absolutely convergent for sufficiently small $t$. The first two terms are easily seen to be $b_0=0$ and $b_1 = 1/\sigma^2$, thus 
	\be \label{eq:ttaylor}
			t(\tau) = \frac{\tau}{\sigma^2} + \sum_{k=2}^\infty b_k \tau^k.
	\ee
	Therefore
	\begin{align*}
		I(\mu+\tau) & = \int_0^{\tau} I'(\mu + u) \dd u = \int_0^{\tau} t(u) \dd u 
			 =  \frac{\tau^2}{2 \sigma^2} + \sum_{k=2}^\infty b_k \frac{\tau^{k+1}}{k+1}		
	\end{align*}
	hence the Taylor series of $I$ and $(\varphi')^{-1}$ around $\mu$ have the same radius of convergence, moreover 
	\begin{equation*}
		\lambda_{k} = - \frac{b_{k-1}}{k}\qquad (k \geq 3). 
	\end{equation*}
The coefficients $b_k$ from~\eqref{eq:ttaylor} are computed as follows: we must have 
$$
	\mu + \tau = \varphi'\bigl(t(\tau)\bigr) = \mu + \sigma^2 \Bigl( \sum_{k\geq 1} b_k \tau^k \Bigr)  + \sum_{r = 2}^\infty \frac{\kappa_{r+1}}{r!} \Bigl( \sum_{j=1}^\infty b_j \tau^j \Bigr) ^r
$$
for all sufficiently small $\tau$. The left- and right-hand sides are power series in $\tau$ and therefore must have all their coefficients equal. The coefficients of order zero and one are equal because $\mu = \mu$ and $1 = b_1 \sigma^2$. For orders $k\geq 2$, we obtain the equation
$$
	0= \sigma^2 b_k + \sum_{r=2}^\infty \frac{\kappa_{r+1}}{r!} \sum_{\substack{j_1,\ldots,j_r\geq 1:\\ j_1+\cdots + j_r = k}} b_{j_1}\cdots b_{j_r}.
$$
In the second summand, because of $j_\ell \geq 1$, the only relevant contributions come from $r \leq k$ and $1\leq j_1,\ldots, j_r \leq k-1$, so the recurrence relation for the $b_k$'s follows. 
\end{proof} 
An immediate consequence is the following:  If $\sigma^2 =1$, then each coefficient $\lambda_j$ is a polynomial of the cumulants $\kappa_3,\ldots,\kappa_j$. Explicit formulas for the first few coefficients $b_1,\ldots, b_4$ are given in~\cite[p.~19]{SS91}, see also~\cite[Eq.~(7.2.20)]{ibragimov-linnik}. 

It is instructive to work out an explicit bound on the radius of convergence. 
\begin{prop}\label{prop:cramer-concrete1}
	Assume $\E[X]=0$, $\mathbb V[X]=1$ and $|\kappa_j| \leq (j-2)!/\Delta^{j-2}$ for all $j\geq 3$. 
	Then the radius of convergence of the Cram{\'e}r-Petrov series is at least $\frac{3}{10}\Delta $.
\end{prop}

Notice
$
	 \frac{\sqrt{2}}{3\sqrt{\e}}\approx 0.2859 < 0.3
$
so the radius of convergence from Proposition~\ref{prop:cramer-concrete1} is slightly better than the lower bound $2/(3\sqrt e)$ to the radius of convergence to the Cram{\'e}r-Petrov series proven in \cite{SS91}.

\begin{proof} 
	We bound the radius of convergence of the series $\sum_k b_k \tau^k$ with the help of Lagrange inversion, a trick used for the virial expansion in classical statistical mechanics~\cite{lebowitz-penrose64}.
	The radius of convergence of $\sum_j \kappa_j t^j$  is at least $\Delta$ (obvious), and  for $|t|<\Delta$ we have 
	$$
		|\varphi''(t)| \leq \frac{1}{1-|t|/\Delta}.
	$$
and 
	\be\label{eq:phipeps}
		\Bigl|\frac{\varphi'(t)}{t}- 1\Bigr|
			\leq \sum_{j=3}^\infty \frac{|\kappa_j|}{(j-1)!} |t|^{j -2} \leq \sum_{j=3}^\infty \frac{1}{j-1} \frac{|t|^{j-2}}{\Delta^{j-2}}
				= \frac{1}{|t|/\Delta}\Bigl[ - \log\Bigl( 1- \frac{|t|}{\Delta}\Bigr) - \frac{|t|}{\Delta}\Bigr] =:\eps\Bigl( \frac{|t|}{\Delta}\Bigr).
	\ee
	By the Lagrange inversion formula~\cite[Appendix A.6]{flajolet-sedgewick-book}, the coefficient of $\tau^k$ in the expansion of $t(\tau)$ is equal to $1/k$ times the coefficient of $t^{k-1}$ in the expansion of $(\varphi'(t)/t)^{-k}$, which we write as 
	\be \label{eq:bkttt}
		b_k = [\tau^k] t(\tau) = \frac{1}{k} [t^{k-1}]  \Bigl( \frac{\varphi'(t)}{t}\Bigr)^{-k}. 
	\ee
	Since coefficients in convergent integrals can be extracted by complex contour integrals, it follows that 
	$$
		b_k = \frac1 k \frac{1}{2 \pi \mathrm i}\oint \frac{1}{(\varphi'(t)/t)^{k}} \frac{ \dd t}{t^k}
	$$
	with the contour of integration a circle $|t|=r$ with $r\in (0,\Delta)$. It follows that 
	$$
		|b_k|\leq \frac1{k r^{k-1}} \sup_{|t| = r} \frac{1}{|\varphi'(t)/t|^k} \qquad (r\in (0,\Delta))
	$$
	which yields, together with~\eqref{eq:phipeps},
	$$
		|b_k|\leq \frac{1}{k r^{k-1} (1- \eps(r/\Delta))^k} \qquad (r\in (0,\Delta)).
	$$
		Let us choose $r=\Delta/2$. Then 
	$$
		\eps(r/\Delta) = \eps(\frac12 ) =2( \log 2 - \frac12) = \log 4 - 1  \simeq 0.386\leq 0.4
	$$ 
	hence $r(1- \eps( r/\Delta))\geq \frac12  \Delta (1- 0.4) = 0.3\, \Delta$ and 
	\be \label{eq:bkestim}
		|b_k| \leq \frac rk \frac{1}{(r [1-\eps(\frac12)])^{k}} \leq  \frac{\Delta}{2k} \frac{1}{ (0.3\,  \Delta)^{k}}.
	\ee
	Therefore the radius of convergence of $\sum_k b_k t^k$ and $\sum_j \lambda_j t^{j}$ is at least $0.3\, \Delta$. 
\end{proof}

\subsection{\texorpdfstring{When $X$ has moments up to order $s\geq 3$.}{When X has moments up to order s equal or larger than 3.}}

More generally, we adopt the recurrence relation from Proposition~\ref{prop:cramer-petrov1} as a definition of coefficients $b_k$ and $\lambda_k$. 

\begin{definition}\label{def:cramer-petrov}
	Fix $s\geq 3$. 
	Let $X$ be a real-valued random variable with mean $\E[X] = \mu$ and variance $\sigma^2= \mathbb V(X) >0$. Assume $\mathbb E[|X|^s]< \infty$. Then we define coefficients $b_1,\ldots, b_{s-1}$ and $\lambda_2,\ldots, \lambda_s$ as follows: 
	\begin{itemize} 
		\item $b_1:= 1/\sigma^2$.
		\item $b_2,\ldots, b_{s-1}$ are defined recursively by~\eqref{eq:brecurrence}.
		\item $\lambda_k:= - b_{k-1}/k$ for $k=1,\ldots, s$. 
	\end{itemize}
	 If $s=\infty$, the series $\sum_{j=0}^\infty \lambda_{j+3} \tau^j$ is called \emph{Cram{\'e}r-Petrov series}.
\end{definition} 

\noindent A substitute for Proposition~\ref{prop:cramer-petrov1} is the following. 
Define 
$$
	\tilde \varphi(t)  := \sum_{j=1}^s \frac{\kappa_j}{j!} t^j,\quad 
	\tilde \kappa_j:= \begin{cases}
						\kappa_j, &\quad j \leq s,\\
						0, &\quad j >s.
					\end{cases} 
$$
Let $(\tilde b_k)_{k\geq 1}$ and $(\tilde \lambda_k)_{k\geq 2}$ be the coefficients defined by $\tilde b_1 = 1/\sigma^2$, $\tilde \lambda_k = - \tilde b_{k-1} /k$, and the recurrence relation~\eqref{eq:brecurrence} with $\tilde \kappa_{r+1}$ instead of $\kappa_{r+1}$ or equivalently, 
\be \label{eq:tildebk}
	\tilde b_k = - \frac{1}{\sigma^2}  \sum_{r=2}^{s-1} \frac{\kappa_{r+1}}{r!} \sum_{\substack{1\leq j_1,\ldots,j_r\leq k-1:\\ j_1+\cdots + j_r = k}} \tilde b_{j_1}\cdots \tilde b_{j_r}. 
\ee
A finite induction over $k$ shows that 
\be \label{eq:tilde-notilde}
	\forall k \in \{2,\ldots, s-1\}:\, \tilde b_k = b_k,\quad \forall k \in \{2,\ldots, s\}: \, \tilde \lambda_k = \lambda_k. 
\ee

\begin{prop} \label{prop:cramer-petrov2}
	Under the assumptions of Definition~\ref{def:cramer-petrov}, there exist open intervals $(-\eps,\eps')\ni 0$, $(\mu-\delta,\mu+ \delta)$ such that the following holds true: 
	\begin{enumerate} 
		\item [(a)] 	$\tilde \varphi'$ is a bijection from $(-\eps,\eps')$ onto $(\mu-\delta,\mu+\delta)$.
		\item [(b)] The series $\sum_{k=1}^\infty  \tilde b_k \tau^k $ and $\sum_{k = 3}^\infty \tilde \lambda_k \tau^k$ have radius of convergence larger or equal to $\delta$. 
		\item [(c)]  If $t\in (-\eps,\eps')$ and  $\tau \in (-\delta,\delta)$, then 
		$$
			 \tilde \varphi'(t) = \mu + \tau \ \Leftrightarrow \ t= \sum_{k=1}^\infty \tilde b_k \tau^k.
		$$
		\item [(d)] For $x\in (\mu-\delta,\mu+\delta)$ and $t \in (- \eps,\eps')$ the solution of $\tilde \varphi'(t) = x$, we have 
		$$
			\tilde I(x):= t x - \tilde \varphi(t) = - \frac{x^2}{2\sigma^2} + \sum_{j=3}^\infty \tilde \lambda_j x^j.		
		$$
	\end{enumerate} 
\end{prop} 

\begin{proof}
	Parts (a) and (b) follow from the holomorphic inverse function theorem. The function $\tilde \varphi':\C\to \C$ is a polynomial, so in particular holomorphic. Its derivative at $0$ is $\tilde \varphi''(0) = \frac{1}{\sigma^2} \neq 0$, and at $0$ it takes the value $\tilde \varphi'(0) = \mu$. Therefore there exist open neighborhoods $U,V\subset \mathbb C$ of $0$ and $\mu$, respectively, such that $\tilde\varphi'$, restricted to $U$, is a bijection from $U$ onto $V$ with holomorphic inverse\footnote{Strictly speaking, we should write $(\tilde\varphi'|_{U})^{-1}$, since $\tilde \varphi'$ with domain $\mathbb C$ is \emph{not} injective.} $(\varphi')^{-1}$. 
	
	Let $\delta>0$ be such that the open disk $B(\mu,\delta) := \{z \in \mathbb C\mid \ |z-\mu|<\delta\}$ is contained in $V$ and $(-\eps,\eps'):= (\tilde \varphi')^{-1}((\mu-\delta,\mu+\delta))$. 
	Then part (a) of the lemma is clearly satisfied. 
	The Taylor series 
of $(\tilde \varphi')^{-1}$ at $\mu$ has radius of convergence at least $\delta$. A reasoning completely analogous to the proof of Proposition~\ref{prop:cramer-petrov1} shows that the coefficients of the Taylor series are equal to $\tilde b_k$. Part (b) and (c) of the lemma follow. 

	For (d), we note 
	$$
		\frac{\dd}{\dd x}\tilde I(x) = (\tilde \varphi')^{-1}(x) + x \frac{1}{\tilde \varphi''( (\tilde \varphi')^{-1}(x) } - \tilde \varphi'\bigl( (\tilde \varphi')^{-1}(x)\bigr)  \frac{1}{\tilde \varphi''( (\tilde \varphi')^{-1}(x) } = (\tilde \varphi')^{-1}(x) 
	$$
	and conclude by an argument similar to the proof of Proposition~\ref{prop:cramer-petrov1}. 
\end{proof}

\noindent Eq.~\eqref{eq:bkestim} in the proof of Proposition~\ref{prop:cramer-concrete1} has a counterpart as well. 

\begin{prop} \label{prop:cramer-concrete2}
	Assume $\E[X]=0$, $\mathbb V(X) =1$, and $\kappa_j \leq (j-2)!/\Delta^{j-2}$ for all $j=3,\ldots, s$. Then 
	$$
		|\tilde b_k|\leq \frac{\Delta}{2k} \frac{1}{ (0.3\,  \Delta)^{k}}
	$$
	for all $k\geq 2$. 
\end{prop} 

\noindent The proof is similar to the proof of Eq.~\eqref{eq:bkestim} and is omitted. 

We conclude with a representation of the coefficients $\lambda_k$ needed in the proof of Proposition~\ref{prop:cracob} below. We assume $\E[X]=0$ and $\mathbb V(X)=1$. Eq.~\eqref{eq:tilde-notilde} and the analogue of~\eqref{eq:bkttt} for $\tilde b_k$ instead of $b_k$ yields
\be \label{eq:bkt4}
		\lambda_k = - \frac{\tilde b_{k-1}}{k} = - \frac{1}{k(k-1)} [t^{k-2}] \Bigl( \frac{\tilde \varphi'(t)}{t}\Bigr)^{-(k-1)} = -  \frac{1}{k(k-1)} [t^{k-2}] \Bigl( 1 + \sum_{j=3}^s \frac{\kappa_j}{(j-1)!} t^{j-2}\Bigr)^{-(k-1)}
\ee
for all $k\leq s$. 

\subsection{\texorpdfstring{Bounds under the Statulevi{\v{c}}ius condition}{Bounds under the Statulevicius condition}}

\begin{prop} \label{prop:cracob} 
	Under condition~\ref{condition-sgamma}, the coefficients $\lambda_k$ of the Cram{\'e}r-Petrov series $\sum_{k\geq 3} \lambda_k x^{k}$ satisfy 
	\[
		|\lambda_k|\leq \frac{(k+2)!^\gamma}{(\Delta/15)^{k-2}} \quad (k\geq 3). 
	\]
\end{prop} 

\noindent The proposition is a slightly improved version of~\cite[Eq.~(2.67)]{SS91}. 

\begin{proof}
	By Eq.~\eqref{eq:bkt4} (applied to $s=k$), we have
	\[
		\lambda_k = - \frac{1}{k(k-1)} [t^{k-2}]\Bigl(1+\sum_{j=3}^k \frac{\kappa_j}{(j-1)!} t^{j-2}\Bigr)^{-(k-1)}.
	\] 
	Let $g_k>0$ be such that $|\kappa_j|\leq j! / g_k^{j-2}$ for all $j=3,\ldots,k$; a bound for $g_k$ is given shortly. Then Cauchy's inequality yields 
	\[
		|\lambda_k| \leq \frac{1}{k(k-1)}\, \inf_{r}  r^{-(k-2)} \Bigl( 1 - \sum_{j=3}^k j\bigl( \frac{r}{g_k}\bigr)^{j-2}\Bigr)^{-(k-1)}. 
	\] 
	The infimum is over intervals $r\in [0, \alpha]$ on which the denominator is non-zero. We write $r= \rho\, g_k$, bound the sum by a series, and obtain 
	 \[
		|\lambda_k| \leq \frac{g_k^{-(k-2)}}{k(k-1)}\, \inf_{\rho}  \rho^{-(k-2)} \Bigl( 1 - \sum_{j=3}^\infty j\rho^{j-2}\Bigr)^{-(k-1)}. 
	\] 
	A numerical evaluation yields 
	\[
		\inf_\rho \rho^{-1} \Bigl( 1 - \sum_{j=3}^\infty j\rho^{j-2}\Bigr)^{-1} \simeq 14.5\leq 15
	\] 	
	and the minimum is attained at $\rho\simeq 	0.126\leq 0.13$, therefore 
	\be \label{eq:facto1}
		|\lambda_k|\leq \frac{g_k^{-(k-2)}}{k(k-1)}\, 0.13\cdot 15^{k-1}\leq 2\, \frac{(15/g_k)^{(k-2)}}{k(k-1)}.
	\ee
	In order to get a bound on $g_k$, we check that 
	\be \label{eq:factobound}
		j!^{1/(j-2)}\leq (k+2)!^{1/(k-2)} \quad (j=3,\ldots,k)
	\ee
	or equivalently, $j!^{k-2}\leq (k+2)!^{j-2}$. The proof is by induction over $k\geq j$ at fixed $j\geq 3$. For $k=j$, the claim is obviously true. For the induction step, we note that for all $k\geq j\geq 3$, we have $j! \leq 3!\, j^{j-3}$ hence 
	\[
		\frac{j!}{(k+3)^{j-2}} \leq \frac{3!\, j^{j-3}}{(j+3)^{j-2}} \leq \frac{3!}{j+3} \leq 1.
	\]	
	Therefore, under the induction hypothesis $(k+2)!^{j-2}\geq j!^{k-2}$, we have
	\[
		(k+3)!^{j-2} = (k+3)^{j-2} (k+2)!^{j-2} \geq (k+3)^{j-2} j!^{k-2} \geq j!^{k-1}.
	\]
	This completes the induction. 
	Condition~\eqref{condition-sgamma} and Eq.~\eqref{eq:factobound} yield $|\kappa_j|\leq j!/g_k^{j-2}$ for all $j\leq k$ by choosing
	\[
		\frac{1}{g_k} = \frac{1}{\Delta} (k+2)!^{\gamma/(k-2)}. 		
	\] 	
	We deduce from~\eqref{eq:facto1} that 
	\[
		|\lambda_k|\leq \frac{2}{k(k-1)}		
		 \, (15/\Delta)^{k-2} (k+2)!^{\gamma} \leq (15/\Delta)^{k-2} (k+2)!^{\gamma}. \qedhere
	\]
\end{proof}

\subsubsection*{Acknowledgments} 
We thank Zakhar Kabluchko, Christoph Th{\"a}le and all members of the DFG scientific network \emph{Cumulants, concentration and superconcentration}, as well as the participants of the workshops, for helpful discussions. This work is funded by the scientific network \emph{Cumulants, concentration and superconcentration} by the DFG (German Research Foundation) -- project number 318196255.

\medskip 

\nocite{*}
\bibliographystyle{plain}
\bibliography{cumulants}

\begin{thebibliography}{100}

\bibitem{ameur2011}
Y.~Ameur, H.~Hedenmalm, and N.~Makarov.
\newblock Fluctuations of eigenvalues of random normal matrices.
\newblock {\em Duke Math. J.}, 159(1):31--81, 07 2011.

\bibitem{amosova99}
N.~N. Amosova.
\newblock Necessity of the {C}ram\'{e}r, {L}innik and {S}tatulevi\v{c}ius
  conditions for the probabilities of large deviations.
\newblock {\em Zap. Nauchn. Sem. S.-Peterburg. Otdel. Mat. Inst. Steklov.
  (POMI)}, 260(Veroyatn. i Stat. 3):9--16, 317, 1999.
\newblock Translated in J. Math. Sci. (New York) 109 (2002), no. 6,
  2031–2036.

\bibitem{amosova99a}
N.~N. Amosova.
\newblock On the necessity of the {S}tatulevi\v{c}ius condition in limit
  theorems for probabilities of large deviations.
\newblock {\em Liet. Mat. Rink.}, 39(3):293--303, 1999.
\newblock Translation in Lithuanian Math. J. 39 (1999), no. 3, 231–239.

\bibitem{Arc94}
M.~A. Arcones.
\newblock Limit theorems for nonlinear functionals of a stationary {G}aussian
  sequence of vectors.
\newblock {\em Ann. Probab.}, 22(4):2242--2274, 1994.

\bibitem{AHLV2015}
O.~Arizmendi, T.~Hasebe, F.~Lehner, and C.~Vargas.
\newblock Relations between cumulants in noncommutative probability.
\newblock {\em Adv. Math.}, 282:56--92, 2015.

\bibitem{Barbour}
R.~Arratia, A.~D. Barbour, and S.~Tavar\'{e}.
\newblock {\em Logarithmic combinatorial structures: a probabilistic approach}.
\newblock EMS Monographs in Mathematics. European Mathematical Society (EMS),
  Z\"{u}rich, 2003.

\bibitem{bahadur-rao60}
R.~R. Bahadur and R.~Ranga~Rao.
\newblock On deviations of the sample mean.
\newblock {\em Ann. Math. Statist.}, 31:1015--1027, 1960.

\bibitem{barbour-kowalski-nikeghbali2014}
A.~D. Barbour, E.~Kowalski, and A.~Nikeghbali.
\newblock Mod-discrete expansions.
\newblock {\em Probab. Theory Related Fields}, 158(3-4):859--893, 2014.

\bibitem{barndorff-nielsen-cox}
O.~E. Barndorff-Nielsen and D.~R. Cox.
\newblock {\em Asymptotic techniques for use in statistics}.
\newblock Monographs on Statistics and Applied Probability. Chapman \& Hall,
  London, 1989.

\bibitem{baryshnikov-yukich2005}
Yu. Baryshnikov and J.~E. Yukich.
\newblock Gaussian limits for random measures in geometric probability.
\newblock {\em Ann. Appl. Probab.}, 15(1A):213--253, 2005.

\bibitem{bentkus}
R.~Bentkus and R.~Rudzkis.
\newblock On exponential estimates of the distribution of random variables.
\newblock {\em Lithuanian Math. J.}, 20:15--30, 01 1980.

\bibitem{berry41}
A.~C. Berry.
\newblock The accuracy of the {G}aussian approximation to the sum of
  independent variates.
\newblock {\em Trans. Amer. Math. Soc.}, 49:122--136, 1941.

\bibitem{BYY2018}
B.~B{\l}aszczyszyn, D.~Yogeshwaran, and J.~E. Yukich.
\newblock Limit theory for geometric statistics of point processes having fast
  decay of correlations.
\newblock {\em Ann. Probab.}, 47(2):835--895, 2019.

\bibitem{bobkov2016}
S.~G. Bobkov.
\newblock Closeness of probability distributions in terms of
  {F}ourier-{S}tieltjes transforms.
\newblock {\em Uspekhi Mat. Nauk}, 71(6(432)):37--98, 2016.
\newblock translation in Russian Math. Surveys 71 (2016), no. 6, 1021–1079.

\bibitem{bodineau-straymond-hardspheres}
T.~Bodineau, I.~Gallagher, L.~Saint-Raymond, and S.~Simonella.
\newblock Statistical dynamics of a hard sphere gas: fluctuating {B}oltzmann
  equation and large deviations.
\newblock Online preprint arXiv:2008.10403, 2020.

\bibitem{borot}
G.~Borot and A.~Guionnet.
\newblock Asymptotic expansion of {$\beta$} matrix models in the one-cut
  regime.
\newblock {\em Comm. Math. Phys.}, 317(2):447--483, 2013.

\bibitem{brillinger1975}
D.~R. Brillinger.
\newblock Statistical inference for stationary point processes.
\newblock In {\em Stochastic processes and related topics ({P}roc. {S}ummer
  {R}es. {I}nst. {S}tatist. {I}nference for {S}tochastic {P}rocesses, {I}ndiana
  {U}niv., {B}loomington, {I}nd., 1974, {V}ol. 1; dedicated to {J}erzy
  {N}eyman)}, pages 55--99, 1975.

\bibitem{bryc1993}
W.~Bryc.
\newblock A remark on the connection between the large deviation principle and
  the central limit theorem.
\newblock {\em Statist. Probab. Lett.}, 18(4):253--256, 1993.

\bibitem{CNN17}
R.~Chhaibi, J.~Najnudel, and A.~Nikeghbali.
\newblock The circular unitary ensemble and the {R}iemann zeta function: the
  microscopic landscape and a new approach to ratios.
\newblock {\em Invent. Math.}, 207(1):23--113, 2017.

\bibitem{costin}
O.~Costin and J.~L. Lebowitz.
\newblock Gaussian fluctuation in random matrices.
\newblock {\em Phys. Rev. Lett.}, 75:69--72, Jul 1995.

\bibitem{cramer38}
H.~Cram{\'e}r.
\newblock Sur un nouveau th{\'e}or{\`e}me-limite de la th{\'e}orie des
  probabilit{\'e}s.
\newblock {\em Actual. Sci. Industr.}, 736:5--23, 1938.

\bibitem{cramer-touchette2018}
H.~Cram{\'e}r and H.~Touchette~(translator).
\newblock On a new limit theorem in probability theory. (translation of: Sur un
  nouveau th{\'e}or{\`e}me-limite de la th{\'e}orie des probabilit{\'e}s).
\newblock Electronic preprint arXiv:1802.05988v3 [math.HO], 2018.

\bibitem{DelbaenModPhi}
F.~Delbaen, E.~Kowalski, and A.~Nikeghbali.
\newblock Mod-{$\varphi$} convergence.
\newblock {\em Int. Math. Res. Not. IMRN}, 2015(11):3445--3485, 2014.

\bibitem{dembo-zeitouni}
A.~Dembo and O.~Zeitouni.
\newblock {\em Large deviations techniques and applications}, volume~38 of {\em
  Applications of Mathematics (New York)}.
\newblock Springer-Verlag, New York, second edition, 1998.

\bibitem{denisov-dieker-shneer}
D.~Denisov, A.~B. Dieker, and V.~Shneer.
\newblock Large deviations for random walks under subexponentiality: the
  big-jump domain.
\newblock {\em Ann. Probab.}, 36(5):1946--1991, 2008.

\bibitem{DetteT:2019}
H.~Dette and D.~Tomecki.
\newblock Determinants of block {H}ankel matrices for random matrix-valued
  measures.
\newblock {\em Stochastic Process. Appl.}, 129(12):5200--5235, 2019.

\bibitem{dobrushin-shlosman1987}
R.~L. Dobrushin and S.~B. Shlosman.
\newblock Completely analytical interactions: constructive description.
\newblock {\em J. Statist. Phys.}, 46(5-6):983--1014, 1987.

\bibitem{DEi4}
H.~D\"{o}ring and P.~Eichelsbacher.
\newblock Edge fluctuations of eigenvalues of {W}igner matrices.
\newblock In {\em High dimensional probability {VI}}, volume~66 of {\em Progr.
  Probab.}, pages 261--275. Birkh\"{a}user/Springer, Basel, 2013.

\bibitem{DEi1}
H.~D\"{o}ring and P.~Eichelsbacher.
\newblock Moderate deviations for the eigenvalue counting function of {W}igner
  matrices.
\newblock {\em ALEA Lat. Am. J. Probab. Math. Stat.}, 10(1):27--44, 2013.

\bibitem{DEi3}
H.~D\"{o}ring and P.~Eichelsbacher.
\newblock Moderate deviations via cumulants.
\newblock {\em J. Theoret. Probab.}, 26(2):360--385, 2013.

\bibitem{doukhan}
P.~Doukhan and M.~H. Neumann.
\newblock Probability and moment inequalities for sums of weakly dependent
  random variables, with applications.
\newblock {\em Stochastic Processes and their Applications}, 117(7):878--903,
  2007.

\bibitem{DousseF:2019}
J.~Dousse and V.~F\'{e}ray.
\newblock Weighted dependency graphs and the {I}sing model.
\newblock {\em Ann. Inst. Henri Poincar\'{e} D}, 6(4):533--571, 2019.

\bibitem{duneau-iagolnitzer-souillard1973}
M.~Duneau, D.~Iagolnitzer, and B.~Souillard.
\newblock Decrease properties of truncated correlation functions and
  analyticity properties for classical lattices and continuous systems.
\newblock {\em Comm. Math. Phys.}, 31:191--208, 1973.

\bibitem{EichelsbacherK19}
P.~Eichelsbacher and L.~Knichel.
\newblock Moment estimates of rosenthal type via cumulants, 2019.
\newblock arXiv:1901.04865.

\bibitem{EichelsbacherKnichel21}
P.~Eichelsbacher and L.~Knichel.
\newblock Fine asymptotics for models with {G}amma type moments.
\newblock {\em Random Matrices Theory Appl.}, 10(1):2150007, 51, 2021.

\bibitem{EichRaicSchreiber}
P.~Eichelsbacher, M.~Rai\v{c}, and T.~Schreiber.
\newblock Moderate deviations for stabilizing functionals in geometric
  probability.
\newblock {\em Ann. Inst. Henri Poincar\'{e} Probab. Stat.}, 51(1):89--128,
  2015.

\bibitem{embrechts-klueppelberg-mikosch}
P.~Embrechts, C.~Kl\"{u}ppelberg, and T.~Mikosch.
\newblock {\em Modelling extremal events}, volume~33 of {\em Applications of
  Mathematics (New York)}.
\newblock Springer-Verlag, Berlin, 1997.
\newblock For insurance and finance.

\bibitem{ercolani-jansen-ueltschi2019}
N.~M. Ercolani, S.~Jansen, and D.~Ueltschi.
\newblock Singularity analysis for heavy-tailed random variables.
\newblock {\em J. Theoret. Probab.}, 32(1):1--46, 2019.

\bibitem{feller-vol2}
W.~Feller.
\newblock {\em An introduction to probability theory and its applications.
  {V}ol. {II}}.
\newblock Second edition. John Wiley \& Sons, Inc., New York-London-Sydney,
  1971.

\bibitem{lambert2020}
M.~Fenzl and G.~Lambert.
\newblock Precise deviations for disk counting statistics of invariant
  determinantal processes, 2020.

\bibitem{Feray:2018}
V.~F{\'e}ray.
\newblock Weighted dependency graphs.
\newblock {\em Electron. J. Probab.}, 23:Paper No. 93, 65, 2018.

\bibitem{Feray:2020}
V.~F{\'e}ray.
\newblock Central limit theorems for patterns in multiset permutations and set
  partitions.
\newblock {\em Ann. Appl. Probab.}, 30(1):287--323, 2020.

\bibitem{FMNbook}
V.~F\'{e}ray, P.-L. M\'{e}liot, and A.~Nikeghbali.
\newblock {\em Mod-{$\phi$} convergence}.
\newblock SpringerBriefs in Probability and Mathematical Statistics. Springer,
  Cham, 2016.
\newblock Normality zones and precise deviations.

\bibitem{feray-meliot-nikeghbali2019}
V.~F{\'e}ray, P.-L. M{\'e}liot, and A.~Nikeghbali.
\newblock Mod-$\phi$ convergence, {II}: Estimates on the speed of convergence.
\newblock In C.~Donati-Martin, A.~Lejay, and A.~Rouault, editors, {\em
  S{\'e}minaire de Probabilit{\'e}s L}, pages 405--477, Cham, 2019. Springer
  International Publishing.

\bibitem{FMN17}
V.~F\'{e}ray, P.-L. M\'{e}liot, and A.~Nikeghbali.
\newblock Graphons, permutons and the thoma simplex: three mod-gaussian moduli
  spaces.
\newblock {\em Proceedings of the London Mathematical Society},
  121(4):876--926, 2020.

\bibitem{FW32}
R.~A. Fisher and J.~Wishart.
\newblock The derivation of the pattern formulae of two-way partitions from
  those of simpler patterns.
\newblock {\em Proc. London Math. Soc. (2)}, 33(3):195--208, 1931.

\bibitem{flajolet-sedgewick-book}
P.~Flajolet and R.~Sedgewick.
\newblock {\em Analytic combinatorics}.
\newblock Cambridge University Press, Cambridge, 2009.

\bibitem{friedli_velenik}
S.~Friedli and Y.~Velenik.
\newblock {\em Statistical mechanics of lattice systems: a concrete
  mathematical introduction}.
\newblock Cambridge University Press, 2017.

\bibitem{gnedenko-kolmogorov}
B.~V. Gnedenko and A.~N. Kolmogorov.
\newblock {\em Limit distributions for sums of independent random variables}.
\newblock Translated from the Russian, annotated, and revised by K. L. Chung.
  With appendices by J. L. Doob and P. L. Hsu. Revised edition. Addison-Wesley
  Publishing Co., Reading, Mass.-London-Don Mills., Ont., 1968.

\bibitem{gordon}
R.~D. Gordon.
\newblock Values of {M}ills' ratio of area to bounding ordinate and of the
  normal probability integral for large values of the argument.
\newblock {\em Ann. Math. Statist.}, 12(3):364--366, 09 1941.

\bibitem{GoetzeHeinrichHipp95}
F.~G\"{o}tze, L.~Heinrich, and C.~Hipp.
\newblock {$m$}-dependent random fields with analytic cumulant generating
  function.
\newblock {\em Scand. J. Statist.}, 22(2):183--195, 1995.

\bibitem{GKT17}
J.~Grote, Z.~Kabluchko, and Ch. Th\"{a}le.
\newblock Limit theorems for random simplices in high dimensions.
\newblock {\em ALEA Lat. Am. J. Probab. Math. Stat.}, 16(1):141--177, 2019.

\bibitem{GroteThaele:2015}
J.~Grote and Ch. Th\"{a}le.
\newblock Concentration and moderate deviations for {P}oisson polytopes and
  polyhedra.
\newblock {\em Bernoulli}, 24(4A):2811--2841, 2018.

\bibitem{GT2016}
J.~Grote and Ch. Th\"{a}le.
\newblock Gaussian polytopes: a cumulant-based approach.
\newblock {\em J. Complexity}, 47:1--41, 2018.

\bibitem{GruebelKabluchko}
R.~Gr\"{u}bel and Z.~Kabluchko.
\newblock Edgeworth expansions for profiles of lattice branching random walks.
\newblock {\em Ann. Inst. Henri Poincar\'{e} Probab. Stat.}, 53(4):2103--2134,
  2017.

\bibitem{GT:2021}
A.~Gusakova and Ch. Th\"{a}le.
\newblock The volume of simplices in high-dimensional poisson--delaunay
  tessellations.
\newblock {\em Annales Henri Lebesgue}, 4:121--153, 2021.

\bibitem{hald2000}
A.~Hald.
\newblock The early history of the cumulants and the gram-charlier series.
\newblock {\em International Statistical Review}, 68(2):137--153, 2000.

\bibitem{Heinrich1985}
L.~Heinrich.
\newblock Some estimates of the cumulant-generating function of a sum of
  {$m$}-dependent random vectors and their application to large deviations.
\newblock {\em Math. Nachr.}, 120:91--101, 1985.

\bibitem{heinrich1987survey}
L.~Heinrich.
\newblock A method for the derivation of limit theorems for sums of weakly
  dependent random variables: a survey.
\newblock {\em Optimization}, 18(5):715--735, 1987.

\bibitem{Heinrich1990}
L.~Heinrich.
\newblock Some bounds of cumulants of {$m$}-dependent random fields.
\newblock {\em Math. Nachr.}, 149:303--317, 1990.

\bibitem{Heinrich2005}
L.~Heinrich.
\newblock Large deviations of the empirical volume fraction for stationary
  {P}oisson grain models.
\newblock {\em Ann. Appl. Probab.}, 15(1A):392--420, 2005.

\bibitem{Heinrich2007}
L.~Heinrich.
\newblock An almost-{M}arkov-type mixing condition and large deviations for
  {B}oolean models in the line.
\newblock {\em Acta Appl. Math.}, 96(1-3):247--262, 2007.

\bibitem{Heinrich2016}
L.~Heinrich.
\newblock On the strong {B}rillinger-mixing property of
  {$\alpha$}-determinantal point processes and some applications.
\newblock {\em Appl. Math.}, 61(4):443--461, 2016.

\bibitem{HeinrichRichter}
L.~Heinrich and W.-D. Richter.
\newblock On moderate deviations of sums of {$m$}-dependent random vectors.
\newblock {\em Math. Nachr.}, 118:253--263, 1984.

\bibitem{HeinrichSpiess}
L.~Heinrich and M.~Spiess.
\newblock Berry-{E}sseen bounds and {C}ram\'{e}r-type large deviations for the
  volume distribution of {P}oisson cylinder processes.
\newblock {\em Lith. Math. J.}, 49(4):381--398, 2009.

\bibitem{HeinrichSpiess13}
L.~Heinrich and M.~Spiess.
\newblock Central limit theorems for volume and surface content of stationary
  {P}oisson cylinder processes in expanding domains.
\newblock {\em Adv. in Appl. Probab.}, 45(2):312--331, 2013.

\bibitem{Hofer:2017}
L.~Hofer.
\newblock A central limit theorem for vincular permutation patterns.
\newblock {\em Discrete Math. Theor. Comput. Sci.}, 19(2):Paper No. 9, 26,
  2017.

\bibitem{hwang1998}
H.-K. Hwang.
\newblock On convergence rates in the central limit theorems for combinatorial
  structures.
\newblock {\em European J. Combin.}, 19(3):329--343, 1998.

\bibitem{iagolnitzer-souillard1979}
D.~Iagolnitzer and B.~Souillard.
\newblock Lee-{Y}ang theory and normal fluctuations.
\newblock {\em Phys. Rev. B (3)}, 19(3):1515--1518, 1979.

\bibitem{ibragimov-linnik}
I.~A. Ibragimov and Yu.~V. Linnik.
\newblock {\em Independent and stationary sequences of random variables}.
\newblock Wolters-Noordhoff Publishing, Groningen, 1971.
\newblock With a supplementary chapter by I. A. Ibragimov and V. V. Petrov,
  Translation from the Russian edited by J. F. C. Kingman.

\bibitem{ivanoff1982}
G.~Ivanoff.
\newblock Central limit theorems for point processes.
\newblock {\em Stochastic Process. Appl.}, 12(2):171--186, 1982.

\bibitem{JacodModPhi}
J.~Jacod, E.~Kowalski, and A.~Nikeghbali.
\newblock Mod-{G}aussian convergence: new limit theorems in probability and
  number theory.
\newblock {\em Forum Math.}, 23(4):835--873, 2011.

\bibitem{JansonCLT}
S.~Janson.
\newblock Normal convergence by higher semi-invariants with applications to
  sums of dependent random variables and random graphs.
\newblock {\em Ann. Probab.}, 16(1):305--312, 1988.

\bibitem{KabluchkoVysotskyZaporozhets}
Z.~Kabluchko, V.~Vysotsky, and D.~Zaporozhets.
\newblock Convex hulls of random walks: expected number of faces and face
  probabilities.
\newblock {\em Adv. Math.}, 320:595--629, 2017.

\bibitem{kallabis-neumann2006}
R.~S. Kallabis and M.~H. Neumann.
\newblock An exponential inequality under weak dependence.
\newblock {\em Bernoulli}, 12(2):333--350, 2006.

\bibitem{keating-snaith2000}
J.~P. Keating and N.~C. Snaith.
\newblock Random matrix theory and {$\zeta(1/2+it)$}.
\newblock {\em Comm. Math. Phys.}, 214(1):57--89, 2000.

\bibitem{KowalskiModPhi10}
E.~Kowalski and A.~Nikeghbali.
\newblock Mod-{P}oisson convergence in probability and number theory.
\newblock {\em Int. Math. Res. Not. IMRN}, 2010(18):3549--3587, 2010.

\bibitem{KowalskiModPhi12}
E.~Kowalski and A.~Nikeghbali.
\newblock Mod-{G}aussian convergence and the value distribution of
  {$\zeta(\frac12+it)$} and related quantities.
\newblock {\em J. Lond. Math. Soc. (2)}, 86(1):291--319, 2012.

\bibitem{lambert}
G.~Lambert.
\newblock Limit theorems for biorthogonal ensembles and related combinatorial
  identities.
\newblock {\em Advances in Mathematics}, 329:590--648, 2018.

\bibitem{lambert2018}
G.~Lambert.
\newblock Mesoscopic fluctuations for unitary invariant ensembles.
\newblock {\em Electron. J. Probab.}, 23:33 pp., 2018.

\bibitem{lebowitz-penrose64}
J.~L. Lebowitz and O.~Penrose.
\newblock Convergence of virial expansions.
\newblock {\em J. Mathematical Phys.}, 5:841--847, 1964.

\bibitem{lebowitz-pittel-ruelle-speer2016}
J.~L. Lebowitz, B.~Pittel, D.~Ruelle, and E.~R. Speer.
\newblock Central limit theorems, {L}ee-{Y}ang zeros, and graph-counting
  polynomials.
\newblock {\em J. Combin. Theory Ser. A}, 141:147--183, 2016.

\bibitem{LeonovS1959}
V.~P. Leonov and A.~N. {\v{S}}irjaev.
\newblock On a method of semi-invariants.
\newblock {\em Theor. Probability Appl.}, 4:319--329, 1959.
\newblock Translated from Teor.\ Verojatnost.\ i Primenen.~\textbf{4} (1959),
  342--355.

\bibitem{linnik61b}
Yu.~V. Linnik.
\newblock Limit theorems for sums of independent variables, taking large
  deviations into account. {II}.
\newblock {\em Teor. Verojatnost. i Primenen.}, 6, 1961.

\bibitem{linnik61a}
Yu.~V. Linnik.
\newblock Limit theorems for the sums of independent variables taking into
  account the large deviations. {I}.
\newblock {\em Teor. Verojatnost. i Primenen.}, 6:145--163, 1961.

\bibitem{linnik62}
Yu.~V. Linnik.
\newblock Limit theorems for sums of independent quantities, taking large
  deviations into account. {III}.
\newblock {\em Teor. Verojatnost. i Primenen}, 7:121--134, 1962.

\bibitem{lukkarinen-marcozzi-nota}
J.~Lukkarinen, M.~Marcozzi, and A.~Nota.
\newblock Summability of connected correlation functions of coupled lattice
  fields.
\newblock {\em J. Stat. Phys.}, 171(2):189--206, 2018.

\bibitem{mason-zhou2012}
D.~M. Mason and H.~H. Zhou.
\newblock Quantile coupling inequalities and their applications.
\newblock {\em Probab. Surv.}, 9:439--479, 2012.

\bibitem{McCullagh87}
P.~McCullagh.
\newblock {\em Tensor methods in statistics}.
\newblock Monographs on Statistics and Applied Probability. Chapman \& Hall,
  London, 1987.

\bibitem{MeliotModPhi}
P.-L. M\'{e}liot and A.~Nikeghbali.
\newblock Mod-{G}aussian convergence and its applications for models of
  statistical mechanics.
\newblock In {\em In memoriam {M}arc {Y}or---{S}\'{e}minaire de
  {P}robabilit\'{e}s {XLVII}}, volume 2137 of {\em Lecture Notes in Math.},
  pages 369--425. Springer, Cham, 2015.

\bibitem{michelen-saha2019b}
M.~Michelen and J.~Sahasrabudhe.
\newblock Central limit theorems and the geometry of polynomials.
\newblock Online preprint arXiv:1908.09020 [math.PR], 2019.

\bibitem{michelen-saha2019}
M.~Michelen and J.~Sahasrabudhe.
\newblock Central limit theorems from the roots of probability generating
  functions.
\newblock {\em Adv. Math.}, 358:106840, 27, 2019.

\bibitem{mikosch-nagaev1998}
T.~Mikosch and A.~V. Nagaev.
\newblock Large deviations of heavy-tailed sums with applications in insurance.
\newblock {\em Extremes}, 1(1):81--110, 1998.

\bibitem{MoehlePitters}
M.~M\"{o}hle and H.~Pitters.
\newblock Absorption time and tree length of the {K}ingman coalescent and the
  {G}umbel distribution.
\newblock {\em Markov Process. Related Fields}, 21(2):317--338, 2015.

\bibitem{nagaev68}
A.~V. Nagaev.
\newblock Local limit theorems with regard to large deviations when
  {C}ram\'{e}r's condition is not satisfied.
\newblock {\em Litovsk. Mat. Sb.}, 8:553--579, 1968.
\newblock Selected {T}ransl. in {M}ath. {S}tat. {P}robab. 11, 249–278 (1973).

\bibitem{nagaev1969}
S.~V. Nagaev.
\newblock Some limit theorems for large deviations.
\newblock {\em Teor. Verojatnost. i Primenen}, 10:231--254, 1965.
\newblock English translation in Theor.\ Probability Appl.~10 (1965),
  214–-235.

\bibitem{NP2009}
I.~Nourdin and G.~Peccati.
\newblock Cumulants on the {W}iener space.
\newblock {\em J. Funct. Anal.}, 258(11):3775--3791, 2010.

\bibitem{PWZ:2017}
G.~Pan, S.~Wang, and W.~Zhou.
\newblock Limit theorems for linear spectrum statistics of orthogonal
  polynomial ensembles and their applications in random matrix theory.
\newblock {\em J. Math. Phys.}, 58(10):103301, 2017.

\bibitem{PeccatiTaqqu}
G.~Peccati and M.~S. Taqqu.
\newblock {\em Wiener chaos: moments, cumulants and diagrams}, volume~1 of {\em
  Bocconi \& Springer Series}.
\newblock Springer, Milan; Bocconi University Press, Milan, 2011.
\newblock A survey with computer implementation, Supplementary material
  available online.

\bibitem{pemantle-wilson}
R.~Pemantle and M.~C. Wilson.
\newblock {\em Analytic combinatorics in several variables}, volume 140 of {\em
  Cambridge Studies in Advanced Mathematics}.
\newblock Cambridge University Press, Cambridge, 2013.

\bibitem{petrov54}
V.~V. Petrov.
\newblock Generalization of {C}ram\'{e}r's limit theorem.
\newblock {\em Uspehi Matem. Nauk (N.S.)}, 9(4(62)):195--202, 1954.

\bibitem{petrov-book}
V.~V. Petrov.
\newblock {\em Limit theorems of probability theory}, volume~4 of {\em Oxford
  Studies in Probability}.
\newblock The Clarendon Press, Oxford University Press, New York, 1995.
\newblock Sequences of independent random variables, Oxford Science
  Publications.

\bibitem{Pitters2017}
H.~Pitters.
\newblock On the number of segregating sites, 2017.
\newblock arXiv:1708.05634.

\bibitem{Pitters2019}
H.~Pitters.
\newblock The number of cycles in a random permutation and the number of
  segregating sites jointly converge to the {B}rownian sheet, 2019.
\newblock arXiv:1903.04906.

\bibitem{rider2007}
B.~Rider and B.~Virág.
\newblock {The Noise in the Circular Law and the Gaussian Free Field}.
\newblock {\em International Mathematics Research Notices}, 2007, 01 2007.
\newblock rnm006.

\bibitem{Robbins}
H.~Robbins.
\newblock A remark on {S}tirling’s formula.
\newblock {\em American Mathematical Monthly}, 62:402--405, 1955.

\bibitem{RotaShen:2000}
G.-C. Rota and J.~Shen.
\newblock On the combinatorics of cumulants.
\newblock {\em J. Combin. Theory Ser. A}, 91(1-2):283--304, 2000.
\newblock In memory of Gian-Carlo Rota.

\bibitem{Rudzkis_Bakshaev}
R.~Rudzkis and A.~Bakshaev.
\newblock General theorems on large deviations for random vectors.
\newblock {\em Lith. Math. J.}, 57(3):367--390, 2017.

\bibitem{rudzkis-saulis-statulevicius78}
R.~Rudzkis, L.~Saulis, and V.~Statuljavi\v{c}us.
\newblock A general lemma on probabilities of large deviations.
\newblock {\em Litovsk. Mat. Sb.}, 18(2):99--116, 217, 1978.
\newblock Translated in Lithuanian Math. J. 18 (1978), no. 2, 226–238 (1979).

\bibitem{ruelle-book}
D.~Ruelle.
\newblock {\em Statistical mechanics: {R}igorous results}.
\newblock W. A. Benjamin, Inc., New York-Amsterdam, 1969.

\bibitem{saulis73}
L.~Saulis.
\newblock Limit theorems that take into account large deviations in the case
  when {J}u. {V}. {L}innik's condition is satisfied.
\newblock {\em Litovsk. Mat. Sb.}, 13(4), 1973.

\bibitem{SS91}
L.~Saulis and V.~A. Statulevi\v{c}ius.
\newblock {\em Limit theorems for large deviations}, volume~73 of {\em
  Mathematics and its Applications (Soviet Series)}.
\newblock Kluwer Academic Publishers Group, Dordrecht, 1991.
\newblock Translated and revised from the 1989 Russian original.

\bibitem{SchulteThaele:2014}
M.~Schulte and Ch. Th\"{a}le.
\newblock Cumulants on {W}iener chaos: moderate deviations and the fourth
  moment theorem.
\newblock {\em J. Funct. Anal.}, 270(6):2223--2248, 2016.

\bibitem{schuetzenberger1954}
P.~M. Sch{\"u}tzenberger.
\newblock Contribution aux applications statistiques de la th{\'e}orie de
  l\'{}information.
\newblock Publ. Inst. Statist. Univ. Paris, 1954.
\newblock {T}h{\`e}se d’{\'E}tat).

\bibitem{scott-sokal2005}
A.~D. Scott and A.~D. Sokal.
\newblock The repulsive lattice gas, the independent-set polynomial, and the
  {L}ov\'{a}sz local lemma.
\newblock {\em J. Stat. Phys.}, 118(5-6):1151--1261, 2005.

\bibitem{soshnikov1998}
A.~Soshnikov.
\newblock Level spacings distribution for large random matrices: Gaussian
  fluctuations.
\newblock {\em Annals of Mathematics}, 148(2):573--617, 1998.

\bibitem{Soshnikov:2000}
A.~Soshnikov.
\newblock The central limit theorem for local linear statistics in classical
  compact groups and related combinatorial identities.
\newblock {\em Ann. Probab.}, 28(3):1353--1370, 2000.

\bibitem{Soshnikov:2002}
A.~Soshnikov.
\newblock Gaussian limit for determinantal random point fields.
\newblock {\em Ann. Probab.}, 30(1):171--187, 2002.

\bibitem{Speed:1983}
T.~P. Speed.
\newblock Cumulants and partition lattices.
\newblock {\em Australian Journal of Statistics}, 25(2):378--388, 1983.

\bibitem{Statulevicius1966}
V.~A. Statulevi\v{c}ius.
\newblock On large deviations.
\newblock {\em Z. Wahrscheinlichkeitstheorie und Verw. Gebiete}, 6:133--144,
  1966.

\bibitem{SturmfelsZwiernik:2013}
B.~Sturmfels and P.~Zwiernik.
\newblock {Binary cumulant varieties}.
\newblock {\em Annals of combinatorics}, 17(1):229--250, 2013.

\bibitem{Thiele1889}
T.N. Thiele.
\newblock Forel{\ae}sninger over almindelig iagttagelsesl{\ae}re:
  Sandsynlighedsregning og mindste kvadraters methode.
\newblock {\em Reitzel, Copenhagen}, 1889.

\bibitem{eom-bernstein}
B.~Tsirelson.
\newblock Bernstein inequality.
\newblock Encyclopedia of Mathematics,
  \url{http://encyclopediaofmath.org/index.php?title=Bernstein_inequality&oldid=15217},
  2012.
\newblock Adapted from an original article by A.V. Prokhorov, N.P. Korneichuk,
  V.P. Motornyi (originator), which appeared in Encyclopedia of Mathematics -
  ISBN 1402006098.

\bibitem{wolf77}
W.~Wolf.
\newblock Asymptotische {E}ntwicklungen f\"{u}r {W}ahrscheinlichkeiten grosser
  {A}bweichungen.
\newblock {\em Z. Wahrscheinlichkeitstheorie und Verw. Gebiete},
  40(3):239--256, 1977.

\bibitem{zolotarev65}
V.~M. Zolotarev.
\newblock On the closeness of the distributions of two sums of independent
  random variables.
\newblock {\em Teor. Verojatnost. i Primenen.}, 10:519--526, 1965.

\bibitem{zolotarev67}
V.~M. Zolotarev.
\newblock A sharpening of the inequality of {B}erry-{E}sseen.
\newblock {\em Z. Wahrscheinlichkeitstheorie und Verw. Gebiete}, 8:332--342,
  1967.

\bibitem{ZwiernikBook}
P.~Zwiernik.
\newblock {\em Semialgebraic statistics and latent tree models}, volume 146 of
  {\em Monographs on Statistics and Applied Probability}.
\newblock Chapman \& Hall/CRC, Boca Raton, FL, 2016.

\end{thebibliography}

\end{document}